\documentclass[a4paper,10pt]{article}
\usepackage[english,frenchb]{babel}
\usepackage[T1]{fontenc}
\usepackage[utf8]{inputenc}
\usepackage{amsmath}
\usepackage{array}
\usepackage{amsthm}
\usepackage{amssymb}
\input xy
\xyoption{all}
\usepackage{hyperref}

\newcommand{\A}{{\mathcal{A}}}
\newcommand{\B}{{\mathcal{B}}}
\newcommand{\bb}{{\rm B}}
\newcommand{\C}{{\mathcal{C}}}
\newcommand{\cc}{{\rm C}}
\newcommand{\D}{{\mathcal{D}}}
\newcommand{\ds}{{\mathbf{\Delta}}}

\newcommand{\G}{{\mathcal{G}}}

\newcommand{\s}{{\mathcal{S}}}
\newcommand{\scg}{{\mathbf{S}_c(\mathbf{gr})}}
\newcommand{\sct}{{\mathbf{S}_c(\mathcal{T})}}

\newcommand{\J}{{\mathcal{J}}}

\newcommand{\N}{{\mathcal{N}}}

\newcommand{\col}{{\rm colim}\,}
\newcommand{\T}{{\mathcal{T}}}
\newcommand{\kk}{{\Bbbk}}
\newcommand{\md}{\text{-}\mathbf{Mod}}
\newcommand{\mdd}{\mathbf{Mod}\text{-}}

\newcommand{\gr}{\mathbf{gr}}
\newcommand{\fct}{\mathbf{Fct}}

\title{D\'ecomposition de Hodge pour l'homologie stable des groupes d'automorphismes des groupes libres}
\author{Aur\'elien Djament\thanks{CNRS, laboratoire Paul Painlev\'e (UMR 8524), Cit\'e scientifique, b\^at. M2, 59655 Villeneuve d'Ascq Cedex, France ; djament@math.cnrs.fr.}}

\newtheorem{thi}{Th\'eor\`eme}

\newtheorem{thm}{Th\'eor\`eme}[section]
\newtheorem{pr}[thm]{Proposition}
\newtheorem{cor}[thm]{Corollaire}
\newtheorem{lm}[thm]{Lemme}

\newtheorem{conj}[thm]{Conjecture}

\theoremstyle{definition}
\newtheorem{defi}[thm]{D\'efinition}
\newtheorem{nota}[thm]{Notation}

\newtheorem{conv}[thm]{Convention}

\theoremstyle{remark}
\newtheorem{rem}[thm]{Remarque}

\begin{document}

\maketitle

\begin{abstract}
On établit une décomposition de l'homologie stable des groupes d'automorphismes des groupes libres à coefficients polynomiaux contravariants en termes d'homologie des foncteurs. Elle permet plusieurs calculs explicites, qui recoupent des résultats établis de manière indépendante par O. Randal-Williams et généralisent certains d'entre eux.

Nos méthodes reposent sur l'examen d'extensions de Kan dérivées associées à plusieurs catégories de groupes libres, la généralisation d'un critère d'annulation homologique à coefficients polynomiaux dû à Scorichenko, le théorème de Galatius identifiant l'homologie stable des groupes d'automorphismes des groupes libres à celle des groupes symétriques, la machinerie des $\Gamma$-espaces et le scindement de Snaith.

\smallskip

\begin{center}
\textbf{Abstract}
\end{center}

We establish a decomposition of stable homology  of automorphism groups of free groups with polynomial contravariant coefficients in terms of functor homology. This allows several explicit computations, intersecting results obtained by independent methods by O. Randal-Williams and extending some of them.

Our methods rely on the investigation of Kan extensions associated to several categories of free groups, the extension of a cancellation criterion for homology with polynomial coefficients due to Scorichenko, Galatius's Theorem identifying the stable homology of automorphism groups of free groups to that of symmetric groups, the machinery of $\Gamma$-spaces and the Snaith splitting.
\end{abstract}

\smallskip

\noindent {\em Classification MSC 2010 :} 18A25, 18G15, 20E36, 20J06, 55P42 (18A40, 18E30, 18G40, 55P47).

\smallskip

\noindent {\em Mots clefs :} groupes d'automorphismes des groupes libres, homologie des groupes, homologie des foncteurs, foncteurs polynomiaux, $\Gamma$-espaces et spectres, extensions de Kan dérivées.

\noindent {\em Keywords: } automorphism groups of free groups, group homology, functor homology, polynomial functors, $\Gamma$-spaces and spectra, derived Kan extensions.

\section*{Introduction}

Cet article est une contribution à l'étude de la (co)homologie des groupes d'automorphismes des groupes libres à coefficients tordus, dans le domaine stable (c'est-à-dire pour des groupes libres de rang assez grand par rapport au degré (co)homologique). À coefficients constants, S. Galatius \cite{Gal} a montré que l'homologie stable des groupes d'automorphismes des groupes libres s'identifie à celle des groupes symétriques, déterminée cinquante ans auparavant par M. Nakaoka \cite{Nak2}. Les profondes méthodes topologiques mises en \oe uvre par Galatius servent de point de départ à un travail récent d'O. Randal-Williams \cite{RW2}, qui explicite certains groupes de cohomologie stable de ces groupes à coefficients tordus. Nous retrouvons et généralisons  la plupart de ces résultats de Randal-Williams, par une méthode indépendante. Nous obtenons également une nouvelle construction de classes de cohomologie introduites par N. Kawazumi \cite{K-Magnus}.

\subsubsection*{Calculs (co)homologiques}

Voici deux exemples fondamentaux de calculs cohomologiques que nos résultats généraux impliquent.

Le premier (théorème~\ref{trw} dans le corps de l'article) est un théorème établi indépendamment par Randal-Williams \cite[{\em Corollary}~D]{RW2}. Dans son énoncé, $\Lambda^d$ (resp. $S^d$) désigne la $d$-ème puissance extérieure (resp. symétrique).

\begin{thi}[Randal-Williams]\label{torw}
 Soient $n$, $r$ et $d$ des entiers tels que $r\geq 2(n+d)+3$ et $G$ un groupe libre de rang $r$. Notons $H$ la représentation $G_{ab}\otimes\mathbb{Q}$ de ${\rm Aut}(G)$. Le $\mathbb{Q}$-espace vectoriel $H^n({\rm Aut}(G);\Lambda^d(H))$ est nul si $n\neq d$ ; pour $n=d$, sa dimension est le nombre de partitions de $d$. Le $\mathbb{Q}$-espace vectoriel $H^n({\rm Aut}(G);S^d(H))$ est nul sauf pour $n=d=0$ ou $1$, auquel cas il est de dimension $1$.
\end{thi}

Le second (théorème~\ref{cr-tens} dans le corps de l'article) améliore \cite[{\em Theorem}~A (ii)]{RW2}, valable seulement rationnellement.

\begin{thi}\label{cri-tens}
Soient $n$, $r$ et $d$ des entiers tels que $r\geq 2(n+d)+3$ et $G$ un groupe libre de rang $r$. Notons $H$ la représentation $G_{ab}$ de ${\rm Aut}(G)$. On dispose d'un isomorphisme
\[H^n({\rm Aut}(G);H^{\otimes d})\simeq H^{n-d}({\rm Aut}(G);\mathbb{Z})^{\oplus\mathrm{p}(d)}\]
où $\mathrm{p}(d)$ désigne le nombre de partitions d'un ensemble à $d$ éléments.
\end{thi}

\subsubsection*{Décomposition de Hodge}

Nous obtenons les calculs cohomologiques précédents à partir d'une décomposition naturelle de certains groupes de (co)homologie stable des groupes d'automorphismes des groupes libres à coefficients tordus. Cette décomposition repose, à partir du cadre formel développé avec C. Vespa dans \cite{DV}, sur des considérations d'algèbre homologique dans des catégories de foncteurs, en particulier l'étude de certaines extensions de Kan dérivées. L'argument final de la démonstration s'appuie sur le scindement de Snaith ; la dénomination de {\em décomposition de Hodge} provient de l'analogie avec la décomposition de Hodge pour l'homologie de Hochschild supérieure de Pirashvili\,\footnote{Voir la remarque~\ref{rq-pira} pour plus de précisions.} \cite{P-hodge}.

Le soubassement catégorique de cette décomposition est le suivant. Soit $\gr$ la catégorie des groupes libres de rang fini, munie de la structure monoïdale symétrique donnée par le produit libre, noté $*$. Comme observé dans \cite{DV15}, le cadre général naturel pour étudier stablement l'homologie des groupes d'automorphismes des groupes libres n'est pas de considérer les seuls foncteurs définis sur $\gr$ pour produire des familles compatibles de représentations de ces groupes, mais plutôt des foncteurs définis sur une catégorie auxiliaire de groupes libres notée $\G$ (dont la définition est rappelée au début de la section~\ref{sstr}). Dans le présent travail, nous utilisons également une autre catégorie de groupes libres notée $\scg$ (voir la notation~\ref{not-sv}). Les catégories $\G$ et $\scg$ ont les mêmes objets et les mêmes groupes d'automorphismes que la catégorie $\gr$, mais pas les mêmes morphismes ; elles possèdent également une structure monoïdale symétrique induite par le produit libre $*$. On dispose de foncteurs canoniques
\begin{equation}\label{eqd1}\xymatrix{& & \gr^{op}\\
\G\ar[r]^-\gamma & \scg\ar[ru]^-\beta\ar[rd]^-\alpha & \\
& & \gr
}\end{equation}
qui sont monoïdaux et égaux à l'identité sur les objets.

Notre première décomposition à la Hodge (théorème~\ref{thab1}) provient de la détermination du foncteur dérivé total de l'extension de Kan à gauche le long de $\G^{op}\xrightarrow{\gamma^{op}}\scg^{op}$ sur le foncteur constant (théorème~\ref{th-hs} et corollaire~\ref{crsn}). On ramène ainsi l'homologie stable des groupes d'automorphismes des groupes libres à coefficients dans un foncteur défini sur $\scg$ à de l'homologie de cette catégorie et de l'homologie de groupes symétriques à coefficients tordus. Toutefois, l'homologie de $\scg$ s'avère généralement hors d'atteinte directe.

Afin d'obtenir une forme plus exploitable, pour certains coefficients, de notre décomposition de Hodge, nous étudions également le comportement homologique du foncteur $\scg\xrightarrow{\beta}\gr^{op}$. Nous montrons au théorème~\ref{th-scoc} qu'il induit un isomorphisme entre groupes de torsion sous une hypothèse de {\em polynomialité} des coefficients. La notion de foncteur polynomial depuis les groupes libres vers une catégorie abélienne que nous utilisons est celle provenant des travaux d'Eilenberg-Mac Lane \cite{EML}, où l'on remplace la somme directe de modules, à la source, par le produit libre de groupes. Les puissances tensorielles ou symétriques de l'abélianisation constituent des archétypes de foncteurs polynomiaux dans ce contexte.

Nous obtenons ainsi au théorème~\ref{thab2} une décomposition de Hodge pour l'homologie stable des groupes d'automorphismes des groupes libres à coefficients dans un foncteur polynomial contravariant sur $\gr$ qui fait intervenir des groupes de torsion sur la catégorie $\gr$ via une suite spectrale d'hyperhomologie où apparaissent aussi des groupes symétriques. Rationnellement, le théorème~\ref{thab2} prend la forme simple du corollaire~\ref{corq}, qu'on peut énoncer comme suit, où $\mathfrak{a}_\mathbb{Q}$ désigne la tensorisation par $\mathbb{Q}$ du foncteur d'abélianisation et $\Lambda^s$ la $s$-ème puissance extérieure :

\begin{thi} Soient $i$, $m$, $d$ des entiers naturels tels que $m\geq 2(i+d)+3$ et $F$ un foncteur de $\gr^{op}$ vers les $\mathbb{Q}$-espaces vectoriels, polynomial de degré $d$. Il existe un isomorphisme naturel
\[H_i({\rm Aut}(\mathbb{Z}^{*m});F(\mathbb{Z}^{*m}))\simeq\bigoplus_{r+s=i}{\rm Tor}^\gr_r(F,\Lambda^s(\mathfrak{a}_\mathbb{Q})).\]
\end{thi}

Nos résultats possèdent des variantes \textbf{co}homologiques, comme le théorème~\ref{th-coh}, qui rationnellement peut s'exprimer ainsi :

\begin{thi} Soient $i$, $m$, $d$ des entiers naturels tels que $m\geq 2(i+d)+3$ et $X$ un foncteur de $\gr$ vers les $\mathbb{Q}$-espaces vectoriels, polynomial de degré $d$. Il existe un isomorphisme naturel
\[H^i({\rm Aut}(\mathbb{Z}^{*m});X(\mathbb{Z}^{*m}))\simeq\bigoplus_{r+s=i}{\rm Ext}^r_{\gr\md}(\Lambda^s(\mathfrak{a}_\mathbb{Q}),X).\]
\end{thi}

 Nos décompositions s'appliquent à des calculs explicites, comme les théorèmes~\ref{torw} et~\ref{cri-tens} ci-avant, grâce aux calculs cohomologiques sur $\gr$ de Vespa \cite{ves-cal}. On en tire également des conséquences qualitatives, comme des résultats généraux de finitude, dans la deuxième assertion du corollaire~\ref{cor-inv}, ou une description de structures multiplicatives en termes d'homologie des foncteurs, dans les propositions~\ref{pr-muh} et~\ref{pr-muc}.

\subsubsection*{Organisation de l'article}

Ce travail est constitué de trois parties. La première en donne tous les résultats principaux, en admettant les théorèmes~\ref{th-scoc} et~\ref{th-hs}, qui font l'objet des parties~\ref{pg2} et~\ref{pg3} respectivement.

Plus précisément, après quelques préliminaires classiques (section~\ref{sect-rappels}), la section~\ref{sstr} donne nos résultats théoriques sur l'homologie stable des groupes d'automorphismes des groupes libres à coefficients tordus, et les démontre, à partir des théorèmes susmentionnés. La section~\ref{sect-mult} en donne la forme \textbf{co}homologique et précise leurs compatibilités (en homologie comme en cohomologie) à des structures multiplicatives dont la construction est rappelée. La partie~\ref{pg1} se clôt sur des calculs explicites (section~\ref{scst}), qui incluent les théorèmes~\ref{torw} et~\ref{cri-tens} déjà énoncés, et le lien avec les classes de cohomologie de Kawazumi \cite{K-Magnus}.

La partie~\ref{pg2} est consacrée aux méthodes d'annulation et de comparaison homologiques à coefficients polynomiaux inaugurées par A. Scorichenko \cite{Sco}. La section~\ref{secsco} étend, dans la proposition~\ref{lm-sco}, le critère d'annulation originel de Scorichenko (valable pour des foncteurs sur une catégorie {\em additive}) à un contexte plus général qui inclut les foncteurs sur les groupes libres. Dans la section~\ref{ssco}, on montre comment en déduire le théorème~\ref{th-scoc}. Enfin, la section~\ref{s-conj} discute une conjecture~(\ref{conj-biv}) de comparaison homologique à coefficients polynomiaux plus générale qui permettrait de déduire de notre décomposition de Hodge <<~abstraite~>> (théorème~\ref{thab1}) un résultat (cf. théorème~\ref{th-biq}) rendant accessible au calcul la (co)homologie stable des groupes d'automorphismes des groupes libres à coefficients polynomiaux \textbf{bi}variants (c'est-à-dire pour des foncteurs définis sur $\gr^{op}\times\gr$).

Dans la partie~\ref{pg3}, on commence par associer à tout $A$-groupe libre (c'est-à-dire à un groupe muni d'une action d'un autre groupe $A$, et libre dans la catégorie des groupes munis d'une action de $A$) de type fini une catégorie auxiliaire dont on montre le caractère contractile (section~\ref{scc}). Les considérations mises en \oe uvre à cette fin, qui relèvent de la théorie combinatoire des groupes et de la théorie homotopique élémentaire des ensembles ordonnés, sont indépendantes de tout ce qui précède dans l'article, mais elles servent de préparation à la démonstration du théorème~\ref{th-hs}. En effet, l'étude du défaut de fidélité du foncteur $\gamma :\G\to\scg$ (qui est, par construction, plein et essentiellement surjectif) fait naturellement apparaître des structures de $A$-groupe libre (où $A$ est lui-même un groupe libre), comme on l'explique au paragraphe~\ref{skap}. Cela nous permet d'utiliser ces catégories auxiliaires, par l'intermédiaire de constructions de Grothendieck, pour déterminer le type d'homotopie de catégories appropriées et compléter la démonstration de nos résultats. Nous utilisons pour cela le théorème de Galatius \cite{Gal} déjà évoqué et la machinerie, classique en $K$-théorie algébrique, des $\Gamma$-espaces, qui donne une approche fonctorielle pour associer un spectre à une petite catégorie monoïdale symétrique (cf. Segal \cite{Seg}).

\subsubsection*{Liens avec d'autres travaux}

Comme nous l'avons déjà brièvement indiqué, le théorème~\ref{torw} et une version affaiblie du théorème~\ref{cri-tens} (tensorisée par le localisé de $\mathbb{Z}$ dans lequel les nombres premiers inférieurs ou égaux au degré de la puissance tensorielle considérée sont inversés) ont également été obtenus, juste après que les premières versions du présent travail eurent été prépubliées, par Randal-Williams \cite{RW2}, à l'aide de méthodes topologiques indépendantes des nôtres. Nos résultats et ceux de \cite{RW2} se recoupent largement, sans toutefois coïncider : Randal-Williams effectue également des calculs complémentaires de cohomologie des groupes d'automorphismes {\em extérieurs} des groupes libres à coefficients tordus, que nos méthodes semblent insuffisantes à établir (sauf démonstration de la conjecture~\ref{conj-biv}) ; en revanche, l'approche de \cite{RW2} ne semble pas suffisante pour établir le théorème~\ref{cri-tens} dans toute sa généralité. Signalons aussi que l'article \cite{RW2} s'inscrit dans le prolongement de la prépublication \cite{RW} (laquelle constitua une motivation importante à la réalisation du présent travail) qui montrait notamment que, la stabilité homologique étant admise, le théorème~\ref{torw} découlait de l'arrêt conjectural à la deuxième page de certaines suites spectrales de nature topologique. Cet arrêt est justement établi dans \cite{RW2}, qui fournit également plusieurs développements et améliorations de \cite{RW}.

\medskip

Appliqués au foncteur constant, nos résultats (le théorème~\ref{thab2} par exemple) indiquent l'isomorphisme entre l'homologie stable des groupes d'automorphismes des groupes libres et celle des groupes symétriques, profond résultat dû à Galatius \cite{Gal}. En fait, nous {\em utilisons} dès le début ce résultat, que les méthodes d'homologie des foncteurs semblent impuissantes à aborder : elles servent à ramener des calculs d'homologie stable à coefficients tordus à l'homologie stable des mêmes groupes à coefficients constants et à des groupes d'homologie des foncteurs --- cf. la méthode générale introduite dans \cite[section~1]{DV}.

Dans \cite{DV15}, cette méthode générale est déjà appliquée pour obtenir l'annulation de certains groupes d'homologie stable des groupes d'automorphismes des groupes libres à coefficients tordus. Des cas particuliers de cette annulation avaient été obtenus à l'aide de techniques topologico-géométriques par A. Hatcher et N. Wahl \cite{HW,HW-erra} et par Randal-Williams au début de \cite{RW} (repris dans \cite{RW2}), à partir de résultats de Hatcher, K. Vogtmann et Wahl \cite{HV04,HVW}. L'article \cite{DV15} traitait de foncteurs polynomiaux {\em co}variants (pour l'homologie stable des automorphismes des groupes libres ; en cohomologie il faut changer la variance), dont le comportement s'avère assez différent des coefficients {\em contra}variants examinés dans le présent travail. Cette différence, qui contraste avec ce qui advient pour l'homologie stable des groupes linéaires à coefficients tordus, par exemple, est déjà discutée dans \cite{RW} et \cite[remarque~7.1]{DV15}. Le point de départ (l'utilisation de la catégorie de groupes libres $\G$) de \cite{DV15} et du présent article coïncident toutefois. De fait, \cite{DV15} élucide le comportement homologique à coefficients polynomiaux du foncteur composé $\alpha\circ\gamma : \G\to\gr$ du diagramme~\eqref{eqd1}, de manière directe, c'est-à-dire sans transiter par l'intermédiaire de la catégorie $\scg$.

De plus, \cite{DV15} démontre déjà le théorème~\ref{thab2} en degré homologique $1$, pour les foncteurs se factorisant par l'abélianisation. Des cas particuliers ou variantes de ce résultat en degré (co)homologique $1$ avaient été obtenus auparavant par Kawazumi (voir la fin de la section~6 de \cite{K-Magnus}) et T. Satoh \cite{Sat1}. L'article \cite{Sat2} de Satoh établit également des cas particuliers en degré (co)homologique $2$ de nos résultats. Reposant sur une approche de théorie des groupes assez explicite, les résultats de Satoh fournissent aussi certains calculs de (co)homologie instables en bas degré.

Les méthodes d'homologie des foncteurs pourraient permettre (cf. section~\ref{s-conj}) d'obtenir des résultats généralisant à la fois ceux de \cite{DV15} et du présent article, à savoir le calcul (au moins rationnellement) de groupes de (co)homologie stable des groupes d'automorphimes des groupes libres à coefficients polynomiaux {\em bi}variants, c'est-à-dire de la forme $H_*({\rm Aut}(\mathbb{Z}^{*r});B(\mathbb{Z}^{*r},\mathbb{Z}^{*r}))$ (pour $r$ assez grand, ou l'analogue en cohomologie), où $B : \gr^{op}\times\gr\to\mathbf{Ab}$ est un foncteur polynomial.

Signalons enfin que le calcul de la cohomologie des groupes d'automorphismes des groupes libres à coefficients dans des représentations variées figure comme  $17^{\text{\`eme}}$ problème dans l'article de survol \cite{Mor} de Morita.

\paragraph*{Remerciements} 
L'auteur est reconnaissant envers les collègues suivants pour des discussions ou encouragements utiles à la réalisation de ce travail : Gregory Arone, Vincent Franjou, Camille Horbez, Nariya Kawazumi, Teimuraz Pirashvili, Lionel Schwartz et Christine Vespa.

Il remercie Oscar Randal-Williams de lui avoir communiqué des versions préliminaires de \cite{RW2} et pour les échanges y afférents, qui ont contribué au renforcement des premières versions du présent travail (qui ne donnaient la décomposition à la Hodge que rationnellement).

Il sait gré enfin au rapporteur anonyme dont la relecture attentive a permis de substantielles améliorations de la version d'origine.

\part{Homologie stable}\label{pg1}

\section{Notations et rappels généraux}\label{sect-rappels}

\paragraph*{Catégories ensemblistes}
On note $\mathbf{Ens}$ la catégorie des ensembles et $\mathbf{ens}$ la sous-catégorie pleine des ensembles finis, ou plutôt son squelette constitué des ensembles $\mathbf{n}:=\{1,\dots,n\}$ pour $n\in\mathbb{N}$. On désigne par $\Omega$ la sous-catégorie de $\mathbf{ens}$ ayant les mêmes objets et dont les morphismes sont les fonctions {\em surjectives}. On note également $\Gamma$ la catégorie ayant les mêmes objets que $\mathbf{ens}$ et dont les morphismes sont les {\em applications partiellement définies}. Autrement dit, un morphisme $\varphi : X\to Y$ est la donnée d'un sous-ensemble de $X$ noté ${\rm Def}(\varphi)$ et d'une fonction ${\rm Def}(\varphi)\to Y$. Cette catégorie est équivalente à la catégorie des ensembles finis {\em pointés}, les morphismes étant les applications (partout définies) préservant le point de base (une équivalence est donnée par l'adjonction d'un point de base externe à un objet de la catégorie $\Gamma$). On prendra garde que cette convention, classique, pour la catégorie $\Gamma$ (qui concorde par exemple avec les articles de Pirashvili comme \cite{P-hodge}) correspond à la catégorie {\em opposée} de la catégorie notée $\Gamma$ par Segal \cite{Seg}.

L'objet $\mathbf{0}=\varnothing$ est nul dans $\Gamma$ ; la somme ensembliste induit une somme catégorique sur $\Gamma$, notée $\sqcup$. 

On désigne par $\Pi$ la sous-catégorie de $\Gamma$ ayant les mêmes objets et dont les morphismes sont les applications surjectives partiellement définies.

 On note enfin $\mathbf{Ens}_\bullet$ la catégorie des ensembles pointés.

\paragraph*{Catégories de foncteurs}
Si $\C$ et $\D$ sont des catégories, avec $\C$ (essentiellement) petite, on note $\fct(\C,\D)$ la catégorie des foncteurs de $\C$ vers $\D$ (les morphismes étant les transformations naturelles).

Des références possibles pour le matériel classique ici rappelé sont \cite[{\em Appendix}~C]{Lli} ou \cite[§\,3]{FPs}.

Dans tout cet article, $\kk$ désigne un anneau commutatif (au-dessus duquel seront pris les produits tensoriels de base non spécifiée), qui sera le plus souvent un corps ou l'anneau des entiers. Si $\C$ est une petite catégorie, on note $\C\text{-}\,{_\kk\mathbf{Mod}}$, ou simplement $\C\md$ s'il n'y a pas d'ambiguïté sur $\kk$, pour $\fct(\C,{_\kk\mathbf{Mod}})$, où $_\kk\mathbf{Mod}$ est la catégorie des $\kk$-modules à gauche. Pour $\kk=\mathbb{Z}$, la catégorie $_\kk\mathbf{Mod}$ des groupes abéliens sera également notée $\mathbf{Ab}$. On note $_\kk\mdd\C$ ou $\mdd\C$ pour $\C^{op}\md$. La catégorie $\C\md$ est une catégorie abélienne de Grothendieck. Pour $c\in {\rm Ob}\,\C$, on pose $P^\C_c:=\kk[\C(c,-)]$ (pour alléger les notations, on ne mentionne pas l'anneau $\kk$, de même que dans plusieurs autres notations introduites ci-dessous), où $\C(c,d)$ désigne l'ensemble des morphismes de $c$ vers $d$ dans $\C$ et $\kk[-]$ le foncteur de $\kk$-linéarisation, c'est-à-dire l'adjoint à gauche au foncteur d'oubli $_\kk\mathbf{Mod}\to\mathbf{Ens}$. Le foncteur $P^\C_c$ représente l'évaluation en $c$ de $\C\md$ vers $_\kk\mathbf{Mod}$ (lemme de Yoneda), il est donc projectif de type fini, et les $P^\C_c$ engendrent $\C\md$ lorsque $c$ parcourt ${\rm Ob}\,\C$.

Si $F$ et $G$ sont des foncteurs de $\C\md$ et $\D\md$ respectivement, où $\C$ et $\D$ sont des petites catégories, on note $F\boxtimes G$ leur {\em produit tensoriel extérieur}, c'est-à-dire le foncteur de $(\C\times\D)\md$ défini par $(F\boxtimes G)(c,d):=F(c)\otimes G(d)$.

Le bifoncteur  $-\underset{\C}{\otimes}- : (\mdd\C)\times (\C\md)\to\,_\kk\mathbf{Mod}$ peut se définir comme la composée du produit tensoriel extérieur et du foncteur  {\em cofin} $(\C^{op}\times\C)\md\to\,_\kk\mathbf{Mod}$ (pour une présentation plus directe, voir par exemple \cite[§\,C.10]{Lli}). Le bifoncteur $\underset{\C}{\otimes}$ est {\em équilibré} (\cite[§\,2.7]{Weib}), en le dérivant à gauche par rapport à l'une ou l'autre des variables, on obtient un foncteur gradué ${\rm Tor}^\C_*(-,-)$. On dispose d'un isomorphisme canonique ${\rm Tor}^{\C^{op}}_*(G,F)\simeq {\rm Tor}^\C_*(F,G)$.

L'{\em homologie} de $\C$ à coefficients dans un foncteur $F$ de $\C\md$ est $H_*(\C;F):={\rm Tor}^\C_*(\kk,F)$. Noter que $H_0(\C;F)$ n'est autre que la colimite de $F$. 

Notons $\ds$ le squelette usuel de la catégorie des ensembles finis non vides totalement ordonnés et $\s:=\mathbf{Fct}(\ds^{op},\mathbf{Ens})$ la catégorie des ensembles simpliciaux. On désignera par $\N$ le foncteur {\em nerf} de la catégorie $\mathbf{Cat}$ des petites catégories vers $\s$ (cf. par exemple \cite[§\,B.12]{Lli}). Si $G$ est un groupe, vu comme catégorie à un objet, $\N(G)$ sera noté $\bb(G)$.

 On notera $C_*(-;\kk)$ (ou simplement $C_*(-)$ pour $\kk=\mathbb{Z}$) le foncteur de $\s$ vers la catégorie des complexes de chaînes de $\kk$-modules associant à un ensemble simplicial son complexe de Moore à coefficients dans $\kk$. Si $M$ est un $\kk$-module, qu'on peut voir comme foncteur constant dans $\C\md$, l'homologie $H_*(\C;M)$ est canoniquement isomorphe à l'homologie de l'ensemble simplicial $\N(\C)$ à coefficients dans $M$, c'est-à-dire à l'homologie du complexe $C_*(\N(\C);\kk)\otimes M$.

\smallskip

Il est classique que toutes les constructions précédentes possèdent des fonctorialités relativement à la petite catégorie de base. Par exemple, si $\varphi : \C\to\D$ est un foncteur entre petites catégories, le foncteur de précomposition par $\varphi$, noté $\varphi^* : \D\md\to\C\md$, induit un morphisme naturel $H_*(\C;\varphi^*F)\to H_*(\D;F)$ pour $F\in {\rm Ob}\,\D\md$, et plus généralement un morphisme naturel
\begin{equation}\label{eqft}
\mathrm{Tor}^\C_*(\varphi^* G,\varphi^*F)\to\mathrm{Tor}^\D_*(G,F)
\end{equation}
pour $G\in {\rm Ob}\,\mdd\D$.

\paragraph*{Catégories d'éléments, extensions de Kan et suites spectrales de Grothendieck}

Si $\C$ est une catégorie et $T : \C\to\mathbf{Ens}$ un foncteur, nous noterons $\C_T$ la {\em catégorie d'éléments} associée : ses objets sont les couples $(c,x)$ constitués d'un objet $c$ de $\C$ et d'un élément $x$ de $T(c)$, et les morphismes $(c,x)\to (d,y)$ sont les morphismes $f : c\to d$ de $\C$ tels que $T(f)(x)=y$. On notera $o_T : \C_T\to\C$ le foncteur d'oubli. Le foncteur de précomposition $o_T^* : \C\md\to\C_T\md$ possède un adjoint à gauche {\em exact} $\omega_T : \C_F\md\to\C\md$ donné sur les objets par
$$\omega_T(X)(c)=\underset{x\in T(c)}{\bigoplus}X(c,x).$$

Cette adjonction entre foncteurs exacts induit un isomorphisme naturel en homologie :
\begin{equation}\label{east}
{\rm Tor}^\C_*(Y,\omega_T(X))\simeq {\rm Tor}^{\C_T}_*(o_T^*Y,X)
\end{equation}
pour $Y\in {\rm Ob}\,\mdd\C$ et $X\in {\rm Ob}\,\C_T\md$ (par abus, on note encore $o_T$ pour $o_T^{op} : \C_T^{op}\to\C^{op}$). C'est le {\em lemme de Shapiro pour l'homologie des petites catégories} de \cite[§\,C.12]{Lli}.

Si $\xi : \C\to\D$ est un morphisme de $\mathbf{Cat}$, nous noterons $\xi_! : \mdd\C\to\mdd\D$ l'extension de Kan à gauche le long du foncteur $\xi^{op} : \C^{op}\to\D^{op}$. On dispose d'un isomorphisme
$$\xi_!(F)\underset{\D}{\otimes}G\simeq F\underset{\C}{\otimes}\xi^*G$$
naturel en les objets $F$ et $G$ de $\mdd\C$ et $\D\md$ respectivement. Cet isomorphisme d'adjonction peut se dériver comme suit, où $\mathbf{L}$ indique la dérivation à gauche des foncteurs :

\begin{pr}\label{pr-ssga} Soit $\xi : \C\to\D$ un foncteur entre petites catégories.
\begin{enumerate}
\item On dispose d'un quasi-isomorphisme naturel de complexes de chaînes de $\kk$-modules
\begin{equation}\label{ssggc}
 F\overset{\mathbf{L}}{\underset{\C}{\otimes}}\xi^*G\simeq\mathbf{L}\xi_!(F)\overset{\mathbf{L}}{\underset{\D}{\otimes}}G.
\end{equation}
\item Il existe un quasi-isomorphisme
\begin{equation}\label{dgekc}
 \mathbf{L}(\xi_!)(\kk)(d)\simeq C_*\big(\mathcal{N}(\C_{\xi^*\D(d,-)});\kk\big)
\end{equation}
naturel en l'objet $d$ de $\D^{op}$.
\item Il existe une suite spectrale
\begin{equation}\label{ssgg}
 {\rm E}^2_{i,j}={\rm Tor}^\D_i(\mathbf{L}_j(\xi_!)(F),G)\Rightarrow{\rm Tor}^\C_{i+j}(F,\xi^*G).
\end{equation}
naturelle en les objets $F$ et $G$ de $\mdd\C$ et $\D\md$ respectivement ;
\item les foncteurs dérivés à gauche de $\xi_!$ sont donnés par des isomorphismes
\begin{equation}\label{dgek}
 \mathbf{L}_\bullet(\xi_!)(F)(d)\simeq {\rm Tor}_\bullet^\C(F,\xi^*P^\D_d)\simeq H_\bullet(\C_{\xi^*\D(d,-)}^{op};o^*_{\xi^*\D(d,-)}F)
\end{equation}
de $\kk$-modules gradués naturels en $F$ et en l'objet $d$ de $\D^{op}$ ;
\item En particulier, $\mathbf{L}_\bullet(\xi_!)(\kk)(d)$ est naturellement isomorphe à l'homologie du complexe $C_*\big(\mathcal{N}(\C_{\xi^*\D(d,-)});\kk\big)$.
\end{enumerate}
\end{pr}

\begin{proof} Il est classique que l'adjonction (ou sa variation en termes de produit tensoriel) entre le foncteur {\em exact} $\xi^*$ et le foncteur $\xi_!$ se propage aux catégories dérivées, fournissant le quasi-isomorphisme naturel \eqref{ssggc} et, en passant à l'homologie, la suite spectrale \eqref{ssgg} (cf. par exemple \cite[§\,10.8]{Weib}). Les autres assertions s'en déduisent en appliquant ce qu'on vient d'observer au foncteur projectif $G=P^\D_d$ et en utilisant l'isomorphisme naturel \eqref{east}.
\end{proof}

\section{Résultats principaux}\label{sstr}

La démonstration de nos résultats repose sur la considération de différentes catégories de groupes libres et leur comparaison homologique. Commençons par donner des notations qui interviendront dans tout l'article. On note $\mathbf{Grp}$ la catégorie des groupes et $\mathbf{gr}$ la sous-catégorie pleine des groupes libres de rang fini, ou plutôt son squelette constitué des groupes $\mathbb{Z}^{*n}$, pour $n\in\mathbb{N}$, où $*$ désigne le produit libre, c'est-à-dire la somme catégorique de $\mathbf{Grp}$. On désigne par $\G$, comme dans \cite{DV15} (avec la terminologie de \cite{RWW}, c'est la {\em catégorie homogène} associée aux groupes d'automorphismes des groupes libres), la catégorie ayant les mêmes objets que $\mathbf{gr}$ et dont les morphismes $G\to H$ sont les couples $(u,T)$ constitués d'un monomorphisme de groupes $u : G\hookrightarrow H$ et d'un sous-groupe $T$ de $H$ tels que $H=u(G)*T$. Ici, $*$ désigne le produit libre {\em interne}, qui est une opération {\em partiellement définie} sur les familles de sous-groupes de $H$ ; l'égalité précédente signifie donc que le morphisme de groupes $G*T\to H$ dont la composante $G\to H$ est $u$ et la composante $T\to H$ est l'inclusion est un isomorphisme. La composition $G\xrightarrow{(u,T)} H\xrightarrow{(v,S)} K$ dans $\G$ est le morphisme $(v\circ u,S*v(T))$.

Si $F$ est un foncteur de $\G\md$, nous noterons
\[H^{st}_*(F):=\underset{n\in\mathbb{N}}{\col}H_*({\rm Aut}(\mathbb{Z}^{*n});F(\mathbb{Z}^{*n}))\]
l'homologie stable des groupes d'automorphismes des groupes libres à coefficients tordus par $F$. Les flèches qui définissent la colimite sont induites par les inclusions canoniques ${\rm Aut}(\mathbb{Z}^{*n})\to {\rm Aut}(\mathbb{Z}^{*(n+1)})$ (induites par le foncteur $-*\mathbb{Z}$) et, sur les coefficients, par les morphismes $\mathbb{Z}^{*n}\to\mathbb{Z}^{*(n+1)}$ donnés par l'inclusion $\mathbb{Z}^{*n}\hookrightarrow\mathbb{Z}^{*(n+1)}=\mathbb{Z}^{*n}*\mathbb{Z}$ munie du sous-groupe du but donné par le dernier facteur $\mathbb{Z}$.

La première étape de l'approche de l'homologie stable des groupes d'automorphismes des groupes libres par l'homologie des foncteurs est la suivante.

\begin{pr}[\cite{DV15}, proposition~4.4 et corollaire~4.5]\label{dv-gl}
 Soit $F$ un objet de $\G\md$. Il existe un isomorphisme naturel
 \[H^{st}_*(F)\xrightarrow{\simeq} H_*(\G\times G_\infty;\pi^*F)\]
 de $\kk$-modules gradués, où $G_\infty$ désigne le groupe $\underset{n\in\mathbb{N}}{\col}{\rm Aut}(\mathbb{Z}^{*n})$ (vu comme catégorie à un objet) et $\pi : \G\times G_\infty\to\G$ le foncteur de projection.

 En particulier, si $\kk$ est un corps, on a un isomorphisme naturel
\[H^{st}_*(F)\simeq H_*(\G;F)\otimes H_*(G_\infty;\kk)\simeq H_*(\G;F)\otimes H_*(\mathfrak{S}_\infty;\kk)\]
et donc $H^{st}_*(F)\simeq H_*(\G;F)$ si $\kk$ est de caractéristique nulle.
 \end{pr}
 
 Cet énoncé, présent dans \cite{DV15}, est un corollaire de résultats généraux de comparaison d'homologie stable des groupes à coefficients tordus et d'homologie des foncteurs apparaissant dans \cite{DV} et de l'identification de l'homologie de $G_\infty$ à celle de $\mathfrak{S}_\infty$ (déterminée par Nakaoka \cite{Nak2}), due à Galatius \cite{Gal}.
 
 Dans la suite, nous nous attacherons donc à étudier l'homologie de la catégorie $\G$ à coefficients dans des foncteurs appropriés. La comparaison homologique directe avec la catégorie de groupes libres usuelle $\mathbf{gr}$ (qui se prête à des calculs explicites), hors des cas traités dans \cite{DV15}, semblant hors de portée, nous utiliserons plusieurs catégories intermédiaires. 
 
  \begin{nota}\label{not-sv}
  \begin{enumerate}
  \item Soit $\C$ une catégorie. On note $\mathbf{S}(\C)$ la catégorie ayant les mêmes objets que $\C$ et telle que
  \[\mathbf{S}(\C)(c,d):=\{(f,g)\in\C(c,d)\times\C(d,c)\,|\,g\circ f={\rm Id}_c\},\]
  avec la composition $(f,g)\circ (f',g'):=(f\circ f',g'\circ g)$.
  \item Soit $(\T,*,0)$ une catégorie monoïdale symétrique dont l'unité $0$ est objet nul. On désigne par $\sct$ la sous-catégorie de $\mathbf{S}(\T)$ ayant les mêmes objets et dont les morphismes $(u,v) : G\to H$ sont ceux tels qu'existe un objet $T$ de $\T$ et un isomorphisme $H\simeq G*T$ faisant commuter le diagramme
 $$\xymatrix{G\ar[r]^u\ar[rd] & H\ar[r]^v\ar[d]^-\simeq & G\\
 & G*T\ar[ru] &
 }$$
 dont les flèches obliques sont l'inclusion et la projection canoniques.
  \end{enumerate}
 \end{nota}

 Remarquer que $\scg$ est l'image du foncteur canonique $\G\to\mathbf{S}(\gr)$ qui est l'identité sur les objets et associe à un  morphisme $(u,T) : G\to H$ le couple $(u,H=u(G)*T\twoheadrightarrow u(G)\xrightarrow{\simeq} G)$. On notera $\gamma : \G\to\scg$ le foncteur plein et essentiellement surjectif qu'il induit.
 
 On dispose de foncteurs canoniques $\mathbf{S}(\C)\to\C$ et $\mathbf{S}(\C)\to\C^{op}$, qui sont l'identité sur les objets et associent $f$ et $g$ respectivement à une flèche $(f,g)$ de $\mathbf{S}(\C)$.
 
 La restriction à $\scg$ du foncteur canonique $\mathbf{S}(\gr)\to\gr^{op}$ sera notée $\beta$.
 
La démonstration du théorème de Scorichenko \cite{Sco} donnée dans \cite[§\,5.2]{Dja-JKT} repose sur la comparaison de l'homologie à coefficients analytiques de $\mathbf{S}(\A)$ à celle de $\A$, lorsque $\A$ est une petite catégorie additive. Dans le cas des groupes d'automorphismes des groupes libres qui nous intéresse ici, une étape intermédiaire cruciale analogue interviendra : nous démontrerons (sous une forme plus générale) le théorème~\ref{th-scoc} ci-dessous dans la section~\ref{ssco}.

Avant de l'énoncer, on rappelle qu'un foncteur est dit {\em analytique} s'il est isomorphe à une colimite, qu'on peut supposer filtrante, de foncteurs polynomiaux. La notion classique de {\em foncteur polynomial}, qui remonte à Eilenberg-Mac Lane \cite[{\em Chapter}~II]{EML} pour des foncteurs entre catégories de modules, est exposée dans le cadre général d'une catégorie source monoïdale symétrique dont l'unité est objet nul dans \cite[§\,2]{HPV}, par exemple. Dans les cas qu'on considère, la structure monoïdale symétrique sera toujours donnée par le produit libre.
 
 \begin{thm}\label{th-scoc}
 Si $F$ est un foncteur analytique de  $\mdd\gr$ et $G$ un foncteur arbitraire de $\gr\md$, alors le morphisme
 \[{\rm Tor}_*^\scg(\beta^*G,\beta^*F)\to {\rm Tor}_*^\gr(F,G)\]
 qu'induit le foncteur $\beta : \scg\to\gr^{op}$
 est un isomorphisme.
\end{thm} 

Nous aurons besoin ensuite de comparer l'homologie des catégories $\G$ et $\scg$ (sans aucune hypothèse d'analycité). Le foncteur $\gamma : \G\to\scg$ est plein et essentiellement surjectif ; il est loin d'être fidèle, mais son défaut de fidélité peut être contrôlé. Cette étude sera menée dans la partie~\ref{pg3} ; avant de donner son résultat principal, nous avons besoin d'une notation\,\footnote{Le symbole $A$ utilisé dans la notation suivante (et qui sera repris dans les deux énoncés qui la suivent, puis surtout dans la partie~\ref{pg3}) pour désigner un groupe (non abélien) libre ne doit pas être confondu avec la notation $\mathrm{A}_\kk$ qui sera employée dans le théorème~\ref{thab1} et le corollaire~\ref{cor-qab}.} :

\begin{nota}\label{ncc}
On note $\cc : \scg^{op}\to\mathbf{Cat}$ le foncteur associant à un groupe libre de rang fini $A$ la catégorie d'éléments $\G_{\gamma^*\scg(A,-)}$ (les objets de $\G_{\gamma^*\scg(A,-)}$ sont donc les couples constitués d'un objet $G$ de $\G$ et d'un morphisme $A\to\gamma(G)$ de $\scg$), la fonctorialité venant de celle du foncteur Hom sur $\scg$.
\end{nota}

Nous relierons les objets de $\G_{\gamma^*\scg(A,-)}$ aux $A$-groupes libres, dans la section~\ref{skap}. Nous verrons également que $\cc(A)$ possède une structure monoïdale symétrique naturelle en $A$, induite par la somme amalgamée au-dessus de $A$. Par conséquent, l'ensemble simplicial $\N(\cc(A))$ possède une structure d'espace de lacets infinis naturelle en $A$. 

La partie~\ref{pg3} visera à établir :

\begin{thm}\label{th-hs}
Il existe une équivalence d'homotopie d'espaces de lacets infinis
\[\N(\cc(A))\simeq\Omega^\infty\Sigma^\infty\big(\bb\big(\beta^{op}(A)\big)\big)\]
naturelle en l'objet $A$ de $\scg^{op}$.
\end{thm}

En utilisant le scindement de Snaith \cite{Snaith}, on en tire le résultat suivant, dans lequel $\Sigma$ désigne la suspension algébrique, $\overset{\mathbf{L}}{\underset{\mathfrak{S}_n}{\otimes}}$ le produit tensoriel dérivé au-dessus du groupe symétrique $\mathfrak{S}_n$ et $\mathbb{Z}_\epsilon$ la représentation de signe de $\mathfrak{S}_n$. 

\begin{cor}\label{crsn}
Il existe un quasi-isomorphisme
$$C_*\big(\N(\cc(A));\mathbb{Z}\big)\simeq\underset{n\in\mathbb{N}}{\bigoplus}\Sigma^n\big(A_{ab}^{\otimes n}\overset{\mathbf{L}}{\underset{\mathfrak{S}_n}{\otimes}}\mathbb{Z}_\epsilon\big)$$
naturel en l'objet $A$ de $\scg^{op}$.
\end{cor}

\begin{proof}
Il est classique (cf. \cite{Snaith,Kahn,BEc} ou \cite[{\em Theorem}~8.2]{CMT}) que, pour un espace pointé connexe $X$, $\Omega^\infty\Sigma^\infty X$ a naturellement le même type d'homotopie stable que $\underset{n\in\mathbb{N}}{\bigvee}(X^{\wedge n})_{h\mathfrak{S}_n}$ (où l'indice $h\mathfrak{S}_n$ désigne les coïnvariants homotopiques sous l'action canonique de $\mathfrak{S}_n$) ; en particulier, le complexe $C_*(\Omega^\infty\Sigma^\infty X)$ est naturellement quasi-isomorphe à $\underset{n\in\mathbb{N}}{\bigoplus} \tilde{C}_*(X)^{\otimes n}\overset{\mathbf{L}}{\underset{\mathfrak{S}_n}{\otimes}}\mathbb{Z}$, où $\tilde{C}_*(X)$ désigne le complexe des chaînes {\em réduites} sur l'espace pointé $X$. Si $G$ est un groupe libre, alors $\tilde{C}_*(\bb G)$ est naturellement quasi-isomorphe à $\Sigma G_{ab}$, de sorte qu'on obtient un quasi-isomorphisme naturel
\[C_*(\Omega^\infty\Sigma^\infty(\bb G))\simeq\underset{n\in\mathbb{N}}{\bigoplus}\Sigma^n\big(G_{ab}^{\otimes n}\overset{\mathbf{L}}{\underset{\mathfrak{S}_n}{\otimes}}\mathbb{Z}_\epsilon\big)\;;\]
on conclut en prenant $G=\beta^{op}(A)$.
\end{proof}

Ce résultat permet d'obtenir le théorème suivant, dans lequel on note simplement $H_*^{st}(F)$ pour $H_*^{st}(\gamma^*F)$ --- d'autres abus de notation analogues seront effectués par la suite afin d'alléger les notations pour l'homologie stable.
\begin{thm}\label{thab1}
 Soit $F$ un foncteur de $\scg\md$. Le $\kk$-module gradué $H_*^{st}(F)$ est naturellement isomorphe à l'hyperhomologie du multicomplexe
 \[\kk\overset{\mathbf{L}}{\underset{\mathfrak{S}_\infty}{\otimes}}\underset{n\in\mathbb{N}}{\bigoplus}\Sigma^n\Big(\kk_\epsilon\overset{\mathbf{L}}{\underset{\mathfrak{S}_n}{\otimes}}\rm A_\kk^{\otimes n}\overset{\mathbf{L}}{\underset{\scg}{\otimes}}F\Big)\]
 où le groupe symétrique infini $\mathfrak{S}_\infty$ opère trivialement, $\kk_\epsilon$ est la représentation de signature sur $\kk$ du groupe symétrique $\mathfrak{S}_n$ et ${\rm A}_\kk$ désigne le foncteur de $\mdd\scg$ composé de $\beta^{op} : \scg^{op}\to\mathbf{gr}$, de l'abélianisation $\gr\to\mathbf{Ab}$ et de la tensorisation par $\kk$, $\mathbf{Ab}\to\,_\kk\mathbf{Mod}$.
\end{thm}

\begin{proof} Par la proposition~\ref{dv-gl}, $H_*^{st}(F)$ est naturellement isomorphe à l'hyperhomologie de
\[\kk\overset{\mathbf{L}}{\underset{\mathfrak{S}_\infty\times\G}{\otimes}}\pi^*\gamma^*(F)\simeq\kk\overset{\mathbf{L}}{\underset{\mathfrak{S}_\infty}{\otimes}}\Big(\kk\overset{\mathbf{L}}{\underset{\G}{\otimes}}\gamma^*(F)\Big)\]
(avec action triviale de $\mathfrak{S}_\infty$). La proposition~\ref{pr-ssga} montre par ailleurs que $\kk\overset{\mathbf{L}}{\underset{\G}{\otimes}}\gamma^*(F)$ est naturellement quasi-isomorphe à $\mathbf{L}\gamma_!(\kk)\overset{\mathbf{L}}{\underset{\scg}{\otimes}}F$, de sorte que la conclusion résulte du corollaire~\ref{crsn} et du quasi-isomorphisme naturel \eqref{dgekc} de la proposition~\ref{pr-ssga}.
\end{proof}

Cette première forme de notre {\em décomposition de Hodge} pour l'homologie stable des groupes d'automorphismes des groupes libres demeure théorique, en raison de l'inaccessibilité de la catégorie $\scg$ aux calculs homologiques directs. Nous pensons toutefois qu'elle devrait permettre d'aller au-delà du théorème~\ref{thab2} ci-après (voir notamment la section~\ref{s-conj}), par l'intermédiaire duquel nous exploiterons le théorème~\ref{thab1}.

Sur $\kk=\mathbb{Q}$, la semi-simplicité des représentations des groupes symétriques conduit à la simplification suivante du théorème~\ref{thab1} :

\begin{cor}\label{cor-qab} Soient $F$ un foncteur de $\scg\text{-}\,_\mathbb{Q}\mathbf{Mod}$ et $m\in\mathbb{N}$. Il existe un isomorphisme naturel
\[H_m^{st}(F)\simeq\bigoplus_{i+j=m}{\rm Tor}^\scg_i(\Lambda^j({\rm A}_\mathbb{Q}),F)\]
où $\Lambda^j$ désigne la $j$-ème puissance extérieure (sur les $\mathbb{Q}$-espaces vectoriels).
\end{cor}

\begin{rem}\label{rq-pira} Ce corollaire peut s'établir directement à partir du théorème~\ref{th-hs}, sans invoquer le scindement de Snaith, à l'aide d'un raisonnement très analogue à celui de Pirashvili \cite[{\em Theorem}~2.7]{P-hodge}, dans la catégorie $\Gamma\md$, qui donne une approche de la décomposition de Hodge pour l'homologie de Hochschild des $\mathbb{Q}$-algèbres commutatives par l'homologie des foncteurs et en permet des généralisations à l'homologie de Hochschild supérieure (voir \cite[{\em Corollary}~2.5]{P-hodge}).

En effet, l'espace pointé $\Omega^\infty\Sigma^\infty\bb(A)$ a naturellement le même type d'homotopie {\em rationnelle} que $\bb(A_{ab})$, dont l'homologie (rationnelle) est naturellement isomorphe à $\Lambda^*(A_{ab}\otimes\mathbb{Q})$. Cela permet de déduire le corollaire~\ref{cor-qab} du théorème~\ref{th-hs}, en montrant que l'on sait que la suite spectrale
\[E^2_{i,j}={\rm Tor}_i^\scg(\mathbf{L}_j(\gamma_!)(\mathbb{Q}),F)\Rightarrow H_{i+j}(\G;\gamma^*F)\]
(proposition~\ref{pr-ssga}, \eqref{ssgg}) s'effondre à la deuxième page et que les extensions associées à la filtration correspondante sont triviales. Cela s'obtient par un argument de formalité, issu de la théorie de l'obstruction pour les complexes de chaînes, due à Dold \cite{Dold} (voir aussi \cite[Proposition~1.6]{P-hodge}) : celle-ci montre qu'il suffit de vérifier l'annulation de certains groupes d'extensions entre puissances extérieures de l'abélianisation rationalisée. Cette annulation se déduit des calculs cohomologiques de Vespa \cite{ves-cal} et du corollaire~\ref{sco-ext} (qui constitue une variante en cohomologie du théorème~\ref{th-scoc}).
\end{rem}

Dans l'énoncé suivant, les actions des groupes symétriques sont les mêmes que dans le théorème~\ref{thab1}.

\begin{thm}\label{thab2}
 Soit $F$ un foncteur analytique de $\mdd\gr$. Il existe un isomorphisme naturel de $\kk$-modules gradués
 \[H_*^{st}(F)\simeq\mathbb{H}_*\left(\kk\overset{\mathbf{L}}{\underset{\mathfrak{S}_\infty}{\otimes}}\underset{n\in\mathbb{N}}{\bigoplus}\Sigma^n\Big(F\overset{\mathbf{L}}{\underset{\gr}{\otimes}}\mathfrak{a}_\kk^{\otimes n}\overset{\mathbf{L}}{\underset{\mathfrak{S}_n}{\otimes}}\kk_\epsilon\Big)\right)\]
 où $\mathfrak{a}_\kk$ désigne le foncteur d'abélianisation tensorisée par $\kk$ (dans $\gr\md$) et $\mathbb{H}_*$ l'hyperhomologie.
\end{thm}

\begin{proof} Le théorème~\ref{th-scoc} montre que les morphismes naturels
\[{\rm A}_\kk^{\otimes n}\overset{\mathbf{L}}{\underset{\scg}{\otimes}}\beta^*(F)\to F\overset{\mathbf{L}}{\underset{\gr}{\otimes}}\mathfrak{a}_\kk^{\otimes n}\]
 qu'induit $\beta$ sont des quasi-isomorphismes. La conclusion découle donc du théorème~\ref{thab1}.
\end{proof}

Là encore, le résultat se simplifie nettement rationnellement :
\begin{cor}\label{corq} Soient $F$ un foncteur analytique de $_\mathbb{Q}\mdd\gr$ et $m\in\mathbb{N}$. Il existe un isomorphisme naturel
\[H_m^{st}(F)\simeq\bigoplus_{i+j=m}{\rm Tor}^\gr_i(F,\Lambda^j(\mathfrak{a}_\mathbb{Q})).\]
\end{cor}

On peut également déduire du théorème~\ref{thab2} le résultat plus précis suivant, dans lequel on se place dans le cas universel $\kk=\mathbb{Z}$. Sa première assertion est due à Randal-Williams et Wahl.

\begin{cor}\label{cor-inv}
 Soient $d$, $i$, $m$ des entiers naturels tels que $m\geq 2(i+d)+3$ et $F : \gr^{op}\to\mathbf{Ab}$ un foncteur polynomial de degré $d$. 
 \begin{enumerate}
 \item (Randal-Williams et Wahl \cite{RWW}) Le morphisme canonique $H_i({\rm Aut}(\mathbb{Z}^{*m});F(\mathbb{Z}^{*m}))\to H_i^{st}(F)$
est un isomorphisme.
 \item Il existe un isomorphisme naturel de groupes abéliens gradués
 \[H_*^{st}(F)\simeq\mathbb{H}_*\left(\mathbb{Z}\overset{\mathbf{L}}{\underset{\mathfrak{S}_\infty}{\otimes}}\;\underset{0\leq n\leq d}{\bigoplus}\Sigma^n\Big(F\overset{\mathbf{L}}{\underset{\gr}{\otimes}}\mathfrak{a}^{\otimes n}\overset{\mathbf{L}}{\underset{\mathfrak{S}_n}{\otimes}}\mathbb{Z}_\epsilon\Big)\right).\]
  \item Supposons que les valeurs de $F$ sont des groupes abéliens de type fini. Alors le groupe abélien $H_i({\rm Aut}(\mathbb{Z}^{*m});F(\mathbb{Z}^{*m}))$ est de type fini.
  \item Supposons que $F$ est à valeurs dans les $\mathbb{Z}[\frac{1}{r!}]$-modules, où $r=\min(i,d)$. Il existe un isomorphisme naturel
\[H_i^{st}(F)\simeq\underset{t+n\leq d}{\underset{s+t+n=i}{\bigoplus}}H_s(\mathfrak{S}_\infty;{\rm Tor}^\gr_t(F,\mathfrak{a}^{\otimes n})^\epsilon_{\mathfrak{S}_n})\]
où $\mathfrak{S}_\infty$ opère trivialement et l'exposant $\epsilon$ indique que l'action canonique de $\mathfrak{S}_n$ est tordue par la signature.
  \item Supposons que $F$ est réduit (i.e. nul sur le groupe trivial) et à valeurs dans les $\mathbb{Z}[\frac{1}{s!}]$-modules, où $s=\max(2,i)$. Alors $H_i^{st}(F)$ est nul pour $i>d$ et isomorphe sinon à
\[\underset{n>0}{\underset{t+n=i}{\bigoplus}}{\rm Tor}^\gr_t(F,\mathfrak{a}^{\otimes n})^\epsilon_{\mathfrak{S}_n}.\]
 \end{enumerate}
\end{cor}

\begin{proof} La première assertion est le théorème de stabilité à coefficients tordus pour l'homologie des groupes d'automorphismes des groupes libres de Randal-Williams et Wahl \cite[{\em Theorem}~5.4]{RWW}.

Comme $F$ est un foncteur polynomial de degré $d$ de $\mdd\gr$, on a ${\rm Tor}^\gr_i(F,\mathfrak{a}^{\otimes n})=0$ pour $i>d-n$. Cette propriété se déduit de la résolution barre (déjà utilisée dans ce contexte dans \cite[§\,5\,A]{PJ}) ; elle apparaît explicitement dans \cite[Remarque~5.3]{DV15} pour $n=1$ et le cas général peut se trouver dans \cite[Proposition~4.1]{DPV}. Par conséquent, la deuxième assertion résulte du théorème~\ref{thab2}.

Cette même résolution projective explicite de $\mathfrak{a}^{\otimes n}$ issue de la résolution barre montre que ${\rm Tor}^\gr_i(F,\mathfrak{a}^{\otimes n})$ est un groupe abélien de type fini pour tous $n$ et $i$ si $F$ prend ses valeurs dans les groupes abéliens de type fini ; ce groupe est par ailleurs évidemment nul pour $n=0$ si $F$ est réduit.

On en déduit la troisième assertion, ainsi que la quatrième, en utilisant la formule de Künneth et le fait que l'homologie en degré non nul de $\mathfrak{S}_n$ est nulle à coefficients dans une représentation définie sur $\mathbb{Z}[1/n!]$.

Quant à la dernière assertion, elle se déduit de la précédente et du résultat de stabilité homologique suivant  : si $M$ est un groupe abélien, qu'on munit de l'action triviale des groupes symétriques, le morphisme canonique $H_i(\mathfrak{S}_n;M)\to H_i(\mathfrak{S}_\infty;M)$ est un isomorphisme pour $n>2i$ ; si $M$ est un $\mathbb{Z}[\frac{1}{2}]$-module c'est même un isomorphisme pour $n>i$. Cela provient par exemple des résultats de Nakaoka \cite{Nak}.
\end{proof}

\section{Variante en cohomologie et structures multiplicatives}\label{sect-mult}

Soit $X$ un foncteur de $\gr\md$. Nous noterons, de façon duale de $H_*^{st}$,
\[H_{st}^*(X):=\underset{n\in\mathbb{N}}{\lim}\,H^*({\rm Aut}(\mathbb{Z}^{*n});X(\mathbb{Z}^{*n}))\]
la cohomologie stable des groupes d'automorphismes des groupes libres à coefficients tordus par $X$, la limite étant relative aux monomorphismes de groupes ${\rm Aut}(\mathbb{Z}^{*n})\to {\rm Aut}(\mathbb{Z}^{*(n+1)})$ induits par $-*\mathbb{Z}$ et, sur les coefficients, par les épimorphismes de groupes $\mathbb{Z}^{*(n+1)}\to\mathbb{Z}^{*n}$ de projection sur les $n$ premiers facteurs. En général, cette cohomologie stable se comporte moins bien que son analogue homologique, en raison de la possible non-annulation de la $\underset{n\in\mathbb{N}}{\lim}^1$ correspondante. Cela n'advient toutefois pas si $X$ est {\em polynomial} de degré $d$. En effet, dans ce cas, la flèche canonique $H_{st}^i(X)\to H^i({\rm Aut}(\mathbb{Z}^{*n});X(\mathbb{Z}^{*n}))$ est un isomorphisme lorsque $n\geq 2(i+d)+3$. Cette variante en cohomologie du résultat précédent se déduit facilement des résultats de \cite{RWW} ; pour la commodité du lecteur, nous la démontrerons ci-dessous, après quelques rappels généraux.

Soit $J$ un cogénérateur injectif des $\kk$-modules. Nous noterons $V\mapsto V^\vee$ le foncteur de dualité ${\rm Hom}_\kk(-,J) :\ _\kk\mathbf{Mod}^{op}\to\ _\kk\mathbf{Mod}$, qui est exact et fidèle. Nous noterons encore de la même façon le foncteur exact et fidèle $\mdd\C\to\C\md$ (où $\C$ est une petite catégorie quelconque) qu'il induit entre catégories de foncteurs, par postcomposition.

\begin{lm}\label{lm-dte} Soient $\C$ une petite catégorie, $F$ et $T$ des foncteurs de $\mdd\C$ et $\C\md$ respectivement et $n\in\mathbb{N}$. Il existe un isomorphisme naturel
\[\big({\rm Tor}_n^\C(F,T)\big)^\vee\simeq {\rm Ext}^n_{\C\md}(T,F^\vee)\]
de $\kk$-modules.
\end{lm}

\begin{proof} Lorsque $T$ est un foncteur projectif du type $P^\C_c$, le résultat découle du lemme de Yoneda. Le cas général s'en déduit formellement par comparaison de foncteurs cohomologiques universels.
\end{proof}

Dans le cas de la (co)homologie des groupes, on a en particulier, la limite étant duale de la colimite :

\begin{lm}\label{dualhst} Si $F$ est un foncteur de $\mdd\gr$ et $n$ un entier naturel. On dispose d'un isomorphisme naturel de $\kk$-modules $H^n_{st}(F^\vee)\simeq \big(H_n^{st}(F)\big)^\vee$.
\end{lm}
  
Nous sommes maintenant en mesure de déduire la variante suivante en cohomologie de certains de nos résultats principaux (on laisse au lecteur le soin de traductions exhaustives de nos résultats homologiques en termes de cohomologie). Le premier point est dû à Randal-Williams et Wahl, comme indiqué plus haut.
  
\begin{thm}\label{th-coh} Soient $i$, $m$, $d$ des entiers naturels et $X$ un foncteur de $\gr\md$ polynomial de degré $d$.
\begin{enumerate}
\item Le morphisme canonique $H_{st}^i(X)\to H^i({\rm Aut}(\mathbb{Z}^{*m});X(\mathbb{Z}^{*m}))$ est un isomorphisme si $m\geq 2(i+d)+3$.
\item La cohomologie stable $H_{st}^*(X)$ est naturellement isomorphe à l'hypercohomologie
\[\mathbb{H}^*\left(\mathfrak{S}_\infty;\bigoplus_{n\leq d}\Sigma^n\Big(\mathbb{H}^*\big(\mathfrak{S}_n;\mathbf{R}{\rm Hom}_\gr(\mathfrak{a}_\kk^{\otimes n},X)_\epsilon\big)\Big)\right)\]
(où $\mathfrak{S}_\infty$ opère trivialement et $\mathfrak{S}_n$ par permutation des facteurs tordue par la signature).
\item Si $\kk=\mathbb{Q}$, il existe un isomorphisme naturel
\[H^i_{st}(X)\simeq\bigoplus_{r+s=i}{\rm Ext}^r_{\gr\md}(\Lambda^s(\mathfrak{a}_\mathbb{Q}),X).\]
\end{enumerate}
\end{thm}

\begin{proof} Si $X=F^\vee$ pour un foncteur $F$ de $\mdd\gr$ (qui est alors polynomial de degré $d$, comme $X$), le résultat découle directement des lemmes précédents et des corollaires~\ref{cor-inv} et~\ref{corq}.

Pour établir le premier résultat, on raisonne par récurrence sur le degré cohomologique en formant une suite exacte courte
\[0\to X\to F^\vee\to Y\to 0\]
avec $F=X^\vee$ (le premier morphisme étant l'unité de l'auto-adjonction du foncteur de dualité), dans laquelle $F$ et $Y$ sont polynomiaux de degré au plus $d$. Le résultat de stabilité homologique pour les groupes d'automorphismes des groupes libres à coefficients dans $F$  \cite[{\em Theorem}~5.4]{RWW} et l'hypothèse de récurrence appliquée à $Y$ permettent d'obtenir le pas de la récurrence pour $X$ par les suites exactes longues associées pour la cohomologie des groupes d'automorphismes des groupes libres.

Les assertions suivantes s'obtiennent de façon strictement analogue à l'aide des corollaires~\ref{cor-inv} et~\ref{corq} --- on note que $H^*_{st}$ définit un foncteur cohomologique sur les foncteurs polynomiaux de $\gr\md$, grâce au premier point. (Plutôt que de raisonner par récurrence sur le degré cohomologique, on peut, de façon équivalente, résoudre le foncteur $X$ par des foncteurs du type $F^\vee$ avec $F$ polynomial de degré $d$ dans $\mdd\gr$.)
\end{proof}

\paragraph*{Structures multiplicatives externes}
Soient $F$ et $G$ des foncteurs de $\mdd\gr$, $A$ et $B$ des groupes libres de rang fini. Les épimorphismes canoniques $A*B\to A$ et $A*B\to B$ induisent une application $\kk$-linéaire
\[F(A)\otimes G(B)\to F(A*B)\otimes G(A*B)\]
qui est équivariante pour l'action de ${\rm Aut}(A)\times {\rm Aut}(B)$ (donnée au but par restriction de l'action tautologique de ${\rm Aut}(A*B)$ le long du morphisme canonique ${\rm Aut}(A)\times {\rm Aut}(B)\to {\rm Aut}(A*B)$), d'où des morphismes naturels
\[H_*({\rm Aut}(A);F(A))\otimes H_*({\rm Aut}(B);G(B))\to H_*\big({\rm Aut}(A)\times {\rm Aut}(B);F(A)\otimes G(B)\big)\dots\]
\[\to H_*({\rm Aut}(A*B);F(A*B)\otimes G(A*B))\]
de $\kk$-modules gradués. Ceux-ci sont compatibles à la stabilisation en $A$ et $B$ et fournissent donc par passage à la colimite un produit externe en homologie stable
\[H_*^{st}(F)\otimes H_*^{st}(G)\to H_*^{st}(F\otimes G)\]
naturel et monoïdal symétrique en $F$ et $G$.

Par ailleurs, si $X$ et $Y$ sont des foncteurs de $\gr\md$ et que $A$ est un groupe libre de rang fini, le morphisme diagonal sur le groupe ${\rm Aut}(A)$ induit un morphisme
\[H^*({\rm Aut}(A);X(A))\otimes H^*({\rm Aut}(A);Y(A))\to H^*({\rm Aut}(A);X(A)\otimes Y(A))\]
qui fournit un produit externe en cohomologie stable
\[H_{st}^*(X)\otimes H_{st}^*(Y)\to H_{st}^*(X\otimes Y)\]
naturel et monoïdal symétrique en $X$ et $Y$.

Nous décrivons maintenant ces structures multiplicatives à travers nos identifications de $H^{st}_*(F)$ et $H^*_{st}(X)$ (lorsque $F$ et $X$ sont polynomiaux) à l'aide de (co)homologie de foncteurs. Nous nous limiterons à le faire lorsque $\kk=\mathbb{Q}$, hypothèse que l'on fait dans toute la suite de cette section. Le cas général peut se traiter de façon analogue mais est beaucoup plus lourd à écrire ; de plus, les applications que nous avons en vue concernent essentiellement cette situation rationnelle.

\begin{pr}\label{pr-muh} Soient $F$ et $G$ des foncteurs polynomiaux de $\mdd\gr$. À travers l'isomorphisme du corollaire~\ref{cor-qab}, le produit externe $H_i^{st}(F)\otimes H_j^{st}(G)\to H_{i+j}^{st}(F\otimes G)$ s'identifie au morphisme
\[\bigoplus_{n\in\mathbb{N}}{\rm Tor}^\gr_{i-n}(F,\Lambda^n(\mathfrak{a}_\mathbb{Q}))\otimes\bigoplus_{m\in\mathbb{N}}{\rm Tor}^\gr_{j-m}(G,\Lambda^m(\mathfrak{a}_\mathbb{Q}))\to\bigoplus_{l\in\mathbb{N}}{\rm Tor}^\gr_{i+j-l}(F\otimes G,\Lambda^l(\mathfrak{a}_\mathbb{Q}))\]
de composantes
\[{\rm Tor}^\gr_{i-n}(F,\Lambda^n(\mathfrak{a}_\mathbb{Q}))\otimes{\rm Tor}^\gr_{j-m}(G,\Lambda^m(\mathfrak{a}_\mathbb{Q}))\to{\rm Tor}^\gr_{i+j-(n+m)}(F\otimes G,\Lambda^{n+m}(\mathfrak{a}_\mathbb{Q}))\]
données par la composée de l'isomorphisme de Künneth
\[{\rm Tor}^\gr_*(F,\Lambda^n(\mathfrak{a}_\mathbb{Q}))\otimes{\rm Tor}^\gr_*(G,\Lambda^m(\mathfrak{a}_\mathbb{Q}))\xrightarrow{\simeq}{\rm Tor}^{\gr\times\gr}_*(F\boxtimes G,\Lambda^n(\mathfrak{a}_\mathbb{Q})\boxtimes\Lambda^m(\mathfrak{a}_\mathbb{Q})),\]
du morphisme
\[{\rm Tor}^{\gr\times\gr}_*(F\boxtimes G,\Lambda^n(\mathfrak{a}_\mathbb{Q})\boxtimes\Lambda^m(\mathfrak{a}_\mathbb{Q}))\to {\rm Tor}^{\gr\times\gr}_*((F\otimes G)\circ *,\Lambda^{n+m}(\mathfrak{a}_\mathbb{Q})\circ *)\]
induit par les inclusions canoniques $F\boxtimes G\to (F\otimes G)\circ *$ et $\Lambda^n(\mathfrak{a}_\mathbb{Q})\boxtimes\Lambda^m(\mathfrak{a}_\mathbb{Q})\to\Lambda^{n+m}(\mathfrak{a}_\mathbb{Q})\circ *$, et du morphisme
\[{\rm Tor}^{\gr\times\gr}_*((F\otimes G)\circ *,\Lambda^{n+m}(\mathfrak{a}_\mathbb{Q})\circ *)\to {\rm Tor}^\gr_*(F\otimes G,\Lambda^{n+m}(\mathfrak{a}_\mathbb{Q}))\]
induit par le foncteur $* : \gr\times\gr\to\gr$.
\end{pr}

\begin{pr}\label{pr-muc} Soient $X$ et $Y$ des foncteurs polynomiaux de $\gr\md$. À travers l'isomorphisme du théorème~\ref{th-coh}, le produit externe $H^i_{st}(X)\otimes H^j_{st}(Y)\to H^{i+j}_{st}(X\otimes Y)$ s'identifie au morphisme
\[\xymatrix{\underset{n\in\mathbb{N}}{\bigoplus}{\rm Ext}^{i-n}_{\gr\md}(\Lambda^n(\mathfrak{a}_\mathbb{Q}),X)\otimes\underset{m\in\mathbb{N}}{\bigoplus}{\rm Ext}^{j-m}_{\gr\md}(\Lambda^m(\mathfrak{a}_\mathbb{Q}),Y)\ar[d]\\
\underset{l\in\mathbb{N}}{\bigoplus}{\rm Ext}^{i+j-l}_{\gr\md}(\Lambda^l(\mathfrak{a}_\mathbb{Q}),X\otimes Y)
}\]
de composantes
\[{\rm Ext}^{i-n}_{\gr\md}(\Lambda^n(\mathfrak{a}_\mathbb{Q}),X)\otimes {\rm Ext}^{j-m}_{\gr\md}(\Lambda^m(\mathfrak{a}_\mathbb{Q}),Y)\to {\rm Ext}^{i+j-(n+m)}_{\gr\md}(\Lambda^{n+m}(\mathfrak{a}_\mathbb{Q}),X\otimes Y)\]
données par la composée du morphisme
\[{\rm Ext}^*_{\gr\md}(\Lambda^n(\mathfrak{a}_\mathbb{Q}),X)\otimes {\rm Ext}^*_{\gr\md}(\Lambda^m(\mathfrak{a}_\mathbb{Q}),Y)\to {\rm Ext}^*_{(\gr\times\gr)\md}(\Lambda^n(\mathfrak{a}_\mathbb{Q})\otimes\Lambda^m(\mathfrak{a}_\mathbb{Q}),X\otimes Y)\]
induit par le produit tensoriel et du morphisme 
\[{\rm Ext}^*_{(\gr\times\gr)\md}(\Lambda^n(\mathfrak{a}_\mathbb{Q})\otimes\Lambda^m(\mathfrak{a}_\mathbb{Q}),X\otimes Y)\to{\rm Ext}^*_{\gr\md}(\Lambda^{n+m}(\mathfrak{a}_\mathbb{Q}),X\otimes Y)\]
induit par le monomorphisme canonique $\Lambda^{n+m}(\mathfrak{a}_\mathbb{Q})\to\Lambda^n(\mathfrak{a}_\mathbb{Q})\otimes\Lambda^m(\mathfrak{a}_\mathbb{Q})$.
\end{pr}

\begin{proof}[Démonstration des propositions~\ref{pr-muh} et~\ref{pr-muc}] Tout d'abord, l'isomorphisme de la proposition~\ref{dv-gl} est monoïdal --- c'est un fait totalement général provenant du cadre formel de \cite{DV} (voir notamment le §\,4.1 de cette référence). Il en est de même pour l'isomorphisme du théorème~\ref{th-scoc} (ou de sa variante en cohomologie, le corollaire~\ref{sco-ext}, pour la proposition~\ref{pr-muc}), car le foncteur $\beta$ est monoïdal. Il suffit pour terminer de vérifier que le quasi-isomorphisme du corollaire~\ref{crsn} est monoïdal (où la structure monoïdale sur le foncteur $\cc$ provient de ce que le foncteur $\gamma$ est monoïdal), ce qu'établit la proposition~\ref{mpr}.
\end{proof}

\begin{rem} Puisqu'on travaille sur le {\em corps} $\mathbb{Q}$, on dispose également de {\em co}produits externes en homologie $H^{st}_*(F\otimes G)\to H^{st}_*(F)\otimes H^{st}_*(G)$ et en cohomologie $H_{st}^*(X\otimes Y)\to H_{st}^*(X)\otimes H_{st}^*(Y)$, qui se lisent aisément via nos isomorphismes, pour des coefficients polynomiaux (en utilisant aussi, pour le cas de la cohomologie, que la représentation triviale des groupes d'automorphismes des groupes libres et les foncteurs $\Lambda^i(\mathfrak{a}_\mathbb{Q})$ possèdent une résolution projective de type fini). Ceux-ci sont duaux des produits externes en cohomologie et en homologie respectivement. Pour des foncteurs à valeurs de dimension finie, cette dualité vaut au sens le plus immédiat du terme, par l'intermédiaire du lemme~\ref{dualhst}.
\end{rem}

La proposition~\ref{pr-muc} donne également une description concrète des injections canoniques ${\rm Ext}^i_{\gr\md}(\Lambda^j(\mathfrak{a}_\mathbb{Q}),F)\to H^{i+j}_{st}(F)$ déduites du théorème~\ref{th-coh}. Dans l'énoncé qui suit, on désigne par un point (voire l'absence de symbole) le produit externe sur $H^*_{st}$ ou le produit externe pour $\mathrm{Ext}^*_{\gr\md}$ (donné par $\mathrm{Ext}^*_{\gr\md}(A,B)\otimes \mathrm{Ext}^*_{\gr\md}(C,D)\to\mathrm{Ext}^*_{\gr\md}(A\otimes C,B\otimes D)$). On y utilise également une forme du produit de composition, ou produit de Yoneda, qui fournit des morphismes naturels $\mathrm{Ext}^i_{\gr\md}(F,G)\otimes H^j_{st}(F)\xrightarrow{\cup} H^{i+j}_{st}(G)$, qu'on peut définir comme la composée du morphisme de
${\rm Ext}^i_{\gr\md}(F,G)\otimes H^j({\rm Aut}(T);F(T))$ vers ${\rm Ext}^i_{\mathbb{Q}[{\rm Aut}(T)]}(F(T),G(T))\otimes H^j({\rm Aut}(T);F(T))$ induit par l'évaluation sur un groupe libre $T$ de rang assez grand et du produit de composition pour la cohomologie de ${\rm Aut}(T)$.

\begin{cor}\label{corec} Soient $i$ et $j$ des entiers naturels et $h$ un élément non nul de $H^1_{st}(\mathfrak{a}_\mathbb{Q})\simeq\mathbb{Q}$. Notons $\theta_j$ l'image de $h^j\in H^j_{st}(\mathfrak{a}_\mathbb{Q}^{\otimes j})$ dans $H^j_{st}(\Lambda^j(\mathfrak{a}_\mathbb{Q}))$ par l'application induite par le morphisme induit par $\mathfrak{a}_\mathbb{Q}^{\otimes j}\twoheadrightarrow\Lambda^j(\mathfrak{a}_\mathbb{Q})$. Alors l'application naturelle
 ${\rm Ext}^i_{\gr\md}(\Lambda^j(\mathfrak{a}_\mathbb{Q}),F)\to H^{i+j}_{st}(F)$ (où $F$ est polynomial dans $\gr\md$) du théorème~\ref{th-coh} est, à un scalaire inversible (indépendant de $F$) près, le produit de composition $-\cup\theta_j$.
\end{cor}

\begin{proof}Pour des raisons formelles de comparaison de foncteurs cohomologiques, il existe des classes $\theta'_j\in H^j_{st}(\Lambda^j(\mathfrak{a}_\mathbb{Q}))$ telles que l'application soit égale à $-\cup\theta'_j$. 

La proposition~\ref{pr-muc} se traduit par le fait que $(f\cup\theta'_j).(g\cup\theta'_k)$ (pour $f\in {\rm Ext}^i_{\gr\md}(\Lambda^j(\mathfrak{a}_\mathbb{Q}),F)$ et $g\in\mathrm{Ext}^l_{\gr\md}(\Lambda^k(\mathfrak{a}_\mathbb{Q}),G)$, avec $G$ polynomial), qui égale $(f.g)\cup (\theta'_j.\theta'_k)$ au signe près, coïncide avec $(f.g)\cup(\nu^{j,k}_*\theta'_{j+k})$, où $\nu^{j,k} : \Lambda^{j+k}(\mathfrak{a}_\mathbb{Q})\hookrightarrow\Lambda^j(\mathfrak{a}_\mathbb{Q})\otimes\Lambda^k(\mathfrak{a}_\mathbb{Q})$ est l'injection canonique, d'où $\nu^{j,k}_*\theta'_{j+k}=\theta'_j.\theta'_k$ au signe près. Si $m^{j,k} : \Lambda^j(\mathfrak{a}_\mathbb{Q})\otimes\Lambda^k(\mathfrak{a}_\mathbb{Q})\twoheadrightarrow\Lambda^{j+k}(\mathfrak{a}_\mathbb{Q})$ est l'application canonique, $m^{j,k}\nu^{j,k}$ est l'identité à un scalaire non nul près, donc on en déduit que $\theta'_{j+k}=m^{j,k}_*\theta'_j.\theta'_k$ à un scalaire non nul près. Comme $H^1_{st}(\mathfrak{a}_\mathbb{Q})$ est de dimension $1$, $\theta'_j$ et $\theta_j$ coïncident nécessairement à un scalaire non nul près pour $j=1$ ; la relation précédente permet d'en déduire le cas général.
\end{proof}

\paragraph*{Structures multiplicatives internes} Si $F$ est un foncteur de $\mdd\gr$ muni d'une structure d'algèbre $F\otimes F\xrightarrow{\mu} F$, on peut composer le produit externe $H_*^{st}(F)\otimes H_*^{st}(F)\to H_*^{st}(F\otimes F)$ avec le morphisme $H_*^{st}(F\otimes F)\to H_*^{st}(F)$ induit par $\mu$ pour obtenir une structure d'algèbre sur $H^{st}_*(F)$.

De même, si $X$ est un foncteur de $\gr\md$, toute structure d'algèbre $\mu : X\otimes X\to X$ sur $X$ induit une structure d'algèbre sur $H^*_{st}(X)$ en composant le produit externe $H^*_{st}(X)\otimes H^*_{st}(X)\to H^*_{st}(X\otimes X)$ avec $H^*_{st}(\mu)$. Celle-ci apparaîtra dans le théorème~\ref{coralg}. 

\section{Calculs de cohomologie stable}\label{scst}

\paragraph*{Calculs rationnels}

Nous retrouvons ici un résultat (théorème~\ref{torw} de l'introduction) établi indépendamment par Randal-Williams \cite[{\em Corollary}~D]{RW2}.

\begin{thm}[Randal-Williams]\label{trw}
 Soient $n$, $r$ et $d$ des entiers tels que $r\geq 2(n+d)+3$ et $G$ un groupe libre de rang $r$. Notons $H$ la représentation $G_{ab}\otimes\mathbb{Q}$ de ${\rm Aut}(G)$. Le $\mathbb{Q}$-espace vectoriel $H^n({\rm Aut}(G);\Lambda^d(H))$ est nul si $n\neq d$ ; pour $n=d$, sa dimension est le nombre de partitions de $d$. Le $\mathbb{Q}$-espace vectoriel $H^n({\rm Aut}(G);S^d(H))$ est nul sauf pour $n=d=0$ ou $1$, auquel cas il est de dimension $1$.
\end{thm}

\begin{proof}
Vespa a montré \cite[{\em Theorem}~4.2]{ves-cal} que le $\mathbb{Q}$-espace vectoriel ${\rm Ext}^r_{\gr\md}(\Lambda^s(\mathfrak{a}_\mathbb{Q}),\Lambda^d(\mathfrak{a}_\mathbb{Q}))$ a pour dimension le nombre de partitions de l'entier $d$ en $s$ parts si $r=d-s$ et $0$ sinon, tandis que ${\rm Ext}^r_{\gr\md}(\Lambda^s(\mathfrak{a}_\mathbb{Q}),S^d(\mathfrak{a}_\mathbb{Q}))$ est de dimension $1$ si $d\in\{0,1\}$ et $r=d-s$, et est nul dans tous les autres cas. On en déduit le résultat en appliquant le théorème~\ref{th-coh}.
\end{proof}

On peut préciser la structure multiplicative :

\begin{thm}\label{coralg}
  La cohomologie rationnelle stable des groupes d'automorphismes des groupes libres à coefficients dans le foncteur gradué algèbre extérieure est une algèbre de polynômes en une infinité dénombrable d'indéterminées, bigraduée par le degré cohomologique et le degré interne de l'algèbre extérieure.

 Plus précisément, si l'on note $\xi_n$, pour $n\in\mathbb{N}^*$, l'élément de bidegré $(n,n)$ de cette algèbre de cohomologie, image par l'isomorphisme du théorème~\ref{th-coh} d'un générateur du $\mathbb{Q}$-espace vectoriel de dimension~$1$ ${\rm Ext}^{n-1}_{\gr\md}(\mathfrak{a}_\mathbb{Q},\Lambda^n(\mathfrak{a}_\mathbb{Q}))$, alors cette algèbre est une algèbre symétrique sur les $\xi_n$.
\end{thm}

\begin{proof} On applique la proposition~\ref{pr-muc} et la description de la structure multiplicative trigraduée sur $({\rm Ext}^*_{\gr\md}(\Lambda^s(\mathfrak{a}_\mathbb{Q}),\Lambda^d(\mathfrak{a}_\mathbb{Q})))_{s,d}$ déduite de celle de $({\rm Ext}^*_{\gr\md}(\mathfrak{a}_\mathbb{Q}^{\otimes s},\mathfrak{a}_\mathbb{Q}^{\otimes d}))_{s,d}$, donnée par Vespa \cite[{\em Proposition}~3.1\,(2)]{ves-cal}, en prenant les coïnvariants sous l'action des groupes symétriques.
\end{proof}

Il est sans doute possible de déterminer la structure multiplicative donnée par le théorème~\ref{coralg} à l'aide des méthodes topologiques de Randal-Williams \cite{RW2}, mais le résultat ne figure pas explicitement dans son travail.

\begin{rem}\label{rqk}
 Dans \cite{K-Magnus}, Kawazumi introduit des classes de cohomologie $\bar{h}_p\in H^p_{st}(\Lambda^p(\mathfrak{a}_\mathbb{Q}))$ ; il montre dans \cite{K-mm} qu'elles sont algébriquement indépendantes dans l'algèbre bigraduée $H^*_{st}(\Lambda^\bullet(\mathfrak{a}_\mathbb{Q}))$ (c'est même vrai en restreignant la cohomologie des groupes d'automorphismes des groupes libres aux groupes de tresses). Cela fournit une deuxième démonstration du théorème~\ref{coralg} à partir du théorème~\ref{trw}, mais sans utiliser la proposition~\ref{pr-muc} ni la structure multiplicative sur ${\rm Ext}^*_{\gr\md}(\Lambda^*(\mathfrak{a}_\mathbb{Q}),\Lambda^*(\mathfrak{a}_\mathbb{Q}))$. De fait, le théorème~\ref{trw} implique que les classes algébriquement indépendantes $\bar{h}_p$ engendrent l'algèbre $H^*_{st}(\Lambda^\bullet_\mathbb{Q}\mathfrak{a})$ par un simple comptage de dimension en chaque bidegré.
\end{rem}

\begin{pr} Pour tout entier $n>0$, il existe $\lambda_n\in\mathbb{Q}^\times$ tel que $\bar{h}_n=\lambda_n.\xi_n$.
\end{pr}

\begin{proof} Pour $n=1$, cela résulte de ce que $\bar{h}_1$ et $\lambda_1$ sont des classes non nulles dans $H^1_{st}(\mathfrak{a}_\mathbb{Q})\simeq\mathbb{Q}$. Le cas général s'en déduit par un argument de produit comme suit. On dispose d'une application linéaire graduée
 \begin{equation}\label{eqa7}{\rm Ext}^*(\mathfrak{a}_\mathbb{Q},\mathfrak{a}_\mathbb{Q}^{\otimes 2})^{\otimes (n-1)}\to {\rm Ext}^*(\mathfrak{a}_\mathbb{Q},\mathfrak{a}_\mathbb{Q}^{\otimes n})
 \end{equation}
 composée de l'application
\[{\rm Ext}^*(\mathfrak{a}_\mathbb{Q},\mathfrak{a}_\mathbb{Q}^{\otimes 2})^{\otimes (n-1)}\to {\rm Ext}^*(\mathfrak{a}_\mathbb{Q}^{\otimes n-1},\mathfrak{a}_\mathbb{Q}^{\otimes n})\otimes {\rm Ext}^*(\mathfrak{a}_\mathbb{Q}^{\otimes n-2},\mathfrak{a}_\mathbb{Q}^{\otimes n-1})\otimes\dots\otimes {\rm Ext}^*(\mathfrak{a}_\mathbb{Q},\mathfrak{a}_\mathbb{Q}^{\otimes 2})\]
obtenue par produit tensoriel pour $i=n-1,n-2,\dots,1$ des applications
${\rm Ext}^*(\mathfrak{a}_\mathbb{Q},\mathfrak{a}_\mathbb{Q}^{\otimes 2})\to {\rm Ext}^*(\mathfrak{a}_\mathbb{Q}^{\otimes i},\mathfrak{a}_\mathbb{Q}^{\otimes i+1})$ induites par le foncteur exact $-\otimes\mathfrak{a}_\mathbb{Q}^{\otimes i-1}$ et du produit de composition. Notant $a$ un générateur du $\mathbb{Q}$-espace vectoriel ${\rm Ext}^1(\mathfrak{a}_\mathbb{Q},\mathfrak{a}_\mathbb{Q}^{\otimes 2})$ (qui est de dimension $1$), l'application précédente envoie $a^{\otimes (n-1)}$ sur un générateur de l'espace vectoriel ${\rm Ext}^{n-1}(\mathfrak{a}_\mathbb{Q},\mathfrak{a}_\mathbb{Q}^{\otimes n})$ (qui est aussi de dimension $1$). Cela résulte de \cite[proposition~3.1]{ves-cal}. La classe $\xi_n$ s'obtient donc (à un scalaire non nul près) par application du morphisme canonique ${\rm Ext}^*(\mathfrak{a}_\mathbb{Q},\mathfrak{a}_\mathbb{Q}^{\otimes n})\twoheadrightarrow{\rm Ext}^*(\mathfrak{a}_\mathbb{Q},\Lambda^n(\mathfrak{a}_\mathbb{Q}))\to H^*_{st}(\Lambda^n(\mathfrak{a}_\mathbb{Q}))$ décrit par le corollaire~\ref{corec} à  $a^{\otimes (n-1)}$, c'est-à-dire en évaluant cette classe sur un groupe libre de grand assez grand et en prenant le produit de composition par une classe non nulle de $H^1_{st}(\mathfrak{a}_\mathbb{Q})$. 

La conclusion résulte donc de résultats de Kawazumi, à savoir \cite[{\em Theorem}~4.1]{K-Magnus} et \cite[§\,2]{K-Magnus}, dont la description explicite de la classe de cohomologie qui y est notée $[\tau_1^\theta]$ montre qu'elle s'obtient (à un scalaire non nul près) par évaluation d'une classe non nulle de ${\rm Ext}^1(\mathfrak{a}_\mathbb{Q},\mathfrak{a}_\mathbb{Q}^{\otimes 2})\simeq\mathbb{Q}$.
\end{proof}

\paragraph*{Puissances tensorielles de l'abélianisation}

Notre deuxième application généralise sur les entiers un calcul cohomologique rationnel de Randal-Williams \cite[{\em Theorem}~A]{RW2}. On suppose donc $\kk=\mathbb{Z}$ dans ce paragraphe.

Commençons par une conséquence générale facile des résultats de la section~\ref{sect-mult}. Si $X$ est un foncteur de $\gr\md$, alors $H^*_{st}(X)$ est naturellement un module gradué sur l'anneau gradué $H^*_{st}(\mathbb{Z})\simeq H^*(\mathfrak{S}_\infty;\mathbb{Z})$. La deuxième assertion du théorème~\ref{th-coh} se simplifie sous une hypothèse de projectivité adéquate :

\begin{cor}\label{cor-cproj} Soit $X$ un foncteur polynomial de $\gr\md$ tel que, pour tout $n\in\mathbb{N}$, le $\mathbb{Z}[\mathfrak{S}_n]$-module ${\rm Ext}^*_{\gr\md}(\mathfrak{a}^{\otimes n},X)$ soit projectif. Alors le $H^*_{st}(\mathbb{Z})$-module gradué $H^*_{st}(X)$ est libre et l'on dispose d'un isomorphisme 
\[H^*_{st}(X)\simeq H^*_{st}(\mathbb{Z})\otimes\bigoplus_{n\in\mathbb{N}}\Sigma^n\big( {\rm Ext}^*_{\gr\md}(\mathfrak{a}^{\otimes n},X)^{\mathfrak{S}_n}_\epsilon\big)\]
de $H^*_{st}(\mathbb{Z})$-modules gradués.
\end{cor}

Dans \cite{ves-cal}, Vespa établit ({\em Theorem}~2.3) un isomorphisme 
\[{\rm Ext}^{d-n}_{\gr\md}(\mathfrak{a}^{\otimes n},\mathfrak{a}^{\otimes d})\simeq\mathbb{Z}[\Omega(\mathbf{d},\mathbf{n})]\]
et montre que ${\rm Ext}^i_{\gr\md}(\mathfrak{a}^{\otimes n},\mathfrak{a}^{\otimes d})$ est nul en degré $i\neq d-n$ ; elle identifie de plus ({\em Proposition}~2.5) l'action induite par le groupe symétrique $\mathfrak{S}_n$ sur le membre de gauche comme l'action tautologique sur le membre de droite (induite par la post-composition sur l'ensemble $\Omega(\mathbf{d},\mathbf{n})$), {\em à un signe} (explicite) {\em près}. Comme $\Omega(\mathbf{d},\mathbf{n})$ est un $\mathfrak{S}_n$-ensemble {\em libre}, il s'ensuit que ${\rm Ext}^{d-n}_{\gr\md}(\mathfrak{a}^{\otimes n},\mathfrak{a}^{\otimes d})$ est un $\mathbb{Z}[\mathfrak{S}_n]$-module  libre, de sorte que le corollaire précédent fournit le théorème~\ref{cri-tens} de l'introduction, sous la forme plus précise suivante :

\begin{thm}\label{cr-tens}
 Soit $d\in\mathbb{N}$. Le $H^*_{st}(\mathbb{Z})$-module gradué $H^*_{st}(\mathfrak{a}^{\otimes d})$ est libre
sur des générateurs de degré $d$ en nombre égal à celui des partitions d'un ensemble à $d$ éléments.
\end{thm}

Là encore, on peut préciser la structure multiplicative : en utilisant \cite[{\em Proposition}~3.1\,(2)]{ves-cal}, on obtient l'énoncé suivant, où l'on note $\mathrm{Part}(E)$ l'ensemble des partitions d'un ensemble $E$ :

\begin{thm}\label{multens} Soient $i$ et $j$ des entiers naturels. Le diagramme suivant commute
\[\xymatrix{H^*_{st}(\mathfrak{a}^{\otimes i})\otimes H^*_{st}(\mathfrak{a}^{\otimes j})\ar[r]^-\simeq\ar[d] & \ar[d] \big(H^*_{st}(\mathbb{Z})\otimes H^*_{st}(\mathbb{Z})\big)\otimes\big(\mathbb{Z}[\mathrm{Part}(\mathbf{i})]\otimes\mathbb{Z}[\mathrm{Part}(\mathbf{j})]\big)\\
H^*_{st}(\mathfrak{a}^{\otimes (i+j)})\ar[r]^-\simeq & H^*_{st}(\mathbb{Z})\otimes\mathbb{Z}[\mathrm{Part}(\mathbf{i+j})]
}\]
où les flèches horizontales sont les isomorphismes donnés par le théorème~\ref{cr-tens}, la flèche verticale de gauche est le produit externe et la flèche verticale de droite est induite par le produit de $H^*_{st}(\mathbb{Z})$ et la fonction $\mathrm{Part}(\mathbf{i})\times\mathrm{Part}(\mathbf{j})\to\mathrm{Part}(\mathbf{i+j})$ induite par la réunion disjointe.
\end{thm}

\paragraph*{Un autre calcul sur $\mathbb{Z}$} Nous terminons cette partie en montrant comment déterminer, à partir de la cohomologie de groupes symétriques, la cohomologie stable des groupes d'automorphismes des groupes libres à coefficients dans certains foncteurs qui apparaissent naturellement dans l'étude de la filtration polynomiale de $\gr\md$.

Soit $d\in\mathbb{N}$. Notons, comme dans \cite[page~217]{DPV}, $\beta_d$ le foncteur des modules sur $\mathbb{Z}[\mathfrak{S}_d]$ vers la sous-catégorie pleine de $\gr\md$ des foncteurs polynomiaux de degré au plus $d$ qui est adjoint à droite au foncteur $\mathrm{cr}_d$ associant à un foncteur $F$ l'effet croisé (au sens d'Eilenberg-Mac Lane \cite{EML}) $cr_d(F)(\mathbb{Z},\dots,\mathbb{Z})$. (Le lecteur prendra garde à ne pas confondre ce foncteur $\beta_d$, qui ne sera plus utilisé dans le reste de cet article, avec le foncteur $\beta$ introduit dans la section~\ref{sstr}.)

\begin{lm} Soient $i$, $d$ et $n$ des entiers naturels et $M$ un $\mathbb{Z}[\mathfrak{S}_d]$-module. Alors le groupe d'extensions ${\rm Ext}^i_{\gr\md}(\mathfrak{a}^{\otimes n},\beta_d(M))$ est nul sauf si $i=0$ et $d=n$, auquel cas il est naturellement isomorphe à $M$, de manière $\mathfrak{S}_d$-équivariante, où l'action sur ${\rm Hom}(\mathfrak{a}^{\otimes d},\beta_d(M))$ est celle qu'induit la permutation des facteurs du produit tensoriel $\mathfrak{a}^{\otimes d}$.
\end{lm}

\begin{proof}
L'annulation pour $n>d$ résulte de \cite[corollaire~3.3 et proposition~3.5]{DPV}. Pour $n\leq d$, par \cite[proposition~4.4]{DPV}, on dispose d'un isomorphisme naturel
\[{\rm Ext}^i_{\gr\md}(\mathfrak{a}^{\otimes n},\beta_d(M))\simeq{\rm Ext}^i_{\mathbb{Z}[\mathfrak{S}_d]}(\mathrm{cr_d}(\mathfrak{a}^{\otimes n}),M).\]
Si $n<d$, $\mathrm{cr_d}(\mathfrak{a}^{\otimes n})$ est nul puisque $\mathfrak{a}^{\otimes n}$ est alors de degré strictement inférieur à $d$. Si $n=d$, alors $\mathrm{cr_d}(\mathfrak{a}^{\otimes d})\simeq\mathbb{Z}[\mathfrak{S}_d]$ de manière $\mathfrak{S}_d$-équivariante, d'où le lemme.
\end{proof}

Ce lemme et le théorème~\ref{th-coh} entraînent aussitôt :

\begin{thm}\label{calcompl} Soient $d$ et $n$ des entiers naturels et $M$ un $\mathbb{Z}[\mathfrak{S}_d]$-module. Il existe un isomorphisme naturel en $M$
\[H_{st}^n(\beta_d(M))\simeq\underset{i+j=n-d}{\bigoplus}H^i\big(\mathfrak{S}_\infty;H^j(\mathfrak{S}_d;M_\epsilon)\big)\]
où $\mathfrak{S}_\infty$ opère trivialement et l'indice $\epsilon$ indique que l'action initiale de $\mathfrak{S}_d$ sur $M$ est tordue par la signature.
\end{thm}

\part{Homologie des foncteurs polynomiaux}\label{pg2}

\section{Un critère d'annulation à la Scorichenko}\label{secsco}

\begin{conv} Dans toute cette section, $\T$ désigne une petite catégorie possédant un objet nul et des coproduits finis, qu'on note $*$.
\end{conv}

L'outil crucial pour démontrer le théorème de comparaison homologique~\ref{th-scoc} est un critère d'annulation à la Scorichenko \cite{Sco}, analogue à celui utilisé dans \cite{Dja-JKT} pour établir un résultat de comparaison homologique similaire dans un contexte abélien. Pour cela, donnons quelques notations. Si $E$ est un ensemble fini et $I$ une partie de $E$, notons $t_E$ l'endofoncteur de $\T\md$ de précomposition par $a\mapsto a^{*E}$ et $u^I_E : t_E\to {\rm Id}$ la transformation naturelle donnée par la précomposition par le morphisme naturel $a^{*E}\to a$ dont la composante $a\to a$ correspondant au facteur étiqueté par $e\in E$ est l'identité si $e\in I$ et $0$ sinon. Si $(E,e)$ est un ensemble fini pointé, on définit l'effet croisé $cr_{(E,e)}$ comme la transformation naturelle
$$cr_{(E,e)}:=\sum_{e\in I\subset E}(-1)^{{\rm Card}(I)-1}u^I_E : t_E\to {\rm Id}.$$
Un foncteur $F$ de $\T\md$ est polynomial de degré au plus $d$ si et seulement si $cr_{(E,e)}(F)=0$, où $E$ est un ensemble de cardinal $d+2$.

\begin{defi}\begin{enumerate}
\item On note $\mathbf{E}_c(\T)$ la sous-catégorie de $\T$ image essentielle du foncteur canonique $\sct^{op}\to\T$. Autrement dit, $\mathbf{E}_c(\T)$ a les mêmes objets que $\T$ et un morphisme $v : H\to G$ de $\T$ appartient à $\mathbf{E}_c(\T)$ si et seulement s'il 
existe un objet $T$ de $\T$ et un isomorphisme $H\simeq G*T$ faisant commuter le diagramme
 $$\xymatrix{H\ar[r]^v\ar[d]_\simeq & G\\
 G*T\ar[ru] &
 }$$
 dont la flèche oblique est l'épimorphisme canonique déduit de ce que l'unité de la structure monoïdale $*$ en est un objet final.
\item Si $F$ et $G$ sont des foncteurs définis sur $\T$, on dira qu'une collection de morphismes $F(a)\to G(a)$ définis pour $a\in {\rm Ob}\,\T$ est {\em faiblement naturelle} si elle est naturelle relativement aux morphismes de $\mathbf{E}_c(\T)$.
\item Nous dirons que $\T$ vérifie l'hypothèse {\rm (H)} si, pour tout objet $a$ de $\T$, le morphisme de somme $a*a\to a$ appartient à $\mathbf{E}_c(\T)$.
\end{enumerate}
\end{defi}

Ainsi, $\mathbf{E}_c(\T)$ est constituée d'épimorphismes scindés de $\T$. L'indice $c$ de la notation sert à rappeler qu'on ne considère que les épimorphismes scindés possédant un {\em complément} pour le coproduit. Si $\T$ est additive et que ses idempotents se scindent, alors tout épimorphisme scindé de $\T$ appartient à $\mathbf{E}_c(\T)$. Mais dans la catégorie $\gr$, tout épimorphisme scindé n'appartient pas à $\mathbf{E}_c(\gr)$.

On rappelle qu'une structure de {\em cogroupe} sur un objet $a$ de $\T$ est la donnée d'un morphisme $c : a\to a*a$ coassociatif, dont la composée avec chacune des deux projections canoniques $a*a\to a$ est l'identité, pour lequel il existe une <<~coïnversion~>> $a\to a$. Il revient au même de se donner une structure de groupe sur chaque ensemble $\T(a,x)$ {\em naturellement en $x\in {\rm Ob}\,\T$}. Le lemme élémentaire suivant montre comment construire des <<~automorphismes triangulaires~>> à l'aide d'une structure de cogroupe.

\begin{lm} Si $c\in\T(a,a*a)$ est une structure de cogroupe sur un objet $a$ de $\T$, alors, pour tout objet $x$ de $\T$, la fonction $\varphi : \T(a,x)\to\T(a*x,a*x)$
\[f\mapsto\big(a*x\xrightarrow{c*x}a*a*x\xrightarrow{a*f*x}a*x*x\xrightarrow{a*\sigma}a*x\big),\]
où $\sigma : x*x\to x$ désigne la somme,
induit un morphisme de groupes $\T(a,x)\to {\rm Aut}_\T(a*x)$.
\end{lm}

\begin{proof} Il suffit de vérifier que $\varphi(f\square g)=\varphi(f)\circ\varphi(g)$, où $\square$ est le produit sur $\T(a,x)$ induit par la structure de cogroupe sur $a$ --- autrement dit,
\[f\square g=\big(a\xrightarrow{c}a*a\xrightarrow{f*g}x*x\xrightarrow{\sigma}x\big).\]

Cela provient de ce que le diagramme
\[\xymatrix{a*x\ar[r]^-{c*x}\ar[d]_-{c*x} & a*a*x\ar[r]^-{a*g*x}\ar[d]_-{c*a*x} & a*x*x\ar[r]^-{a*\sigma}\ar[d]_-{c*x*x} & a*x\ar[d]^-{c*x}\\
a*a*x\ar[r]^-{a*c*x}\ar[rrdd]_-{a*(f\square g)*x} & a*a*a*x\ar[r]^-{a*a*g*x}\ar[rd]_-{a*f*g*x} & a*a*x*x\ar[r]^-{a*a*\sigma}\ar[d]^-{a*f*x*x} & a*a*x\ar[d]^-{a*f*x}\\
&& a*x*x*x\ar[r]^-{a*x*\sigma}\ar[d]^-{a*\sigma*x} & a*x*x\ar[d]^-{a*\sigma}\\
&& a*x*x\ar[r]_-{a*\sigma} & a*x
}\]
commute (le rectangle en haut à gauche commute par coassociativité de $c$, celui en bas à droite par associativité de $\sigma$, et les autres parties commutent par naturalité de $*$ ou par définition de $f\square g$).
\end{proof}

\begin{pr}\label{prh} Si tout objet de $\T$ possède une structure de cogroupe, alors $\T$ vérifie l'hypothèse {\rm (H)}.
\end{pr}

\begin{proof} Si l'on se donne une structure de cogroupe sur un objet $a$ de $\T$, l'image $f$ de ${\rm Id}_a$ par le morphisme de groupes $\T(a,a)\to {\rm Aut}_\T(a*a)$ fourni par le lemme précédent fait commuter le diagramme
\[\xymatrix{a*a\ar[d]_-f^-\simeq\ar[r]^-\sigma & a\\
a*a\ar[ru]_{p_2} &
},\]
où $\sigma$ est la somme et $p_2$ la deuxième projection, ce qui montre que $\sigma$ appartient à $\mathbf{E}_c(\T)$, comme souhaité.
\end{proof}

Comme tout groupe {\em libre} possède une structure de cogroupe, on en déduit :

\begin{cor}\label{corgh} La catégorie $\gr$ des groupes libres de rang fini vérifie l'hypothèse {\rm (H)}.
\end{cor}

\begin{rem} En revanche, la petite catégorie avec objet nul et coproduits finis $\Gamma$ des ensembles finis avec applications partiellement définies ne vérifie pas l'hypothèse {\rm (H)}. En effet, les morphismes de $\mathbf{E}_c(\Gamma)$ sont ceux isomorphes à des flèches du type $E\sqcup E'\to E$ (égale à l'identité sur $E$ et non définie sur $E'$), ce qui n'est pas le cas de la codiagonale $E\sqcup E\to E$ (qui est partout définie) dès lors que l'ensemble fini $E$ est non vide.
\end{rem}

\begin{pr}\label{lm-sco}
On suppose que $\T$ vérifie l'hypothèse {\rm (H)}.

 Soient $T$ un foncteur de $\T\md$, $d\in\mathbb{N}$ et $(E,e)$ un ensemble pointé de cardinal $d+2$. Supposons que la transformation naturelle $cr_{(E,e)}(T) : t_E(T)\to T$ est un épimorphisme possédant un scindement faiblement naturel. Alors ${\rm Tor}^\T_*(F,T)=0$ pour tout foncteur $F$ de $\mdd\T$ polynomial de degré au plus $d$.
\end{pr}

\begin{proof}
Le foncteur de coproduit itéré $*^E : \T^E\to\T$ est adjoint à gauche au foncteur de diagonale itérée $\delta_E : \T\to\T^E$. Par conséquent, on dispose d'un isomorphisme naturel de $\kk$-modules gradués
\[{\rm Tor}^{\T}_*(F,X\circ\delta_E)\simeq {\rm Tor}^{\T^E}_*(F\circ *^E,X)\]
pour tous objets $F$ et $X$ de $\mdd\T$ et $\T^E\md$ respectivement. En particulier, comme $t_E$ est la précomposition par l'endofoncteur $*^E\circ\delta_E$ de $\T$, on a un isomorphisme naturel
\[{\rm Tor}^{\T}_*(F,t_E(Y))\simeq {\rm Tor}^{\T^E}_*(F\circ *^E,Y\circ *^E)\]
pour $F$ dans $\mdd\T$ et $Y$ dans $\T\md$. Il s'ensuit que, si $F$ est polynomial de degré au plus $d$, le morphisme naturel
${\rm Tor}^{\T}_*(F,t_E(Y))\to {\rm Tor}^{\T}_*(F,Y)$ induit par $cr_{(E,e)}(Y)$ est {\em nul} pour tout $Y$ dans $\T\md$. En effet, via l'isomorphisme précédent, le morphisme induit par $cr_{(E,e)}(Y)$ est la composée de l'endomorphisme de ${\rm Tor}^{\T^E}_*(F\circ *^E,Y\circ *^E)$ induit par l'endomorphisme $\sum_{e\in I\subset E}(-1)^{{\rm Card}(I)-1}\eta^I_E$ de $F\circ *^E$ (où $\eta^I_E$ est la précomposition par la transformation naturelle idempotente $*^E\to *^E$ dont la composante est l'identité sur les facteurs dans $I$ et $0$ sur les autres facteurs), endomorphisme qui est nul si $F$ est polynomial de degré au plus $d$, et du morphisme induit par le foncteur $*^E : \T^E\to\T$.

Notons $U$ le noyau du morphisme $cr_{(E,e)}(T)$, qui est un épimorphisme par hypothèse. Ce qui précède montre que la suite exacte longue pour ${\rm Tor}^\T_*(F,-)$ associée à la suite exacte courte $0\to U\to t_E(T)\xrightarrow{cr_{(E,e)}(T)} T\to 0$ se réduit à des suites exactes courtes
$$0\to {\rm Tor}^\T_i(F,T)\to {\rm Tor}^\T_{i-1}(F,U)\to {\rm Tor}^\T_{i-1}(F,t_E(T))\to 0,$$
ce qui établit déjà la nullité de ${\rm Tor}^\T_0(F,T)$. Pour démontrer le cas général par récurrence sur le degré homologique, il suffit de prouver que l'effet croisé $cr_{(E,e)}(U)$ est un épimorphisme faiblement scindé, ce qu'établit le lemme~\ref{aux-sco} ci-après.
\end{proof}

Le lemme suivant constitue une généralisation directe dans le cadre non additif qui est le nôtre de \cite[§\,2.1]{Dja-JKT}, qui reprend les résultats non publiés de Scorichenko \cite{Sco}.

\begin{lm}\label{aux-sco} Sous les hypothèses de la proposition précédente :
\begin{enumerate}
\item l'automorphisme $\xi$ du foncteur $t_E\circ t_E\simeq t_{E\times E}$ donné par l'échange des deux facteurs du produit cartésien $E\times E$ fait commuter le diagramme
 \[\xymatrix{t_E\circ t_E\ar[rr]^-{cr_{(E,e)}t_E}\ar[d]_\xi & & t_E \\
 t_E\circ t_E\ar[rru]_-{t_E cr_{(E,e)}} & & 
 }\;;\]
\item l'endofoncteur exact $t_E$ de $\T\md$ induit un endofoncteur exact de $\mathbf{E}_c(\T)\md$ : il existe un endofoncteur exact de $\tilde{t}_E$ de $\mathbf{E}_c(\T)\md$ faisant commuter le diagramme
\[\xymatrix{\T\md\ar[r]\ar[d]_{t_E} & \mathbf{E}_c(\T)\md\ar[d]^{\tilde{t}_E}\\
\T\md\ar[r] & \mathbf{E}_c(\T)\md
}\]
dans lequel les flèches horizontales sont la précomposition par l'inclusion $\mathbf{E}_c(\T)\to\T$. De plus, la transformation naturelle $cr_{(E,e)} : t_E\to\mathrm{Id}$ induit une transformation naturelle $\tilde{t}_E\to\mathrm{Id}$ (encore notée $cr_{(E,e)}$) ;
\item le noyau $U$ de l'épimorphisme $cr_{(E,e)}(T)$ est tel que $cr_{(E,e)}(U)$ est un épimorphisme possédant un scindement faiblement naturel.
\end{enumerate}
\end{lm}

\begin{proof} La première assertion est une conséquence formelle directe de ce que le coproduit $*$ définit une structure monoïdale symétrique sur $\T$.

Le fait que l'endofoncteur $a\mapsto a^{*E}$ de $\T$ induit un endofoncteur de $\mathbf{E}_c(\T)$ découle de ce que cette sous-catégorie de $\T$ est monoïdale, ce qui par précomposition fournit l'endofoncteur exact $\tilde{t}_E$ de $\mathbf{E}_c(\T)\md$. Le fait que les transformations naturelles $u^I_E : t_E\to {\rm Id}$ (pour $I$ partie non vide de $E$) induisent des transformations naturelles $\tilde{t}_E\to {\rm Id}$ provient de l'hypothèse (H), qui implique (en raisonnant par récurrence sur le cardinal de $I$) que les morphismes naturels $a^{*E}\to a$ de $\T$ dont les composantes $a\to a$ sont l'identité (pour au moins l'une d'entre elle) ou $0$ appartiennent à $\mathbf{E}_c(\T)$.

Pour établir le dernier point, considérons une section faiblement naturelle de $t_E(T)$ : c'est un morphisme $s : \tilde{T}\to\tilde{t}_E(\tilde{T})$ de $\mathbf{E}_c(\T)\md$ qui est une section de $cr_{(E,e)}(\tilde{T})$, où $\tilde{T}$ désigne la restriction de $T$ à $\mathbf{E}_c(\T)$. Le morphisme $s' : \tilde{t}_E(\tilde{T})\to\tilde{t}_E\big(\tilde{t}_E(\tilde{T})\big)=(\tilde{t}_E\circ\tilde{t}_E)(\tilde{T})$ composé de $\tilde{t}_E(s)$ et de l'automorphisme involutif $\tilde{\xi}$ induit par $\xi$ ($\mathbf{E}_c(\T)$ et $\T$ ont les mêmes automorphismes) constitue une section de $cr_{(E,e)}\big(\tilde{t}_E(\tilde{T})\big)$, en vertu du diagramme commutatif
\[\xymatrix{\tilde{t}_E(\tilde{T})\ar[rr]^-{\tilde{t}_E(s)}\ar[rrd]_-{s'} && \tilde{t}_E\big(\tilde{t}_E(\tilde{T})\big)\ar[rr]^-{\tilde{t}_E(cr_{(E,e)}(\tilde{T}))}\ar[d]_-{\tilde{\xi}} && \tilde{t}_E(\tilde{T})\\
&& \tilde{t}_E\big(\tilde{t}_E(\tilde{T})\big)\ar[rru]_-{cr_{(E,e)}\big(\tilde{t}_E(\tilde{T})\big)} &&
}\]
dont la partie droite est induite par le premier point du lemme.

De surcroît, le diagramme
\[\xymatrix{\tilde{t}_E(\tilde{T})\ar[d]_-{s'}\ar[rr]^-{cr_{(E,e)}(\tilde{T})} && \tilde{T}\ar[d]^-s\\
\tilde{t}_E\big(\tilde{t}_E(\tilde{T})\big)\ar[rr]_-{\tilde{t}_E(cr_{(E,e)}(\tilde{T}))} && \tilde{t}_E(\tilde{T})
}\]
commute, parce que c'est le cas de
\[\xymatrix{\tilde{t}_E(\tilde{T})\ar[rr]^-{cr_{(E,e)}(\tilde{T})}\ar[d]_-{\tilde{t}_E(s)} && \tilde{T}\ar[d]^-s\\
\tilde{t}_E\big(\tilde{t}_E(\tilde{T})\big)\ar[rr]^-{cr_{(E,e)}\big(\tilde{t}_E(\tilde{T})\big)}\ar[d]_-{\tilde{\xi}} && \tilde{t}_E(\tilde{T}) \\
\tilde{t}_E\big(\tilde{t}_E(\tilde{T})\big)\ar[rru]_-{\tilde{t}_E(cr_{(E,e)}(\tilde{T}))} &&
}\]
(la partie supérieure commute par naturalité de $cr_{(E,e)}$, la partie inférieure par la première partie du lemme). Cela entraîne que $s'$ induit une section $\tilde{U}\to\tilde{t}_E(\tilde{U})$ du morphisme $cr_{(E,e)}(\tilde{U})$ (où $\tilde{U}$ désigne la restriction de $U$ à $\mathbf{E}_c(\T)$, qui n'est autre que le noyau de $cr_{(E,e)}(\tilde{T})$) et achève la démonstration.
\end{proof}

\section{Comparaison homologique des catégories $\scg$ et $\gr^{op}$}\label{ssco}

Nous sommes maintenant en mesure d'établir le résultat suivant, dont le théorème~\ref{th-scoc} est un cas particulier, grâce au corollaire~\ref{corgh}. On y fait usage des catégories $\mathbf{S}(\T)$ et $\sct$ introduites dans la notation~\ref{not-sv} (lorsque $\T$ est une catégorie monoïdale symétrique dont l'unité est objet nul) ; la structure monoïdale sur $\T$ induit une structure monoïdale symétrique sur $\mathbf{S}(\T)$ et $\sct$ et les foncteurs canoniques reliant ces trois catégories (qui sont l'identité sur les objets) sont monoïdaux (au sens fort).

 \begin{thm}\label{thsco}
 Soit $\C$ une sous-catégorie monoïdale de $\mathbf{S}(\T)$ contenant $\mathbf{S}_c(\T)$, où $\T$ est une petite catégorie possédant un objet nul et des coproduits finis et vérifiant l'hypothèse {\rm (H)}. Notons $\beta$ la restriction à $\C$ du foncteur canonique $\mathbf{S}(\T)\to\T^{op}$. Si $F$ est un foncteur analytique de  $\mdd\T$ et $G$ un foncteur arbitraire de $\T\md$, alors le morphisme naturel
 \[{\rm Tor}_*^\C(\beta^*G,\beta^*F)\to {\rm Tor}_*^\T(F,G)\]
qu'induit le foncteur $\beta$ est un isomorphisme.
\end{thm}

\begin{proof} Posons $Y_i=\mathbf{L}_i(\beta_!)(\beta^*G)$ pour $i>0$ et notons $Y_0$ le noyau de la coünité $\beta_!\beta^*G\to G$. Comme $\beta$ est essentiellement surjectif, $\beta^*$ est fidèle, ce qui implique que cette coünité est surjective. Grâce à la proposition~\ref{pr-ssga}, il suffit de démontrer l'annulation ${\rm Tor}^\T_*(F,Y_i)=0$ pour tout $i\in\mathbb{N}$ et tout $F$ polynomial dans $\mdd\T$. Cette annulation se propage en effet au cas analytique par passage à la colimite. On va l'établir en utilisant la proposition~\ref{lm-sco} et le lemme ci-dessous.

Soit $(E,e)$ un ensemble fini pointé. Le morphisme
\[\xi_{a,a^{*E\setminus e}} : Y_i(a)\to Y_i(a*a^{*E\setminus e})=t_E(Y_i)(a)\]
fourni par le lemme~\ref{asco} est une section de l'effet croisé $cr_{(E,e)}(Y_i)(a)$ grâce à la première propriété, qui montre que la composée $u^I_E(Y_i)(a)\circ \xi_{a,a^{*E\setminus e}}$ est égale à l'identité pour $I=\{e\}$ et à $0$ lorsque $I\subset E$ contient strictement $\{e\}$, et à l'hypothèse (H), qui montre que la deuxième composante du morphisme $a^{*E}\simeq a*a^{*(E\setminus\{e\})}\to a$ qui induit $u^I_E$ appartient alors à $\mathbf{E}_c(\T)$. La seconde propriété du lemme montre que cette section possède la naturalité faible qui permet d'appliquer la proposition~\ref{lm-sco}, ce qui achève la démonstration.
\end{proof}

\begin{lm}\label{asco} Pour tout entier naturel $i$ et tous objets $a$ et $b$ de $\T$, il existe une application linéaire $\xi_{a,b} : Y_i(a)\to Y_i(a*b)$ de sorte que les propriétés suivantes soient vérifiées.
\begin{enumerate}
\item Soit $f : b\to a$ un morphisme de $\T$ et $\phi : a*b\to a$ le morphisme de composantes ${\rm Id}_a$ et $f$. Alors la composée
 \[c(\phi) : Y_i(a)\xrightarrow{\xi_{a,b}}Y_i(a*b)\xrightarrow{Y_i(\phi)}Y_i(a)\]
 est égale à l'identité si $f=0$ et à $0$ si $f$ appartient à la catégorie $\mathbf{E}_c(\T)$.
\item Le morphisme $\xi_{a,b}$ est naturel en $a\in {\rm Ob}\,\T$ et en $b\in {\rm Ob}\,\mathbf{E}_c(\T)$.
\end{enumerate}
\end{lm}

\begin{proof} La construction repose sur le caractère (fortement) monoïdal du foncteur $\beta$ (cf. \cite[§\,2.3]{Dja-JKT}). Celui-ci permet de considérer le foncteur $\tau_b$ de translation par $b$ (i.e. la précomposition par $-*b$, qu'on peut considérer à la fois dans $\C$ et $\T$), qui constitue un endofoncteur de $\C\md$, $\T\md$, $\mdd\C$ ou $\mdd\T$, et ce naturellement en $b$, et qui commute à isomorphisme naturel près à $\beta^*$, en raison du caractère monoïdal de $\beta$.

On définit, pour $i>0$,
\[\xi_{a,b} : Y_i(a)={\rm Tor}^\C_i(\beta^*G,\beta^* P^{\T^{op}}_a)\to Y_i(a*b)={\rm Tor}^\C_i(\beta^*G,\beta^* P^{\T^{op}}_{a*b})\]
(on fait usage de l'isomorphisme naturel \eqref{dgek} de la proposition~\ref{pr-ssga}) comme la composée suivante :
 \[{\rm Tor}^\C_i(\beta^*G,\beta^* P^{\T^{op}}_a)\to{\rm Tor}^\C_i(\tau_b\beta^*G,\tau_b\beta^* P^{\T^{op}}_{a*b})\to{\rm Tor}^\C_i(\beta^*G,\beta^* P^{\T^{op}}_{a*b})\]
 où la première flèche est induite par les morphismes obtenus en appliquant $\beta^*$ aux morphismes $u : G\to\tau_b G$ provenant du morphisme $b\to 0$ de $\T$ et $v : P^{\T^{op}}_a\to\tau_b P^{\T^{op}}_{a*b}$ induit par l'endofoncteur $-*b$ de $\T$, et la deuxième flèche est la flèche induite en homologie par le foncteur $-*b : \C\to\C$, comme rappelé en \eqref{eqft}. Pour $i=0$, on utilise la même construction, en constatant que ces flèches passent bien au noyau définissant $Y_0$, en raison du diagramme commutatif
 \[\xymatrix{\beta^*G\underset{\C}{\otimes}\beta^* P^{\T^{op}}_a\ar[r]\ar[d] & \tau_b\beta^*G\underset{\C}{\otimes}\tau_b\beta^* P^{\T^{op}}_{a*b}\ar[r]\ar[d] & \beta^*G\underset{\C}{\otimes}\beta^* P^{\T^{op}}_{a*b}\ar[d]\\
G(a)\simeq P^{\T^{op}}_a\underset{\T}{\otimes}G\ar[r] & \tau_b P^{\T^{op}}_{a*b}\underset{\T}{\otimes}\tau_b G\ar[r] & P^{\T^{op}}_{a*b}\underset{\T}{\otimes}G\simeq G(a*b)
 }\]
 dont les flèches horizontales sont définies comme précédemment et les flèches verticales sont induites par le foncteur $\beta : \C\to\T^{op}$.
 
 Pour démontrer la première propriété, on forme le diagramme commutatif
 $$\xymatrix{{\rm Tor}^\C_i(\beta^*G,\beta^* P^{\T^{op}}_a)\ar[d]_-{u_*}\ar[rd]^-{(u,v)_*} & & \\
 {\rm Tor}^\C_i(\tau_b\beta^*G,\beta^* P^{\T^{op}}_a)\ar[r]^-{v_*} & {\rm Tor}^\C_i(\tau_b\beta^*G,\tau_b\beta^* P^{\T^{op}}_{a*b})\ar[r]^{\phi_*}\ar[d] & {\rm Tor}^\C_i(\tau_b\beta^*G,\tau_b\beta^* P^{\T^{op}}_{a})\ar[d] \\
  & {\rm Tor}^\C_i(\beta^*G,\beta^* P^{\T^{op}}_{a*b})\ar[r]^{\phi_*} & {\rm Tor}^\C_i(\beta^*G,\beta^* P^{\T^{op}}_{a})
 }$$
 où les flèches verticales inférieures sont induites par le foncteur $-*b : \C\to\C$. La ligne du milieu est induite par la composée $\beta^* P^{\T^{op}}_a\xrightarrow{v_*}\tau_b\beta^* P^{\T^{op}}_{a*b}\xrightarrow{\phi_*}\tau_b\beta^* P^{\T^{op}}_{a}$.
 
 Pour $f=0$, celle-ci est induite par la transformation naturelle ${\rm Id}=\tau_0\to\tau_b$ (venant de ce que $0$ est objet initial de $\C$). Le diagramme commutatif
 $$\xymatrix{{\rm Tor}^\C_i(\tau_b\beta^*G,\beta^* P^{\T^{op}}_a)\ar[r]\ar[rd] & {\rm Tor}^\C_i(\tau_b\beta^*G,\tau_b\beta^* P^{\T^{op}}_a)\ar[d]\\
 & {\rm Tor}^\C_i(\beta^*G,\beta^* P^{\T^{op}}_a)
 }$$
 dont la flèche horizontale (resp. oblique, verticale) est induite par cette transformation naturelle ${\rm Id}\to\tau_b$ (resp. par la version contravariante de cette transformation naturelle, par le foncteur $-*b : \C\to\C$) et le fait que la composée $\beta^* G\xrightarrow{u_*}\tau_b\beta^*G\to\beta^* G$ égale l'identité montrent que $c(\phi)$ est, dans ce cas, l'identité, comme souhaité.
 
 Pour $f\in\mathbf{E}_c(\T)$, la composée $\beta^* P^{\T^{op}}_a\xrightarrow{v_*}\tau_b\beta^* P^{\T^{op}}_{a*b}\xrightarrow{\phi_*}\tau_b\beta^* P^{\T^{op}}_{a}$ se factorise par $\tau_b P^\C_a$. En effet, par hypothèse, $f$ est du type $b\simeq a*a'\twoheadrightarrow a$, notons $h$ le morphisme canonique $a\hookrightarrow a*a'\simeq b$. Si $g$ est un objet de $\C$, on dispose d'une fonction naturelle en $g$
 $$\T(\beta(g),a)\to\C(a,g*b)\qquad \psi\mapsto (a\xrightarrow{h}b\hookrightarrow g*b,g*b\xrightarrow{(\psi,f)}a)$$
 dont la composée avec la fonction $\C(a,g*b)\to\T(\beta(g)*b,a)$ induite par le foncteur $\beta : \C\to\T^{op}$ induit, après linéarisation, le morphisme $\beta^* P^{\T^{op}}_a\xrightarrow{v_*}\tau_b\beta^* P^{\T^{op}}_{a*b}\xrightarrow{\phi_*}\tau_b\beta^* P^{\T^{op}}_{a}$. La nullité de $c(\phi)$ se déduit de cette factorisation, en l'utilisant pour former un diagramme commutatif
 $$\xymatrix{{\rm Tor}^\C_i(\tau_b\beta^*G,\beta^* P^{\T^{op}}_a)\ar[r] & {\rm Tor}^\C_i(\tau_b\beta^*G,\tau_b P^\C_a)\ar[r]\ar[d] & {\rm Tor}^\C_i(\tau_b\beta^*G,\tau_b\beta^* P^{\T^{op}}_{a})\ar[d] \\
  & {\rm Tor}^\C_i(\beta^*G,P^\C_a)\ar[r] & {\rm Tor}^\C_i(\beta^*G,\beta^* P^{\T^{op}}_{a})
 }$$
 et en utilisant que $P^\C_a$ est projectif.
 
 Montrons maintenant la deuxième assertion. La naturalité en $a$ de $\xi_{a,b}$ découle de ce que $a\mapsto P^{\T^{op}}_a$ et $a\mapsto P^{\T^{op}}_{a*b}$ définissent des foncteurs sur $\T$. Celle en $b$ s'obtient comme suit. Soit $\chi : b\to b'$ un morphisme de $\mathbf{E}_c(\T)$. On dispose d'un diagramme commutatif
 $$\xymatrix{P^{\T^{op}}_a\ar[r]\ar[d] & \tau_b P^{\T^{op}}_{a*b}\ar[d] \\
 \tau_{b'} P^{\T^{op}}_{a*b'}\ar[r] & \tau_b P^{\T^{op}}_{a*b'}
 }$$
 dont les flèches $P^{\T^{op}}_a\to\tau_b P^{\T^{op}}_{a*b}$ et $P^{\T^{op}}_a\to\tau_{b'} P^{\T^{op}}_{a*b'}$ sont induites par les endofoncteurs respectifs $-*b$ et $-*b'$ de $\T$ et les autres flèches sont induites par $\chi : b\to b'$. Comme $\chi$ appartient à $\mathbf{E}_c(\T)$, qui est l'image essentielle du foncteur canonique $\mathbf{S}_c(\T)^{op}\to\T$, ce morphisme est l'image par $\beta$ d'un morphisme $\rho : b'\to b$ de $\C$.
 
 Cela permet de former un diagramme commutatif
 $$\xymatrix{& {\rm Tor}^\C_i(\beta^*G,\beta^* P^{\T^{op}}_a)\ar[ld]\ar[d]\ar[rd] & \\
 {\rm Tor}^\C_i(\tau_b\beta^*G,\beta^* P^{\T^{op}}_a)\ar[r]\ar[d] & {\rm Tor}^\C_i(\tau_b\beta^*G,\tau_{b'}\beta^* P^{\T^{op}}_{a*b'})\ar[r]_-{\rho^*}\ar[d]^-{\rho_*} & {\rm Tor}^\C_i(\tau_{b'}\beta^*G,\tau_{b'}\beta^* P^{\T^{op}}_{a*b'})\ar[d] \\
 {\rm Tor}^\C_i(\tau_b\beta^*G,\tau_b\beta^* P^{\T^{op}}_{a*b})\ar[r]\ar[d] & {\rm Tor}^\C_i(\tau_b\beta^*G,\tau_{b}\beta^* P^{\T^{op}}_{a*b'})\ar[r] & {\rm Tor}^\C_i(\beta^*G,\beta^* P^{\T^{op}}_{a*b'}) \\
{\rm Tor}^\C_i(\beta^*G,\beta^* P^{\T^{op}}_{a*b})\ar[rru]_-{\chi_*} & & 
 }$$
 dans lequel la composée ${\rm Tor}^\C_i(\beta^*G,\beta^* P^{\T^{op}}_a)\to {\rm Tor}^\C_i(\beta^*G,\beta^* P^{\T^{op}}_{a*b})$ est $\xi_{a,b}$ et  ${\rm Tor}^\C_i(\beta^*G,\beta^* P^{\T^{op}}_a)\to {\rm Tor}^\C_i(\tau_b\beta^*G,\tau_{b}\beta^* P^{\T^{op}}_{a*b'})$ est $\xi_{a,b'}$ ; le carré de droite est déduit de la transformation naturelle $-*\rho : -*b'\to -*b$ entre endofoncteurs de $\C$ et celui de gauche du carré commutatif ci-dessus, tandis que le triangle en haut à droite provient de ce que le morphisme canonique $G\to\tau_t G$ est fonctoriel en l'objet $t$ de $\T$. Cela montre la naturalité souhaitée.
\end{proof}

L'énoncé suivant constitue une variante du théorème~\ref{thsco}. Elle s'établit de la même façon, en utilisant la suite spectrale d'hypercohomologie (en Ext) induite par l'adjonction entre $\beta_!$ et $\beta^*$ plutôt que la suite spectrale d'hyperhomologie (en Tor) employée pour le théorème~\ref{thsco}.

\begin{cor}\label{sco-ext}
 Sous les mêmes hypothèses, si $X$ et $Y$ sont des foncteurs de $\T\md$, avec $Y$ polynomial, alors le morphisme naturel
 $${\rm Ext}^*_{\T\md}(X,Y)\to {\rm Ext}_{\mdd\C}^*(\beta^*X,\beta^*Y)$$
induit par $\beta^*$ est un isomorphisme.
\end{cor}

\section{Conjecture pour les coefficients bivariants}\label{s-conj}

Nous allons maintenant énoncer une conjecture de pure homologie des foncteurs qui permettrait de généraliser à la fois les résultats principaux du présent article et ceux de \cite{DV15}.

Si $\C$ est une catégorie, on note $\mathbf{F}(\C)$ la {\em catégorie des factorisations} de Quillen \cite{QK} associée à $\C$ : ses objets sont les flèches de $\C$, et un morphisme dans $\mathbf{F}(\C)$ de $a\xrightarrow{u}b$ vers $c\xrightarrow{v}d$ dans $\mathbf{F}(\C)$ est un couple $(f,g)$ constitué de flèches $f\in\C(a,c)$ et $g\in\C(d,b)$ telles que le diagramme
\[\xymatrix{a\ar[r]^u\ar[d]_f & b \\
c\ar[r]^v & d\ar[u]_g
}\]
de $\C$ commute ; la composition des morphismes dans $\mathbf{F}(\C)$ est induite par celle de $\C$. On dispose d'un foncteur {\em pleinement fidèle} $\mathbf{S}(\C)\to\mathbf{F}(\C)$ associant à un objet de $\C$ l'identité de celui-ci et à une flèche $(f,g)\in\mathbf{S}(\C)(x,y)$ (c'est-à-dire des flèches $f : x\to y$ et $g : y\to x$ de $\C$ telles que $gf={\rm Id}_x$) le morphisme $(f,g)$ de ${\rm Id}_x$ dans ${\rm Id}_y$. On a également des foncteurs canoniques $\mathbf{F}(\C)\to\C$ et $\mathbf{F}(\C)\to\C^{op}$ --- et donc un foncteur canonique $\mathbf{F}(\C)\to\C^{op}\times\C$ --- donnés sur les objets par la source et le but respectivement et associant à un morphisme sa première et deuxième composantes respectivement.

\begin{nota} On désigne par $\nu : \scg\to\mathbf{F}(\gr)$ la restriction à $\scg$ du foncteur canonique  $\mathbf{S}(\gr)\to\mathbf{F}(\gr)$ précédent.
\end{nota}

\begin{conj}\label{conj-biv} Soient $F : \gr^{op}\times\gr\to\mathbf{Ab}$ un foncteur analytique et $G : \gr\to\mathbf{Ab}$ un foncteur arbitraire. Notons $X$ (resp. $Y$) la composée de $F$ (resp. $G$) avec le foncteur canonique $\mathbf{F}(\gr)\to\gr^{op}\times\gr$ (resp. $\mathbf{F}(\gr)^{op}\to\gr$).

Alors l'application
\[{\rm Tor}^\scg_*(\nu^* Y,\nu^* X)\to {\rm Tor}^{\mathbf{F}(\gr)}_*(Y,X)\]
qu'induit le foncteur $\nu$ (cf. \eqref{eqft}) est un isomorphisme.
\end{conj}

Une catégorie du type $\mathbf{F}(\C)$ est beaucoup plus facile à étudier à partir de $\C$ du point de vue homologique que la catégorie $\mathbf{S}(\C)$ ou une variante comme $\scg$, car $\mathbf{F}(\C)$ s'identifie à  la catégorie opposée de la {\em catégorie d'éléments} associée au foncteur ${\rm Hom}_\C : \C^{op}\times\C\to\mathbf{Ens}$. Par exemple, le cas où le bifoncteur $F$ est {\em séparable} (c'est-à-dire produit tensoriel extérieur d'un foncteur de $\mdd\gr$ et d'un foncteur de $\gr\md$) se ramène à des groupes de torsion sur $\gr$ en vertu du lemme suivant.

\begin{lm} Soient $\C$ une petite catégorie, $A$ et $B$ des foncteurs de $\C\md$ et $C$ un foncteur de $\mdd\C$. On suppose que les valeurs de $A$ et de $C$ sont des groupes abéliens sans torsion. Notons $X$ (resp. $Y$) la composée de $C\boxtimes B$ (resp. $A$) avec le foncteur canonique $\mathbf{F}(\C)\to\C^{op}\times\C$ (resp. $\mathbf{F}(\C)^{op}\to\C$).

On dispose d'un isomorphisme naturel
\[{\rm Tor}^{\mathbf{F}(\C)}_*(Y,X)\simeq {\rm Tor}^\C_*(C,B\otimes A).\]
\end{lm}

\begin{proof} L'observation précédente sur la description de $\mathbf{F}(\C)$ comme catégorie d'éléments et l'isomorphisme général d'adjonction~\eqref{east} fournissent un isomorphisme naturel
\begin{equation}\label{eqax}
{\rm Tor}^{\mathbf{F}(\C)}_*(Y,X)\simeq {\rm Tor}^{\C^{op}\times\C}_*(B\boxtimes C,(\mathbb{Z}\boxtimes A)\otimes\mathbb{Z}[{\rm Hom}_\C]).
\end{equation}

Par ailleurs, si $t$ est un objet de $\C$ et $T$ un objet de $(\C^{op}\times\C)\md$, on dispose d'un isomorphisme naturel
\[{\rm Tor}^{\C^{op}\times\C}_*(\mathbb{Z}[\C(t,-)]\boxtimes C,T)\simeq {\rm Tor}^\C_*(C,T(t,-))\]
(qui peut se déduire également de l'isomorphisme~\eqref{east}, appliqué cette fois à la catégorie d'éléments $\C_{\C(t,-)}$). Cela implique en particulier
\[{\rm Tor}^{\C^{op}\times\C}_*(\mathbb{Z}[\C(t,-)]\boxtimes C,(\mathbb{Z}\boxtimes A)\otimes\mathbb{Z}[{\rm Hom}_\C])\simeq {\rm Tor}^\C_*(C,\mathbb{Z}[\C(t,-)]\otimes A),\]
ce qui, combiné à~\eqref{eqax}, est exactement l'isomorphisme souhaité pour $B=\mathbb{Z}[\C(t,-)]$. Le cas général s'en déduit par des arguments formels de comparaison de foncteurs homologiques.
\end{proof}

En combinant ce lemme au corollaire~\ref{cor-qab}, on obtient :

\begin{thm}\label{th-biq}
Supposons la conjecture~\ref{conj-biv} vérifiée. Soient $F$ et $G$ des foncteurs analytiques de $_\mathbb{Q}\mdd\gr$ et $\gr\text{-}\,{_\mathbb{Q}\mathbf{Mod}}$ respectivement.

Pour tout $n\in\mathbb{N}$, il existe un isomorphisme
\[H^{st}_n(F\otimes G)\simeq\underset{i+j=n}{\bigoplus}{\rm Tor}^\gr_i(F,G\otimes\Lambda^j(\mathfrak{a}_\mathbb{Q}))\]
 naturel en $F$ et $G$.
\end{thm}
En toute rigueur, le membre de gauche de l'isomorphisme devrait être noté $H_n^{st}(\beta^*(F)\otimes\alpha^*(G))$.

En utilisant ce résultat et les travaux de Vespa \cite{ves-cal}, de manière analogue à ce qu'on a fait dans la section~\ref{scst}, on voit que la conjecture~\ref{conj-biv} permettrait de déterminer entièrement la (co)homologie de ${\rm Aut}(T)$, pour un groupe libre de rang fini assez grand $T$, à coefficients dans ${\rm Hom}(T_{ab}^{\otimes i},T_{ab}^{\otimes j})\otimes\mathbb{Q}$, avec les actions des groupes symétriques $\mathfrak{S}_i$ et $\mathfrak{S}_j$. Cela permettrait en particulier de clarifier le rôle des classes de cohomologiee $h_p$ de Kawazumi \cite[§\,4]{K-Magnus}, dont les classes $\bar{h}_p$ brièvement discutées dans la remarque~\ref{rqk} ne constituent qu'une version <<~réduite~>>.

\part{Homotopie des petites catégories}\label{pg3}

L'objectif de cette partie est d'établir le théorème~\ref{th-hs}. Les méthodes, assez largement indépendantes de celles utilisées dans les parties précédentes, sont de nature plus topologique, avec une coloration $K$-théorique ; elles reposent sur l'étude du type d'homotopie de plusieurs petites catégories, domaine général inauguré par Quillen \cite{QK}. De fait, pour déterminer le type d'homotopie des catégories $\cc(A)$ introduites dans la notation~\ref{ncc}, nous devrons examiner celui de plusieurs catégories auxiliaires. Notre stratégie consiste à remplacer, à l'aide d'une construction de Grothendieck appropriée, les catégories $\cc(A)$ par d'autres catégories, plus maniables grâce à la machinerie de Segal \cite{Seg} associant un spectre à une catégorie monoïdale symétrique. Cela sera fait dans la section~\ref{sth}, dont le point de départ consiste à interpréter les objets de $\cc(A)$ comme des $A$-groupes libres. La section~\ref{scc}, auto-suffisante, s'attache à des constructions homotopico-catégoriques sur les $A$-groupes libres dont on se servira de façon cruciale à la fin de la section~\ref{sth}, où la catégorie $\mathfrak{D}_A(G)$ qu'on montre contractile dans la section~\ref{scc} permettra de mener à bien l'argument de construction de Grothendieck susmentionné.

\section{La catégorie contractile $\mathfrak{D}_A(G)$}\label{scc}

Toute la partie~\ref{pg3} reposant sur l'étude homotopique de catégories liées à des $A$-groupes libres ($A$ étant un groupe fixé, qui sera libre de rang fini dans les cas qui nous intéressent), nous commençons par rappeler cette notion classique.

\begin{nota}\label{not-agl}
 Soit $A$ un groupe. On note $\mathbf{Grp}_A$ la catégorie des groupes munis d'une action (à gauche) du groupe $A$, les morphismes étant les morphismes de groupes $A$-équivariants. Si $G$ est un objet de $\mathbf{Grp}_A$, $g$ un élément de $G$ et $a$ un élément de $A$, on notera $^a\! g$ l'action de $a$ sur $g$.
 
 On note $\mathcal{L}_A : \mathbf{Grp}\to\mathbf{Grp}_A$ l'adjoint à gauche au foncteur d'oubli, de sorte que $\mathcal{L}_A(G)$ est le groupe $G^{* A}$ avec l'action tautologique de $A$. On notera $\underset{a\in A}{\bigstar}^a\! G$ ce produit libre et $^a\! h$ l'image canonique d'un élément $h$ de $G$ dans le facteur correspondant à $a\in A$ (ce qui est cohérent avec la convention précédente). Un $A$-groupe (i.e. un objet de $\mathbf{Grp}_A$) est dit {\em $A$-libre} (ou simplement {\em libre}) s'il est isomorphe à $\mathcal{L}_A(G)$ pour un groupe {\em libre} $G$ ; il revient au même de demander qu'il appartienne à l'image essentielle de l'adjoint à gauche au foncteur d'oubli $\mathbf{Grp}_A\to\mathbf{Ens}$.
 
 On note $\mathbf{gr}_A$ la sous-catégorie pleine de $\mathbf{Grp}_A$ dont les objets sont les $A$-groupes libres\,\footnote{Dans la section~\ref{sth}, une catégorie, notée $\G_A$, dont les objets sont les $A$-groupes libres, mais les morphismes sont convenablement modifiés par rapport à cette définition usuelle, apparaîtra de façon importante.} de type fini $\mathcal{L}_A(\mathbb{Z}^{*n})$ (pour $n\in\mathbb{N}$).
\end{nota}

Au paragraphe~\ref{pnil}, nous établirons un énoncé sur les $A$-groupes libres (qui revient essentiellement à donner des générateurs explicites de leurs groupes d'automorphismes, analogues à ceux de Nielsen pour les automorphismes des groupes libres) qui nous servira à montrer (§\,\ref{pco}) qu'une certaine petite catégorie $\mathfrak{D}_A(G)$ (introduite dans la notation~\ref{ncdc}) associée à un $A$-groupe libre de type fini $G$ est contractile.

\subsection{Transformations de Nielsen des $A$-groupes libres}\label{pnil}

\begin{defi}\label{base-agl}
 Soient $A$ un groupe et $G$ un $A$-groupe. On note $\mathfrak{B}_A(G)$ l'ensemble des sous-groupes $H$ de $G$ tels que $G$ soit le produit libre (interne) des $^a\! H$ pour $a\in A$ (il revient au même de demander que le morphisme $\mathcal{L}_A(H)\to G$ de $\mathbf{Grp}_A$ adjoint à l'inclusion de groupes $H\hookrightarrow G$ soit un isomorphisme).
 
 On appelle {\em $A$-base} (ou simplement {\em base}) de $G$ toute partie génératrice libre d'un sous-groupe de $G$ appartenant à $\mathfrak{B}_A(G)$. On note $\mathcal{B}_A(G)$ l'ensemble de ces bases, qui est donc non vide si et seulement si $G$ est $A$-libre.
 
 Un élément de $G$ est dit {\em unimodulaire} s'il appartient à une $A$-base de $G$.
\end{defi}

(Cette dernière terminologie est inspirée de l'algèbre linéaire.)

On introduit maintenant une généralisation aux $A$-groupes des opérations de réduction de Nielsen pour les groupes, que ce dernier utilisa dans \cite{Ni18} pour montrer que le groupe des automorphismes d'un groupe libre de rang fini est engendré par les automorphimes élémentaires usuels (une bonne référence plus récente sur le sujet est \cite[chap.~I, §\,2]{LS}). Dans le cas où $A$ est le groupe trivial, tous les résultats de ce paragraphe sont très classiques ; il se peut qu'ils soient bien connus, dans toute leur généralité, des experts de la théorie combinatoire des groupes libres, mais l'auteur n'a trouvé aucune référence à ce sujet. 

\begin{defi}
 Une {\em $A$-transformation de Nielsen élémentaire} d'une partie $E$ d'un $A$-groupe $G$ est une opération consistant soit à transformer l'un des éléments de $E$ en son inverse (et laisser les autres éléments de $E$ inchangés), soit à transformer l'un des éléments $h$ de $E$ en $^a\!h$, pour un $a\in A$, soit à transformer l'un des éléments de $E$ en son produit (à gauche ou à droite) par un autre élément de $E$. 
 
 Une {\em $A$-transformation de Nielsen} est une suite finie de $A$-transformations de Nielsen élémentaires.
 
 Deux parties de $G$ sont dites {\em équivalentes à la Nielsen} s'il existe une $A$-transformation de Nielsen envoyant l'une sur l'autre.
\end{defi}

L'objectif de ce paragraphe consiste à établir le résultat suivant, dont un corollaire immédiat (qui possède un intérêt intrinsèque mais ne sera pas utile dans le présent travail) est la description de générateurs élémentaires explicites du groupe ${\rm Aut}_{\mathbf{Grp}_A}(G)$, lorsque $G$ est un $A$-groupe libre ($A$ étant un groupe quelconque) de type fini (i.e. possédant une $A$-base finie).

\begin{thm}\label{th-niel}
 Soient $A$ un groupe et $G$ un $A$-groupe libre de type fini. Toutes les $A$-bases de $G$ sont équivalentes à la Nielsen.
\end{thm}

Pour le voir, on adapte la méthode suivie par Hoare \cite{Hoa}, elle-même inspirée d'un travail beaucoup plus ancien de Petresco \cite{Petresco} et reposant sur le formalisme des {\em fonctions longueurs} introduit par Lyndon \cite{Lyn}.

Fixons-nous un élément $\mathbf{e}=\{e_1,\dots,e_n\}$ (où les $e_i$ sont deux à deux distincts) de $\B_A(G)$ et considérons la {\em fonction longueur} associée sur $G$, notée $x\mapsto |x|$ (ou $|x|_\mathbf{e}$ lorsqu'une ambiguïté est possible), c'est-à-dire la fonction longueur sur le groupe libre $G$ provenant de sa <<~base~>> constituée des $^a\!e_i$ pour $a\in A$ et $i\in\{1,\dots,n\}$. On note également $d(x,y)=\frac{1}{2}(|x|+|y|-|xy^{-1}|)$ pour $(x,y)\in G^2$. Ces fonctions vérifient les propriétés suivantes.

(C0) $|.|$ et $d$ prennent leurs valeurs dans les entiers ;

(A0') pour $x\in G$ et $a\in A$, on a $|x|<|x.\,^a\!x|$, sauf si $x$ est de la forme $t.\,^a\!t^{-1}$ pour un $t\in G$ ;

(A1) $|x|\geq 0$ pour tout $x\in G$, et $|x|=0$ si et seulement si $x=1$ ;

(A2) $|x^{-1}|=|x|$ ;

(A2') pour tous $x\in G$ et $a\in A$, $|^a\! x|=|x|$ ;

(A') pour $x\in G$ et $a\in A$, si $a\neq 1$, alors $|x.\,^a\!x^{-1}|=2|x|$ (ce qui équivaut à $d(x,\,^a\!x)=0$ compte-tenu de (A2')) ;

(A3) $d(x,y)\geq 0$ ;

(A4) $d(x,y)<d(x,z)$ implique $d(y,z)=d(x,y)$ (souvent utilisé sous la forme : $d(x,y)\geq m$ et $d(y,z)\geq m$ impliquent $d(x,z)\geq m$) ;

(A5') Pour $(x,y)\in G^2$ et $a\in A$ tels que $d(x,y)+d(^a\!x^{-1},y^{-1})\geq |x|=|y|$, il existe $(r,s)\in G^2$ tel que $x=rs$ et $y=\, ^a\!r.s$

Les propriétés numérotées sans prime ne dépendent pas de l'action de $A$ et sont établies dans \cite{Lyn}. La propriété (A2') est évidente ; montrons (A0'), (A') et (A5').

(A0') : Écrivons $x$ sous la forme d'un mot réduit
$$x=\,^{\alpha_1}\!b_1\dots\,^{\alpha_j}\!b_j$$
où $b_i\in\mathbf{e}\cup\mathbf{e}^{-1}$ et $^{\alpha_i}\!b_i.\,^{\alpha_{i+1}}\!b_{i+1}\neq 1$. Alors $|x|=j$, et $|x.\,^a\!x|=2(j-k)$ où $k$ est le plus grand entier tel que $^{\alpha_{j-i+1}}\!b_{j-i+1}.\,^{a\alpha_{i}}\!b_{i}=1$ pour $1\leq i\leq k$. On en déduit que $x$ s'écrit (de façon unique) sous la forme d'un mot réduit $t.u.\,^a\!t^{-1}$ avec $|t|=k$, et $|u.\,^a\!u|=2|u|$. On a $|x|=2|t|+|u|<2|t|+2|u|=|t.u.\,^a\!u.\,^{a^2}\!t^{-1}|=|x.\,^a\!x|$, sauf si $u=1$ i.e. $x=t.\,^a\!t^{-1}$.

(A') : dans l'écriture de $x$ sous forme de mot réduit, la première lettre de $^a\,x^{-1}$ est égale à $^a\!y^{-1}$, où $y$ est la dernière lettre de $x$. Comme $y.^a\!y^{-1}\neq 1$ (pour $a\neq 1$), il n'y a pas de simplification possible, d'où le résultat.

Une conséquence de (A') est la suivante.

\begin{pr}\label{unimod}
 Soient $x\in G$ et $a\in A$. Alors l'élément $x.\,^a\!x^{-1}$ de $G$ n'est jamais unimodulaire.
\end{pr}

\begin{proof}
 Si $a=1$, c'est évident (l'élément est le neutre). Sinon, par (A'), $|x.\,^a\!x^{-1}|_\mathbf{e}\neq 1$ pour {\em tout} élément $\mathbf{e}$ de $\B_A(G)$, ce qui interdit que $x$ soit un terme de $\mathbf{e}$.
\end{proof}

(A5') : écrivons $x=rs$ et $y=ts$ de sorte que les mots $rs$, $ts$ et $rt^{-1}$ soient réduits --- ainsi, $|s|=d(x,y)$. On a $^a\!x^{-1}=\,^a\!s^{-1}.\,^a\!r^{-1}$ et $y^{-1}=s^{-1}t^{-1}$ ; ces écritures sont réduites et par hypothèse $d(\,^a\!x^{-1},y^{-1})\geq |x|-d(x,y)=|r|$. Comme $|r|=|t|$ (car $|x|=|y|$), on en déduit $^a\!r^{-1}=t^{-1}$ i.e. $t=\,^a\!r$, d'où la conclusion.

Pour démontrer le théorème~\ref{th-niel}, utilisons la terminologie suivante, tirée de \cite{Petresco}. Tout d'abord, si $E$ est un sous-ensemble de $G$, on note $\bar{E}$ l'ensemble des éléments de $G$ de la forme $^a\!x$ ou $^a\!x^{-1}$ pour un $x\in E$ et un $a\in A$.  Soit $\mathbf{e}$ un élément de $\B_A(G)$. On dit qu'un sous-ensemble $E$ de $G$ est {\em $\mathbf{e}$-progressif} (ou simplement  {\em progressif} s'il n'y a pas d'ambiguïté sur $\mathbf{e}$) si pour toute suite finie non vide $(u_0,\dots,u_n)$ d'éléments de $\bar{E}$, on a
$$|u_0.u_1\dots u_n|_\mathbf{e}\geq |u_1\dots u_n|_\mathbf{e}$$
dès lors que $u_{i-1}.u_i\neq 1$ pour tout $1\leq i\leq n$. Si c'est le cas, sous la même hypothèse, on a aussi $|u_0.u_1\dots u_n|\geq |u_i\dots u_j|$ pour tous entiers $0\leq i\leq j\leq n$.

Disons par ailleurs qu'un sous-ensemble de $G$ en est une {\em $A$-base partielle} si c'est une partie d'une $A$-base. Le théorème~\ref{th-niel} se déduira aisément de la proposition suivante, qui adapte le {\em Theorem} de \cite{Hoa}.

\begin{pr}\label{prhoa}
 Soient $G$ un $A$-groupe libre de type fini muni d'une $A$-base, $X$ une $A$-base partielle progressive de $G$, $x'$ un élément de $G$ tel que $|x'|\geq |x|$ pour tout $x\in X$ et $X':=X\cup\{x'\}$. On suppose également que $X'$ est encore une $A$-base partielle de $G$. Si $X'$ n'est pas progressif, alors il existe une transformation de Nielsen de $X'$ qui laisse invariant $X$ et réduit strictement la longueur de $x'$.
\end{pr}

\begin{proof}
 Supposons que $X'$ n'est pas progressif : il existe alors des éléments $b_0,\dots,b_n$ de $\overline{X'}$ tels que $|b_0 b_1\dots b_n|<|b_1\dots b_n|$. On peut supposer que $n$ est minimal parmi les entiers pour lesquels existe un tel phénomène. Pour $0\leq i\leq n$, posons $c_i=b_{i+1}\dots b_n$ (de sorte que $c_{n-1}=b_n$ et $c_n=1$). Comme dans la démonstration du {\em Theorem} de \cite{Hoa}, en appliquant son {\em Lemma}, on voit que $|c_0|=|c_1|=\dots=|c_{n-1}|$.
   
 Comme $X$ est progressif, il existe $i$ tel que $b_i$ soit de la forme $^a\!x'$ ou $^a\!x'^{-1}$ pour un $a\in A$ ; soit $i$ le plus petit tel entier. S'il existe au moins deux tels entiers, nous noterons $j$ le plus petit strictement supérieur à $i$.
 
 On va voir qu'il ne peut pas exister deux tels entiers. Cela permettra de conclure, car l'unicité de l'occurrence de $x'$ dans le mot $w=b_0\dots b_n$ montre qu'il existe une transformation de Nielsen qui laisse invariant $X$ et transforme $x'$ en $w$, qui est de longueur strictement inférieure à la longueur de $x$ (car $|w|<|c_0|$ par hypothèse, et $|c_0|=|c_{n-1}|=|b_n|$, or $b_n$ appartient à $\overline{X'}$, donc est de longueur au plus égale à celle de $x'$).
 
 On cherche donc à établir une contradiction, en supposant que $i$ et $j$ existent. On va distinguer à cet effet plusieurs cas, en notant qu'on a soit $b_j=\,^a\!b_i$ pour un $a\in A$, soit $b_j=\,^a\!b_i^{-1}$. Mais on fait tout d'abord quelques remarques préliminaires.
 
 Pour $1\leq k\leq n-1$, on a
 $$d(c_{k-1},c_k)=\frac{1}{2}(|c_{k-1}|+|c_k|-|b_k|)=|c_0|-|b_k|/2\geq |c_0|-|b_i|/2$$
 puisque $|b_i|=|x'|\geq |b_k|$ pour tout $k$ par hypothèse. En utilisant (A4), on en déduit
 \begin{equation}\label{eqa1}
  d(c_l,c_k)\geq |c_0|-|b_i|/2\quad\text{pour}\quad 0\leq l<k\leq n-1.
 \end{equation}

 On a également
\begin{equation}\label{eqa2}
 d(c_k^{-1},b_k)=\frac{1}{2}(|c_k|+|b_k|-|c_{k-1}|)=\frac{1}{2}|b_k|\quad\text{pour}\quad 1\leq k\leq n-1.
\end{equation}

 Supposons d'abord $b_j=\,^a\!b_i$ ($a\in A$) et $j\leq n-1$. Par (\ref{eqa2}), on a $d(c_j^{-1},^a\!b_i)=d(c_j^{-1},b_j)=\frac{1}{2}|b_j|=\frac{1}{2}|b_i|$ et $d(^a\!c_i^{-1},^a\!b_i)=d(c_i^{-1},b_i)=\frac{1}{2}|b_i|$, d'où par (A4) $d(^a\!c_i^{-1},c_j^{-1})\geq\frac{1}{2}|b_i|$. Utilisant (\ref{eqa1}) pour $l=i$ et $k=j$, on en déduit
 $$d(^a\!c_i^{-1},c_j^{-1})+d(c_i,c_j)\geq |c_0|=|c_i|=|c_j|.$$
 Par (A5'), on en déduit qu'il existe $r\in G$ tel que $c_i c_j^{-1}=r.^a\!r^{-1}$. Or $c_i c_j^{-1}=b_{i+1}\dots b_j$ est unimodulaire, car $b_j$ est de la forme $^\alpha\!x'$ ou $^\alpha\!x'^{-1}$ avec $\alpha\in A$, tandis que $b_{i+1},\dots,b_{j-1}$ appartiennent à $\bar{X}$ (par définition de $i$ et $j$), de sorte qu'il existe une transformation de Nielsen qui change $x'$ en $c_i c_j^{-1}$ et laisse $X$ invariant ; comme par hypothèse $X'$ est une $A$-base partielle de $G$, notion préservée par transformation de Nielsen, on voit que $c_i c_j^{-1}$ appartient à une $A$-base partielle, i.e. est unimodulaire, de sorte que la proposition~\ref{unimod} fournit une contradiction.
 
 Supposons maintenant $b_j=\,^a\!b_i$ ($a\in A$) et $j=n$. Comme $b_n=c_{n-1}$, par (\ref{eqa1}) on a $d(c_i,b_n)\geq |c_0|-|b_i|/2=|c_{n-1}|-\frac{1}{2}|b_n|=\frac{1}{2}|b_n|$. Par (\ref{eqa2}) on a d'autre part $d(^a\! c_i^{-1},b_n)=d(^a\! c_i^{-1},^a\!b_i)=d(c_i^{-1},b_i)=\frac{1}{2}|b_i|=\frac{1}{2}|b_n|$. Par (A4), on en tire $d(c_i,^a\!c_i^{-1})\geq\frac{1}{2}|b_n|=\frac{1}{2}|c_{n-1}|=\frac{1}{2}|c_i|$. Comme $d(c_i,^a\!c_i^{-1})=|c_i|-\frac{1}{2}|c_i.^a\!c_i|$, il vient $|c_i.^a\!c_i|\leq |c_i|$ puis, par (A0') que $c_i$ est de la forme $t.^a\!t^{-1}$, donc non unimodulaire (proposition~\ref{unimod}). Comme $c_i=c_i c_n^{-1}$, on en tire une contradiction comme précédemment.
 
 Supposons enfin $b_j=\,^a\!b_i^{-1}$ ($a\in A$). On a
 $$d(c_{j-1}^{-1},^a\!b_i)=d(c_{j-1}^{-1},b_j^{-1})=\frac{1}{2}(|c_{j-1}|+|b_j|-|c_j|)\geq\frac{1}{2}|b_j|=\frac{1}{2}|b_i|$$
(en effet, si $j<n$, alors $|c_{j-1}|=|c_j|$, et si $j=n$, alors $|c_n|=0$). On a également $|b_i c_i|\leq |c_i|$ : pour $i>0$, on a égalité car $b_i c_i=c_{i-1}$ (et $i<n$) ; pour $i=0$, on a une inégalité stricte par l'hypothèse initiale. En conséquence, on a
$$d(^a\!c_i^{-1},^a\!b_i)=d(c_i^{-1},b_i)=\frac{1}{2}(|c_i|+|b_i|-|b_i c_i|)\geq\frac{1}{2}|b_i|.$$
Les deux inégalités qu'on a établies impliquent, par (A4) : $d(^a\!c_i^{-1}c_{j-1}^{-1})\geq\frac{1}{2}|b_i|$. Si l'on suppose de plus $i<j-1$, par (\ref{eqa1}), comme on a $d(c_i,c_{j-1})\geq |c_0|-\frac{1}{2}|b_i|$, il vient
$$d(^a\!c_i^{-1},c_{j-1}^{-1})+d(c_i,c_{j-1})\geq |c_0|=|c_i|=|c_{j-1}|.$$
On en déduit une contradiction par le même argument que plus haut.

Reste à exclure le cas $b_j=\,^a\!b_i^{-1}$ et $j=i+1$. On a alors forcément $a\neq 1$, car sinon on pourrait simplifier $b_i.b_{i+1}=1$, contredisant la minimalité de $n$. Par (A'), $|b_i b_{i+1}|=2|b_i|$, donc $d(b_i b_{i+1},b_{i+1})=\frac{1}{2}(2|b_i|+|b_{i+1}|-|b_i|)=|b_i|$. Par ailleurs, si $i<n-1$, alors $|b_i b_{i+1} c_{i+1}|\leq|c_{i+1}|$ (inégalité stricte si $i=0$, égalité sinon, car $|c_{i-1}|=|c_{i+1}|$ pour $i+1<n$). Mais $i=n-1$ est exclu, car $|b_i.b_{i+1}|>|b_{i+1}|$, ce qui contredit l'hypothèse si $n=1$ et l'égalité $|c_{n-1}|=|c_{n-2}|$ si $n\geq 2$. De $|b_i b_{i+1} c_{i+1}|\leq|c_{i+1}|$ on tire $d(b_i b_{i+1},c_{i+1}^{-1})\geq\frac{1}{2}|b_i b_{i+1}|=|b_i|$. Par (A4) (et l'égalité $d(b_i b_{i+1},b_{i+1})=|b_i|$), on en déduit $|b_i|\leq d(b_{i+1},c_{i+1}^{-1})=\frac{1}{2}(|b_{i+1}|+|c_{i+1}|-|b_{i+1}c_{i+1}|)=\frac{1}{2}(|b_{i}|+|c_{i+1}|-|c_{i}|)=\frac{1}{2}|b_{i}|$ puisque $i+1\leq n-1$. C'est manifestement absurde ($b_i\neq 1$), d'où la conclusion.
\end{proof}

\begin{proof}[Démonstration du théorème~\ref{th-niel}]
 Soit $\mathbf{e}$ une $A$-base de $G$. Appelons {\em longueur} de $\mathbf{e}$ la suite ordonnée des longueurs des éléments de $\mathbf{e}$ (avec ses éventuelles répétitions), classées par ordre croissant. Supposons que $\mathbf{e}$ n'est équivalente à la Nielsen à aucune $A$-base progressive de $G$. On peut supposer que, parmi ces bases, $\mathbf{e}$ est de longueur minimale, pour l'ordre lexicographique sur les suites finies d'entiers naturels. Si $\mathbf{e}=\{e_1,\dots,e_n\}$ avec $|e_1|\leq\dots\leq |e_n|$, notons $i$ le plus grand entier naturel (éventuellement nul) tel que $\{e_1,\dots,e_i\}$ soit progressive : on a par hypothèse $i<n$. Appliquant la proposition~\ref{prhoa}, on voit que $\{e_1,\dots,e_{i+1}\}$ est équivalente à la Nielsen à une suite $\{e_1,\dots,e_i,e'_{i+1}\}$ où $|e'_{i+1}|<|e_{i+1}|$ ; $\mathbf{e}$ est alors équivalente à la Nielsen à $\{e_1,\dots,e_i,e'_{i+1}, e_{i+2},\dots,e_n\}$, ce qui contredit la minimalité de la longueur de $\mathbf{e}$.
 
 Il suffit donc, pour conclure, de voir qu'une base progressive est équivalente à la Nielsen à la base $\mathbf{b}$ fixée depuis le début sur $G$ (celle relativement à laquelle on prend la longueur). En effet, pour tout $i$, $b_i$, qui est de longueur $1$, s'écrit comme un produit $x_1...x_t$ d'éléments non triviaux de $\bar{\mathbf{e}}$, puisque $\mathbf{e}$ est une $A$-base de $G$. Mais l'hypothèse de progressivité implique que chaque $x_j$ est de longueur $1$, et aussi les $x_j x_{j+1}$ (on suppose le mot réduit). Mais $|x_j|=|x_{j+1}|=|x_j x_{j+1}|=1$ est impossible, donc $t=1$ et $b_i\in\bar{\mathbf{e}}$ pour tout $i$, ce qui montre que notre base $\mathbf{e}$ est équivalente à $\mathbf{b}$. 
\end{proof}

\subsection{Un critère abstrait de discrétion}\label{scritabst}

Donnons d'abord une notation : si $E$ est un ensemble préordonné (qu'on voit aussi bien, comme d'habitude, comme une petite catégorie ou un ensemble simplicial via le foncteur nerf, ce qui permet de parler de son type d'homotopie) et $x$ un élément de $E$, on note $[x,\to[_E$, ou simplement $[x,\to[$, l'ensemble $\{t\in E\,|\,x\leq t\}$.

On commence par un lemme général, classique et élémentaire.

\begin{lm}\label{contrac-elem}
 Soit $E$ un ensemble préordonné. On suppose que tout élément de $E$ est supérieur ou égal à un élément minimal et que les deux conditions suivantes sont vérifiées :
 \begin{enumerate}
  \item\label{hun} pour tout ensemble fini non vide $X$ d'éléments minimaux de $E$, le sous-ensemble préordonné $\underset{x\in X}{\bigcap}[x,\to[$ de $E$ est vide ou contractile ;
  \item pour tout ensemble fini non vide $X$ d'éléments minimaux de $E$, si les intersections $[x,\to[\, \cap\, [y,\to[$ sont non vides pour tous éléments $x$ et $y$ de $X$, alors $\underset{x\in X}{\bigcap}[x,\to[$ est non vide.
 \end{enumerate}
Alors $E$ a un type d'homotopie discret.
\end{lm}

\begin{proof}
 L'ensemble des parties $P$ de $E$ telles que $x\in P$ et $x\leq y$ implique $y\in P$ est stable par intersection et réunion, il contient les parties de la forme $[x,\to[$ et la restriction du foncteur nerf à ces parties de $E$ commute aux réunions (elle commute aussi, toujours, aux intersections). 
 
 Par conséquent, l'hypothèse~\ref{hun} montre que l'ensemble simplicial $\N(E)$ a le même type d'homotopie que le nerf du recouvement de $E$ par les $[x,\to[$ pour $x$ minimal. La dernière hypothèse implique que chaque composante connexe de ce nerf est contractile, d'où le lemme.
\end{proof}

Considérons maintenant la situation suivante : $A$ est un groupe, $X$ un ensemble muni d'une action (à gauche) de $A$. On suppose données deux opérations sur les familles (finies ou non) d'éléments de $X$ :
\begin{enumerate}
 \item une opération (partout définie), notée $\bigwedge$, compatible à l'action de $A$, c'est-à-dire telle que
 $$a.\Big(\underset{i\in I}{\bigwedge}x_i\Big)=\underset{i\in I}{\bigwedge} a.x_i$$
 pour tout $a\in A$ et toute famille $(x_i)_{i\in I}$ d'éléments de $X$, envoyant une famille réduite à un élément sur cet élément et vérifiant la propriété d'associativité usuelle
 $$\underset{i\in I}{\bigwedge}x_i=\underset{J\in\J}{\bigwedge}\Big(\underset{j\in J}{\bigwedge}x_j\Big)$$
 si l'ensemble $I$ est la réunion disjointe de l'ensemble $\J$
 et de commutativité : si $\phi : I\to I$ est une bijection, alors $\underset{i\in I}{\bigwedge}x_{\phi(i)}=\underset{i\in I}{\bigwedge}x_i$ ;
 \item une opération {\em partiellement définie}, appelée somme et notée $\sum$, compatible à l'action de $A$ (c'est-à-dire que, si $(x_i)_{i\in I}$ est une famille d'éléments de $X$ et $a$ un élément de $A$, $\underset{i\in I}{\sum} a.x_i$ est défini si et seulement si $\underset{i\in I}{\sum}x_i$ l'est, auquel cas on a $\underset{i\in I}{\sum} a.x_i=a.(\underset{i\in I}{\sum}x_i)$), associative et commutative (au même sens que précédemment, avec l'adaptation évidente liée au caractère partiellement défini de l'opération), telle que la somme de la famille réduite à un élément $x$ soit définie et égale à $x$. On note $0$ la somme de la famille vide\,\footnote{Cette somme vide est définie parce que la somme d'une famille réduite à un élément est supposée définie, que l'axiome d'associativité pour la somme contient le fait qu'une sous-famille d'une famille dont la somme est définie a encore une somme bien définie, et que l'ensemble $X$ est nécessairement non vide puisque l'opération $\bigwedge$ est définie sur toutes les familles d'éléments de $X$, y compris la famille vide.}. On suppose également que $x+x$ n'est défini que si $x=0$, et que la somme est {\em régulière} au sens où $x+y=x$ entraîne $y=0$.
\end{enumerate}

On suppose que nos deux opérations vérifient les propriétés de compatibilité suivantes.

(P0) si $(x_i)_{i\in I}$ est une famille d'éléments de $X$ telle que $\underset{i\in I}{\sum}x_i$ est définie, alors, pour toute famille $(J_t)_{t\in T}$ de parties de $I$, on a
$$\underset{t\in T}{\bigwedge}\underset{j\in J_t}{\sum} x_j=\underset{i\in\underset{t\in T}{\bigcap}J_t}{\sum} x_i.$$

(P1) Si $x$ est un élément de $X$ et $(t_i)_{i\in I}$, $(u_j)_{j\in J}$ sont des familles d'éléments de $X$ telles que $x=\underset{i\in I}{\sum} t_i=\underset{j\in J}{\sum} u_j$, alors il y a équivalence entre :
\begin{enumerate}
 \item $x=\underset{(i,j)\in I\times J}{\sum} t_i\wedge u_j$ ; 
 \item pour tout $i\in I$, $t_i=\underset{j\in J}{\sum} t_i\wedge u_j$ ;
 \item pour tout $j\in J$, $u_j=\underset{i\in I}{\sum} t_i\wedge u_j$.
\end{enumerate}
(Noter que 2. ou 3. implique automatiquement 1.)

(P2) Si $u$, $v$ et $t_i$ ($i\in I$) sont des éléments de $X$ tels que  $\underset{i\in I}{\sum}t_i$ est définie, $u=\underset{i\in I}{\sum} u\wedge t_i$ et $v=\underset{i\in I}{\sum} v\wedge t_i$, alors $u\wedge v=\underset{i\in I}{\sum} u\wedge v\wedge t_i$.

On se fixe un élément $s$ de $X$ {\em invariant par $A$} et on note $\B$ l'ensemble des $x\in X$ tels que $\underset{a\in A}{\sum} a.x=s$. On fait enfin l'hypothèse de finitude suivante.

(PF) Si $\underset{i\in I}{\sum} x_i\in\B$, alors $x_i=0$ sauf pour un nombre fini de $i\in I$.

\begin{defi}\label{catdec}
 On appelle {\em catégorie des décompositions} associée aux données précédentes la catégorie, notée $\D$, dont :
 \begin{enumerate}
  \item les objets sont les familles finies $(x_1,\dots,x_n)$ (où $n\in\mathbb{N}$ est quelconque) d'éléments de $X\setminus\{0\}$ telles que $\underset{1\leq i\leq n}{\sum} x_i\in\B$ ;
  \item les morphismes $(x_1,\dots,x_n)\to (y_1,\dots,y_m)$ sont les couples $(\varphi ; a_1,\dots, a_m)$ constitués d'une fonction surjective $\varphi : \mathbf{m}\to\mathbf{n}$ et d'un élément $(a_1,\dots,a_m)$ de $A^m$ tels que
  $$x_i=\underset{\varphi(j)=i}{\sum}a_j.y_j$$
  pour tout $i\in\mathbf{n}$ ;
  \item la composition de
  $$(x_1,\dots,x_n)\xrightarrow{(\varphi ; a_1,\dots, a_m)} (y_1,\dots,y_m)\xrightarrow{(\psi ; b_1,\dots, b_l)}(z_1,\dots,z_l)$$
  est le morphisme
  $$\big(\varphi\circ\psi,(a_{\psi(k)}b_k)_{1\leq k\leq l}\big) : (x_1,\dots,x_n)\to (z_1,\dots,z_l).$$
 \end{enumerate}
\end{defi}

On dispose ainsi d'un foncteur canonique $\D^{op}\to\Omega$, qui sur les objets s'obtient par la longueur des familles finies.

\begin{pr}\label{discr-abstr}
 Si $s\neq 0$, alors la catégorie $\D$ est un ensemble préordonné dont le type d'homotopie est discret.
\end{pr}

\begin{proof}
 On commence par vérifier que $\D$ est un ensemble préordonné, c'est-à-dire qu'il y a au plus un morphisme entre deux de ses objets. Supposons en effet que
 $(\varphi;a_1,\dots,a_m)$ et $(\varphi';a'_1,\dots,a'_m)$ sont des éléments de $\D\big((x_1,\dots,x_n),(y_1,\dots,y_m)\big)$. Par définition de $\D$ et $\B$, la somme $\underset{i\in\mathbf{n}}{\underset{a\in A}{\sum}}a.x_i$ est définie. Soit $j\in\mathbf{m}$ ; posons $i:=\varphi(j)$ et $i':=\varphi'(j)$. Si $i\neq i'$ ou $a_j\neq a'_j$, la somme $x_i+a_j {a'_j}^{-1}.x_{i'}$ est définie. Cela implique, puisque $x_i$ (resp. $x_{i'}$) est la somme de $a_j.y_j$ (resp. $a'_j.y_j$) et d'autres termes, que $a_j.y_j+a_j.y_j$, donc $y_j+y_j$, est définie, donc, par hypothèse, que $y_j=0$, ce qui contredit la définition des objets de $\D$ et établit que $\varphi'(j)=\varphi(j)$ et $a_j=a'_j$ pour tout $j$.
 
 Avant de terminer la démonstration, on donne un résultat intermédiaire qui nous permettra d'appliquer le lemme~\ref{contrac-elem}.
\end{proof}

Dans ce qui suit, on note de la même façon un élément de $\B$ et l'objet de $\D$ (qui est donc une famille à un objet) qu'il définit. Noter qu'on obtient ainsi exactement tous les élément {\em minimaux} de l'ensemble préordonné $\D$.

\begin{lm}\label{lm-intersect}
 Soit $\mathfrak{X}$ une partie finie non vide de $\B$. Les assertions suivantes sont équivalentes.
 \begin{enumerate}
  \item\label{le1} Le sous-ensemble $\underset{x\in\mathfrak{X}}{\bigcap} [x,\to[$ de $\D$ est non vide ;
  \item\label{le3} pour toute partie non vide $\mathfrak{X}'$ de $\mathfrak{X}$, on a $s=\sum\underset{x\in \mathfrak{X}'}{\bigwedge}a(x).x$, où la somme est prise sur toutes les fonctions $a : \mathfrak{X}'\to A$ ; 
  \item\label{le4} pour tous éléments $x$ et $y$ de $\mathfrak{X}$, on a $y=\underset{a\in A}{\sum}y\wedge a.x$ ;
 \end{enumerate}
 
De plus, lorsqu'elles sont vérifiées, il existe un élément $T$ de $\D$ tel que $\underset{x\in\mathfrak{X}}{\bigcap} [x,\to[ = [T,\to[$.
\end{lm}

\begin{proof}
Supposons {\em \ref{le1}.} vérifié. Il suffit de montrer {\em \ref{le3}} pour $\mathfrak{X}'=\mathfrak{X}$, puisqu'on a a fortiori $\underset{x\in\mathfrak{X}'}{\bigcap} [x,\to[\neq\emptyset$, ce qu'on suppose désormais.

 Soit $Y=(y_1,\dots,y_n)$ un élément de l'intersection $\underset{x\in\mathfrak{X}}{\bigcap} [x,\to[_\D$. Pour tout $x\in\mathfrak{X}$, la relation $x\leq Y$ signifie qu'il existe des éléments $\alpha_i(x)$ de $A$ tels que $x=\underset{i\in\mathbf{n}}{\sum}\alpha_i(x).y_i$. Comme la somme $\underset{(i,a)\in\mathbf{n}\times A}{\sum} a.y_i$ est définie, l'hypothèse (P0) appliquée aux parties $J_x^a(Y):=\{(i,a(x)\alpha_i(x))\,|\,i\in\mathbf{n}\}$ (notées simplement $J_x^a$ s'il n'y a pas d'ambiguïté) de $\mathbf{n}\times A$, pour $x\in\mathfrak{X}$, où $a\in A^\mathfrak{X}$, montre que
 $$\underset{x\in \mathfrak{X}}{\bigwedge}a(x).x=\underset{x\in \mathfrak{X}}{\bigwedge}\underset{i\in\mathbf{n}}{\sum}a(x)\alpha_i(x).y_i=\underset{(i,c_i)\in\underset{x\in\mathfrak{X}}{\bigcap}J_x^a}{\sum}c_i.y_i\;;$$
 $\underset{x\in\mathfrak{X}}{\bigcap}J_x^a$ est constituée des $(i,c_i)$ pour lesquels la fonction $\mathfrak{X}\to A\quad x\mapsto a(x)\alpha_i(x)$ est constante en $c_i$. Étant donné un élément $(i,c_i)$ de $\mathbf{n}\times A$, il existe un et un seul élément $a$ de $A^\mathfrak{X}$ tel que $(i,c_i)$ appartienne à $\underset{x\in\mathfrak{X}}{\bigcap}J_x^a$, qui est donné par $a(x)=c_i\alpha_i(x)^{-1}$. On en déduit
 $$\underset{a\in A^\mathfrak{X}}{\sum}\underset{x\in \mathfrak{X}}{\bigwedge}a(x).x=\underset{(i,a)\in\mathbf{n}\times A}{\sum}a.y_i=s\,,$$
 de sorte que {\em \ref{le1}} implique {\em \ref{le3}}. 
 
 Montrons également la dernière assertion : considérons un élément $T=(t_1,\dots,t_r)$ de longueur $r$ {\em minimale} dans $\underset{x\in\mathfrak{X}}{\bigcap} [x,\to[_\D$. Cela implique que $\underset{x\in\mathfrak{X}}{\bigcap}J_x^a(T)$ possède au plus un élément --- si $(i,c_i)$ et $(j,c_j)$ appartiennent à cet ensemble, on a $\alpha_i=\alpha_j$ (avec les notations précédentes) : si $i\neq j$, on peut alors remplacer $t_i$ et $t_j$ par $t_i+t_j$ dans $T$ pour obtenir une famille $T'$ de longueur $n-1$ qui sera encore supérieure à tous les éléments de $x$. Par conséquent, chaque terme $t_i$ de $T$ peut s'écrire sous la forme $\underset{x\in \mathfrak{X}}{\bigwedge}a(x).x$ pour un $a\in A^\mathfrak{X}$. Si $Y=(y_1,\dots,y_n)$ est un élément de $\underset{x\in\mathfrak{X}}{\bigcap} [x,\to[$, ce qui précède donne des éléments $c_j$ de $A$ et des sous-ensembles $E_i$, pour $i\in\mathbf{r}$, de $\mathbf{n}$ tels que $t_i=\underset{j\in E_i}{\sum}c_j.y_j$. Les $E_i$ forment nécessairement une partition de $\mathbf{n}$ : ils sont non vides car $t_i\neq 0$, disjoints parce que la somme $\underset{(i,a)\in\mathbf{r}\times A}{\sum}a.t_i$ est définie (utiliser, comme dans le début de la démonstration de la proposition~\ref{discr-abstr}, le fait que $x+x$ n'est pas défini pour $x\neq 0$), et ils recouvrent $\mathbf{r}$ en raison de la relation
 $$s=\underset{(i,a)\in\mathbf{r}\times A}{\sum}a.t_i=\underset{(j,a)\in\underset{i\in\mathbf{r}}{\cup} E_i\times A}{\sum}ac_j.y_j=\underset{(j,b)\in\mathbf{r}\times A}{\sum}b.y_i$$
 et de la régularité de la somme.
 
 Montrons que {\em \ref{le3}} implique {\em \ref{le4}} : si l'on applique l'hypothèse {\em \ref{le3}} à $\mathfrak{X}'=\{x,y\}$, on obtient
 $$s=\underset{(a,b)\in A^2}{\sum} a.x\wedge b.y=\underset{a\in A}{\sum}a.x= \underset{b\in A}{\sum}b.y\,,$$
 de sorte que l'hypothèse (P1) donne la conclusion.
 
  {\em \ref{le4}} $\Rightarrow$ {\em \ref{le1}} : l'application itérative de la propriété (P2) montre que, sous l'hypothèse {\em \ref{le4}}, si $n\geq 1$ est un entier, $(x_0,\dots,x_n)$ une famille d'éléments de $\mathfrak{X}$ et $(a_1,\dots,a_n)$ une famille d'éléments de $A$, on a
  $$a_1.x_1\wedge\dots\wedge a_n.x_n=\underset{a_0\in A}{\sum}a_0.x_0\wedge a_1.x_1\wedge\dots\wedge a_n.x_n.$$
  
  On en déduit également (encore par récurrence sur $n$) que
  \begin{equation}\label{eqx1}
x_0=\underset{(a_1,\dots,a_n)\in A^n}{\sum}x_0\wedge a_1.x_1\wedge\dots\wedge a_n.x_n   
  \end{equation}
 
 Fixons-nous un élément $y$ de $\mathfrak{X}$ et considérons la famille (ordonnée arbitrairement) des éléments {\em non nuls} de $X$ de la forme $\underset{x\in \mathfrak{X}}{\bigwedge}a(x).x$, où $a$ parcourt l'ensemble des fonctions $\mathfrak{X}\to A$ telles que $a(y)=1$. La somme de cette famille égale $y\in\B$ (par \eqref{eqx1}), elle est donc finie grâce à l'hypothèse (PF), et définit un élément $T$ de $\D$. La relation \eqref{eqx1} montre également que, pour tout $t\in\mathfrak{X}$, on a
 $$t=\underset{a(y)=1}{ \underset{a\in A^\mathfrak{X}}{\sum}}a(t)^{-1}.\underset{x\in\mathfrak{X}}{\bigwedge}a(x).x$$
 d'où l'on déduit $x\leq T$ (dans $\D$). Ainsi $T$ appartient à $\underset{x\in\mathfrak{X}}{\bigcap} [x,\to[$, ce qui achève la démonstration.
\end{proof}

\begin{proof}[Fin de la démonstration de la proposition~\ref{discr-abstr}]
 Dans l'ensemble préordonné $\D$, tout élément $(x_1,\dots,x_n)$ est supérieur à l'élément (de longueur $1$) $x_1+\dots+x_n$, qui est minimal (en effet, on a supposé $s\neq 0$). La dernière assertion du lemme~\ref{lm-intersect} montre que l'intersection d'une famille finie non vide de sous-ensembles de $\D$ du type $\underset{x\in\mathfrak{X}}{\bigcap}[x,\to[$, où les éléments de $\mathfrak{X}$ sont {\em minimaux} dans $\D$, c'est-à-dire appartiennent à $\B$, est vide ou contractile. L'équivalence {\em \ref{le1}}\,$\Leftrightarrow$ {\em \ref{le4}} de ce même lemme montre que la dernière condition du lemme~\ref{contrac-elem} est satisfaite, de sorte que celui-ci donne la conclusion recherchée.
\end{proof}

\subsection{Application : la catégorie $\mathfrak{D}_A(G)$ des décompositions d'un $A$-groupe libre $G$}\label{pco}

Soient $G$ un groupe et $X$ l'ensemble des sous-groupes de $G$. On va appliquer les résultats du paragraphe~\ref{scritabst} avec pour opération partiellement définie de <<~somme~>> le produit libre interne (noté $\bigstar$), l'opération (partout définie) $\bigwedge$ étant l'intersection. Ces opérations sont bien associatives et commutatives ; le produit libre interne est une opération régulière ($H*K=H$ implique $K=\{1\}$, où $H$ et $K$ sont des sous-groupes de $G$) et le produit libre interne d'un sous-groupe $H$ de $G$ avec lui-même n'est défini que si $H$ est trivial.

Vérifions les propriétés (P0), (P1) et (P2).

(P0) : soient $(H_i)_{i\in I}$ une famille de sous-groupes de $G$ et $(J_t)_{t\in T}$ une famille de parties de $I$. On a toujours
$$\underset{i\in\underset{t\in T}{\bigcap}J_t}{\bigstar}H_i\subset\underset{t\in T}{\bigcap}\,\underset{j\in J_t}{\bigstar}H_j\;;$$
l'inclusion inverse est vraie lorsque $\underset{i\in I}{\bigstar}H_i$ est défini, comme on le voit en examinant l'unique écriture réduite dans cette décomposition en produit libre d'un élément du terme de droite.

(P1) : si $(H_i)_{i\in I}$ et $(K_j)_{j\in J}$ sont des familles de sous-groupes de $G$ telles que
$$\underset{i\in I}{\bigstar}T_i=\underset{j\in J}{\bigstar}U_j=\underset{(i,j)\in I\times J}{\bigstar} T_i\cap U_j\,,$$
considérons la décomposition réduite dans le produit libre de droite d'un élément de $H_t$, pour un $t\in I$ fixé. Cette décomposition ne fait intervenir que des éléments d'intersections de la forme $T_t\cap U_j$ car, quitte à regrouper certains termes (les termes consécutifs pour lesquels la valeur de l'indice $i$ est constante), elle fournit une décomposition de notre élément de $H_t$ dans le produit libre $\underset{i\in I}{\bigstar}T_i$, puisque $T_i\cap U_j\subset T_i$. On a donc $T_t=\underset{j\in J}{\bigstar} T_t\cap U_j$ comme souhaité.

(P2) : si $H$, $K$ et les $L_i$ sont des sous-groupes de $G$ tels que $\underset{i\in I}{\bigstar} L_i$ est défini, la relation $H=\underset{i\in I}{\bigstar} H\cap L_i$ signifie exactement que $H\subset\underset{i\in I}{\bigstar} L_i$ et que, dans l'écriture réduite d'un élément de $H$ dans cette décomposition en produit libre, tous les termes appartiennent à $H$. Si l'on fait la même hypothèse pour $K$, on voit que, dans l'écriture réduite d'un élément de $H\cap K$ dans le produit libre $\underset{i\in I}{\bigstar} L_i$, tous les termes appartiennent à $H$, et à $K$, donc à $H\cap K$ comme souhaité.

Supposons maintenant que $G$ est muni d'une action d'un groupe $A$ : l'action induite sur l'ensemble $X$ des sous-groupes de $G$ est compatible aux opérations d'intersection et de produit libre interne. Si l'on prend pour $s\in X$ le sous-groupe $G$ de $G$, cet élément est invariant par l'action de $A$, et l'ensemble noté $\B$ dans le paragraphe précédent est exactement l'ensemble $\mathfrak{B}_A(G)$ introduit dans la définition~\ref{base-agl}. Si $G$ est de type fini {\em comme $A$-groupe}, tous les éléments de cet ensemble sont de type fini, de sorte qu'un produit libre interne de sous-groupes de $G$ ne peut lui appartenir que s'ils sont tous triviaux, sauf un nombre fini d'entre eux : l'hypothèse (PF) est bien vérifiée.

\begin{nota}\label{ncdc}
 Soient $A$ un groupe et $G$ un $A$-groupe de type fini. On note $\mathfrak{D}_A(G)$ la catégorie des décompositions (notée $\D$ au paragraphe~\ref{scritabst}) associée aux données précédentes.
\end{nota}

\begin{thm}\label{th-ccon}
 Soient $A$ un groupe et $G$ un $A$-groupe libre de type fini. La catégorie $\mathfrak{D}_A(G)$ est un ensemble préordonné dont le nerf est contractile.
\end{thm}

\begin{proof}
 Le cas où $G$ est le $A$-groupe trivial est immédiat, car $\mathfrak{D}_A(G)$ est alors réduite à la famille vide.
 
Si $G$ n'est pas trivial, d'après la proposition~\ref{discr-abstr}, il suffit de vérifier que $\mathfrak{D}_A(G)$ est connexe (cette catégorie est non vide puisque $G$ est $A$-libre de type fini). Cela découle du théorème~\ref{th-niel}, car toute transformation de Nielsen entre deux éléments de $\mathfrak{B}_A(G)$ (dont on choisit des bases pour le produit libre) fournit un chemin entre eux dans $\mathfrak{D}_A(G)$. Supposons en effet que deux éléments $H$ et $K$ de $\mathfrak{B}_A(G)$ sont reliés par une transformation de Nielsen {\em élémentaire} : cela signifie qu'il existe une base $(e_1,\dots,e_n)$ de $H$, un indice $i\in\mathbf{n}$ et un élément $a$ de $A$ tels que $(e_1,\dots,e_{i-1},\,^a\!e_i,e_{i+1},\dots,e_n)$ soit une base de $K$ (les transformations de Nielsen élémentaires sur une base de $H$ autres que l'application d'un élément de $A$ sur l'un d'entre eux ne changent pas le sous-groupe engendré, on peut donc les négliger). On a alors, dans $\mathfrak{D}_A(G)$, $H\leq (T,U)$ et $K\leq (T,U)$, où $T$ (resp. $U$) désigne le sous-groupe de $G$ engendré par les $e_t$ pour $t\neq i$ (resp. par $e_i$), d'où le théorème.
\end{proof}

\section{Type d'homotopie de la catégorie $\cc(A)$}\label{sth}

\subsection{La catégorie $\G_A$ et le foncteur faible $\G_A\to\mathbf{Cat}$}
Le groupe $A$ étant (provisoirement) fixé, nous avons besoin d'une fonctorialité de l'association $G\mapsto\mathfrak{D}_A(G)$. La catégorie usuelle des $A$-groupes libres de rang fini $\mathbf{gr}_A$ n'est pas adaptée à cela ; nous en modifierons donc les flèches (de façon analogue à celle dont on obtient $\G$ à partir de $\gr$).

\begin{nota}\label{not-ga}
 On désigne par $\G_A$ la catégorie ayant les mêmes objets que $\mathbf{gr}_A$ et dont les morphismes $G\to H$ sont les couples $(u,T)$ constitués d'un monomorphisme de $A$-groupes $u : G\to H$ et d'un sous-groupe $T$ de $H$ (qu'on ne suppose pas stable par l'action de $A$) tel que $H$ soit le produit libre interne de l'image de $u$ et des $^a\!T$ pour $a\in A$.
 
 La composée $G\xrightarrow{(u,T)}H\xrightarrow{(v,S)}K$ est le couple constitué du monomorphisme $v\circ u : G\to K$ de $A$-groupes et du sous-groupe $S*v(T)$ de $K$.
\end{nota}

On note d'abord qu'on dispose d'un foncteur $\mathfrak{B}_A : \G_A\to\mathbf{Ens}$ (donné sur les objets par la définition~\ref{base-agl}) dont l'effet sur les morphismes est le suivant. Si $(u,T) : G\to H$ est un morphisme de $\G_A$, $\mathfrak{B}_A(u,T)$ associe à un sous-groupe $K$ de $G$ appartenant à $\mathfrak{B}_A(G)$ le sous-groupe $u(K)*T$ de $H$, qui appartient manifestement à $\mathfrak{B}_A(H)$.

Si $\C$ est une catégorie, on rappelle qu'un {\em foncteur faible} (la terminologie anglaise usuelle est {\em op-lax functor} \cite[§\,3]{T79}) $F : \C\to\mathbf{Cat}$ consiste en la donnée, pour tout objet $c$ de $\C$, d'une petite catégorie $F(c)$, d'un foncteur $F(\phi) : F(c)\to F(d)$ pour tout $\phi\in\C(c,d)$ et de transformations naturelles $F({\rm Id}_c)\Rightarrow {\rm Id}_{F(c)}$ pour tout $c\in {\rm Ob}\,\C$ et $F(\phi\circ\psi)\Rightarrow F(\phi)\circ F(\psi)$ pour tout couple $(\phi,\psi)$ de flèches composables de $\C$, assujettis à vérifier les conditions de cohérence usuelles (qu'on pourra trouver dans \cite[3.1.1]{T79}).

\begin{pr}\label{pr-ffbl} $\mathfrak{D}_A : \G_A\to\mathbf{Cat}$ définit un foncteur faible.
\end{pr}

\begin{proof} Soit $(u,T) : G\to H$ un morphisme de $\G_A$. Si $T$ est le groupe trivial (c'est-à-dire que $u$, ou $(u,T)$, est un isomorphisme), $\mathfrak{D}_A(u,T)$ associe à un objet $(K_1,\dots,K_n)$ de $\mathfrak{D}_A(G)$ l'objet $(u(K_1),\dots,u(K_n))$ de $\mathfrak{D}_A(H)$ et a l'effet évident sur les morphismes (qui restent inchangés). Lorsque $T$ est non trivial, $\mathfrak{D}_A(u,T)$ associe à  $(K_1,\dots,K_n)\in\mathfrak{D}_A(G)$ l'objet $(u(K_1),\dots,u(K_n),T)$ de $\mathfrak{D}_A(H)$. Si $(\varphi;a_1,\dots,a_m) : (K_1,\dots,K_n)\to (K'_1,\dots,K'_m)$ est un morphisme de $\mathfrak{D}_A(G)$ (on a donc $\varphi\in\Omega(\mathbf{m},\mathbf{n})$ et $a_i\in A$ pour tout $i\in\mathbf{m}$), $\mathfrak{D}_A(u,T)$ lui associe le morphisme $(\varphi_+;a_1,\dots,a_m,1) : (u(K_1),\dots,u(K_n),T)\to (u(K'_1),\dots,u(K'_m),T)$ de $\mathfrak{D}_A(H)$, où $\varphi_+\in\Omega(\mathbf{m+1},\mathbf{n+1})$ est la fonction coïncidant avec $\varphi$ sur $\mathbf{m}$ et envoyant $m+1$ sur $n+1$. La compatibilité de $\mathfrak{D}_A(u,T)$ à la composition (et aux identités) est automatique puisque les catégories $\mathfrak{D}_A(G)$ et $\mathfrak{D}_A(H)$ sont des ensembles préordonnés.

De plus, $\mathfrak{D}_A$ envoie les identités sur les identités, et l'on a $\mathfrak{D}_A({\rm v}\circ {\rm u})\leq\mathfrak{D}_A({\rm v})\circ\mathfrak{D}_A({\rm u})$ lorsque u et v sont des flèches composables de $\G_A$ --- avec égalité si et seulement si u ou v est un isomorphisme. En effet, si par exemple ni ${\rm u}=(u,T)$ ni ${\rm v}=(v,S)$ ne sont des isomorphismes, on a
\[\mathfrak{D}_A({\rm v}\circ {\rm u})(K_1,\dots,K_n)=(vu(K_1),\dots,vu(K_n),S*v(T))\]
tandis que
\[\big(\mathfrak{D}_A({\rm v})\circ\mathfrak{D}_A({\rm u})\big)(K_1,\dots,K_n)=(vu(K_1),\dots,vu(K_n),v(T),S).\]
Cela montre que $\mathfrak{D}_A$ est bien un foncteur faible --- comme $\mathfrak{D}_A$ prend ses valeurs dans les ensembles préordonnés, les conditions de cohérence sont automatiques.
\end{proof}

\subsection{Le foncteur $\kappa : \cc(A)\to\G_A$}\label{skap}
On rappelle que, lorsque $A$ est un groupe libre de type fini, $\cc(A)$ désigne la catégorie d'éléments $\G_{\gamma^*\scg(A,-)}$ (cf. section~\ref{sstr}).

Soit $\underline{G}:=(G,A\xrightarrow{i}G\xrightarrow{p}A)$ un objet de $\cc(A)$, où $G$ est un objet de $\G$ et $i$, $p$ sont des morphismes de groupes tels que $p\circ i={\rm Id}_A$ et qu'il existe un sous-groupe $T$ de $G$ et un diagramme commutatif
 \[\xymatrix{A\ar[r]^i\ar[rd] & G\ar[r]^p\ar[d]^\simeq & A\\
 & A*T\ar[ru] &
 }\]
 (où les flèches obliques sont l'inclusion et la projection canoniques). On note $\kappa(\underline{G})$ le sous-groupe ${\rm Ker}\,p$ de $G$, qu'on munit de l'action de $A$ donnée par $^a\!x=i(a).x.i(a)^{-1}$. Le diagramme commutatif précédent détermine un isomorphisme $\mathcal{L}_A(T)\simeq\kappa(\underline{G})$ de $A$-groupes (mais cet isomorphisme n'est pas canonique), de sorte que $\kappa(\underline{G})$ est un $A$-groupe libre de type fini.
 
 Soit $(u,T) : \underline{G}:=(G,A\xrightarrow{i}G\xrightarrow{p}A)\to\underline{H}:=(H,A\xrightarrow{j}H\xrightarrow{q}A)$ un morphisme de $\cc(A)$. Autrement dit, $(u,T) : G\to H$ est un morphisme de $\G$ tel que les diagrammes
 \[\xymatrix{A\ar[r]^i\ar[rd]_j & G\ar[d]^u\\
 & H
 }$$
 et
 $$\xymatrix{G\ar[r]^p & A\\
 H=u(G)*T\ar@{->>}[u]\ar[ur]_q &
 }\]
 commutent. Le monomorphisme de groupes $u : G\to H$ induit un monomorphisme de groupes $A$-équivariant $\bar{u} : \kappa(\underline{G})\to\kappa(\underline{H})$ ; de plus, $\kappa(\underline{H})$ est le produit libre interne de l'image de $\bar{u}$ et des $^a\!T=aTa^{-1}$ pour $a\in A$. Ainsi, $(\bar{u},T)$ définit un morphisme $\kappa(u,T) : \kappa(\underline{G})\to\kappa(\underline{H})$. On obtient de la sorte un foncteur $\kappa : \cc(A)\to\G_A$.
 
 Par conséquent, grâce à la proposition~\ref{pr-ffbl}, on dispose d'un foncteur faible $\kappa^*\mathfrak{D}_A : \cc(A)\to\mathbf{Cat}$.

\subsection{Constructions de Grothendieck}\label{sus-cg}
Soient $\C$ une catégorie et $\Phi : \C\to\mathbf{Cat}$ un foncteur faible. On rappelle que la {\em construction de Grothendieck} $\C\int\Phi$ est la catégorie dont les objets sont les couples $(c,x)$ constitués d'un objet $c$ de $\C$ et d'un objet $x$ de $\Phi(x)$, les morphismes $(c,x)\to (d,y)$ étant les couples constitués d'un morphisme $f : c\to d$ de $\C$ et d'un morphisme $u : \Phi(f)(x)\to y$ de $\Phi(d)$, la composition étant définie de la façon usuelle \cite[3.1.2]{T79}. Si $\Psi : \C^{op}\to\mathbf{Cat}$ est un foncteur faible, on notera $\Psi\int\C$ la catégorie $(\C^{op}\int\Psi)^{op}$.

On définit de manière évidente la catégorie $\Psi\int\C\int\Phi$, lorsque $\Phi : \C\to\mathbf{Cat}$ et $\Psi : \C^{op}\to\mathbf{Cat}$ sont des foncteurs faibles, catégorie dont les objets sont les triplets $(c,x,t)$ constitués d'un objet $c$ de $\C$, d'un objet $x$ de $\Phi(c)$ et d'un objet $t$ de $\Psi(x)$, qu'on peut voir indifféremment comme $\Psi\int (\C\int\Phi)$ (avec l'abus consistant à noter encore $\Psi$ la composée de ce foncteur avec le foncteur canonique $\C\int\Phi\to\C$) ou comme $(\Psi\int\C)\int\Phi$ (avec un abus similaire pour $\Phi$).

Si $\T$ est une petite catégorie, on dispose d'un foncteur (strict) $\mathbf{ens}^{op}\to\mathbf{Cat}$ associant à l'ensemble $E$ la catégorie $\T^E:=\fct(E,\T)$ (où $E$ est vu comme une catégorie discrète). Si $\T$ est munie d'un objet $t$, on a même un foncteur\,\footnote{On rappelle que les catégories $\Gamma$ et $\Pi$ qui sont utilisées abondamment dans la suite ont été introduites au début de la section~\ref{sect-rappels}.} $\Gamma^{op}\to\mathbf{Cat}$ associant à l'ensemble $E$ la catégorie $\T^E$ : si $f : X\to Y$ est une fonction partiellement définie, le foncteur $\T^Y\to\T^X$ envoie $(c_y)_{y\in Y}$ sur la famille $(c'_x)_{x\in X}$ où $c'_x=c_{f(x)}$ si $x\in {\rm Def}(f)$ et $c'_x=t$ sinon, avec l'effet évident sur les morphismes. On notera $\eta(\T,t) : \Pi^{op}\to\mathbf{Cat}$, ou simplement $\eta(\T)$ si le choix de $t$ est clair, la restriction à $\Pi^{op}$ du foncteur précédent. Les cas qui nous intéresseront seront ceux où $\T$ est un groupe (vu comme catégorie à un objet) ou un ensemble (vu comme catégorie discrète) pointé.

Si $(\C,+,0)$ est une petite catégorie monoïdale symétrique, on dispose d'un foncteur faible $\Gamma\to\mathbf{Cat}$ associant à un ensemble fini $X$ la catégorie $\C^X$ et à un morphisme $f : X\to Y$ de $\Gamma$ le foncteur $\C^X\to\C^Y$ dont la composante $\C^X\to\C$ indicée par $y\in Y$ est donnée par la composée du foncteur de projection $\C^X\twoheadrightarrow\C^{f^{-1}(y)}$ et du foncteur $\C^{f^{-1}(y)}\to\C$ obtenu par la somme (au sens de la structure monoïdale $+$) itérée (c'est la définition~4.1.2 de Thomason \cite{T79}). Si $(\C,+)$ est une catégorie monoïdale symétrique {\em sans unité}, on dispose d'une variante du foncteur précédent, notée $\varepsilon(\C,+)$, définie de la même façon, mais seulement sur $\Pi$. Autrement dit, ce foncteur associe $\C^X$ à $X$, et à la surjection partiellement définie $f : X\to Y$ le foncteur donné (sur les objets) par
$$(c_x)_{x\in X}\mapsto\Big(\underset{f(x)=y}{\sum}c_x\Big)_{y\in Y}.$$

\begin{nota}
Dans ce qui suit, $gr$ désigne le {\em groupoïde} sous-jacent à $\mathbf{gr}$ et $gr'$ désigne la sous-catégorie pleine de $gr$ des groupes libres de rang fini strictement positif. Munie du produit libre, c'est donc une catégorie monoïdale symétrique sans unité.
\end{nota}

La vérification de la propriété suivante, qui est un jeu formel d'écriture entre produits libres internes et externes, est immédiate.

\begin{pr}\label{pr-egro}
Soit $A$ un groupe libre de rang fini. Il existe une équivalence de catégories
$$\cc(A)\int\kappa^*\mathfrak{D}_A\simeq\varepsilon(gr',*)\int\Pi^{op}\int\eta(A).$$
Plus précisément, les deux foncteurs définis ci-après sont des équivalences quasi-inverses l'une de l'autre.
 \begin{itemize}
  \item Le foncteur $\cc(A)\int\kappa^*\mathfrak{D}_A\to\varepsilon(gr',*)\int\Pi^{op}\int\eta(A)$ associe à un objet $(\underline{G},H_1,\dots,H_r)$ l'objet $(\mathbf{r},H_1,\dots,H_r)$ ;
  \item si $(u,T) : \underline{G}\to\underline{G}'$ est un morphisme de $\cc(A)$ et
  $$(\varphi;a_1,\dots,a_s) : (H_1,\dots,H_r,T)\to (K_1,\dots,K_s)\qquad\text{(pour }T\neq\{1\})$$
  $$\text{ou}\qquad (\varphi;a_1,\dots,a_s) : (H_1,\dots,H_r)\to (K_1,\dots,K_s)\qquad\text{(pour }T=\{1\})$$
un morphisme de $\mathfrak{D}_A(\kappa(\underline{G}'))$, où $\varphi : \mathbf{s}\to\mathbf{r+1}$ (ou $\varphi : \mathbf{s}\to\mathbf{r}$ si $T=\{1\}$) est un morphisme de $\Omega$ et les $a_j$ sont des éléments de $A$, de sorte que $u(H_i)=\underset{\varphi(j)=i}{\bigstar}^{a_j}\!K_j$, et $T=\underset{\varphi(j)=r+1}{\bigstar}^{a_j}\!K_j$ si $T\neq\{1\}$, le morphisme $(\mathbf{s},K_1,\dots,K_s)\to (\mathbf{r},H_1,\dots,H_r)$ associé de $\varepsilon(gr',*)\int\Pi^{op}\int\eta(A)$ est donné par le morphisme $\mathbf{s}\to\mathbf{r}$ de $\Pi$ défini sur le complémentaire de la préimage de $r+1$ par $\varphi$ et coïncidant dessus avec $\varphi$  (pour $T\neq\{1\}$ ; c'est simplement $\varphi$ si $T=\{1\}$), la famille $(a_j)_{1\leq j\leq s}$ d'éléments de $A$ et la collection d'isomorphismes
  $$H_i\simeq u(H_i)=\underset{\varphi(j)=i}{\bigstar}^{a_j}\!K_j\simeq\underset{\varphi(j)=i}{\bigstar}K_j$$
  (le premier produit libre est interne, le second externe) et, si $T\neq\{1\}$, 
  $$T=\underset{\varphi(j)=r+1}{\bigstar}^{a_j}\!K_j\simeq\underset{\varphi(j)=r+1}{\bigstar}K_j\;;$$
  \item le foncteur $\varepsilon(gr',*)\int\Pi^{op}\int\eta(A)\to\cc(A)\int\kappa^*\mathfrak{D}_A$ associe à un objet\linebreak[4] $(H_1,\dots,H_n)$ (où les $H_i$ sont des groupes libres de rangs finis non nuls) l'objet $\underline{G}=(G,A\to A*G\to A)$ (avec les morphismes canoniques) de $\cc(A)$, où $G$ est le produit libre des $H_i$, muni de la famille de sous-groupes $(H_1,\dots,H_n)$ (qui appartient à $\mathfrak{D}_A(\kappa(\underline{G})))\simeq\mathfrak{D}_A(\mathcal{L}_A(G))$) ;
  \item à un morphisme $(\varphi;a_1,\dots,a_m;\psi_1,\dots,\psi_n) : (H_1,\dots,H_n)\to (K_1,\dots,K_m)$ de $\varepsilon(gr',*)\int\Pi^{op}\int\eta(A)$, où $\varphi : \mathbf{m}\to\mathbf{n}$ est un morphisme de $\Pi$, les $a_j$ sont des éléments de $A$ et les $\psi_i$ des isomorphismes de groupes $H_i\simeq\underset{\varphi(j)=i}{\bigstar}K_j$, on associe le morphisme
  $$(\underset{i\in\mathbf{n}}{\bigstar}H_i,A\to A*\underset{i\in\mathbf{n}}{\bigstar}H_i\to A,(H_i))\to (\underset{j\in\mathbf{m}}{\bigstar}K_j,A\to A*\underset{j\in\mathbf{m}}{\bigstar}K_j\to A, (K_j))$$
  de $\cc(A)\int\kappa^*\mathfrak{D}_A$ donné par le monomorphisme de groupes
  $$A*\underset{i\in\mathbf{n}}{\bigstar}H_i\simeq A*\underset{i\in\mathbf{n}}{\bigstar}\,\underset{\varphi(j)=i}{\bigstar}K_j\simeq A*\underset{j\in {\rm Def}(\varphi)}{\bigstar}K_j\simeq A*\underset{j\in {\rm Def}(\varphi)}{\bigstar}\,^{a_j}\!K_j\hookrightarrow A*\underset{j\in\mathbf{m}}{\bigstar}K_j$$
  (le premier isomorphisme est induit par les $\psi_i$, le second par la structure monoïdale symétrique du produit libre, le troisième par la conjugaison partielle sur les $K_j$ par les $a_j$), le sous-groupe $\underset{j\in\mathbf{m}\setminus {\rm Def}(\varphi)}{\bigstar}\,^{a_j}\!K_j$ de $A*\underset{j\in\mathbf{m}}{\bigstar}K_j$, et le morphisme de $\mathfrak{D}_A(\kappa(\underset{j\in\mathbf{m}}{\bigstar}K_j,A\to A*\underset{j\in\mathbf{m}}{\bigstar}K_j\to A))\simeq\mathfrak{D}_A(\mathcal{L}_A(\underset{j\in\mathbf{m}}{\bigstar}K_j))$ dont le morphisme de $\Omega$ sous-jacent est $\varphi$ si cette fonction est partout définie, et sinon la fonction surjective partout définie $f : \mathbf{m}\to\mathbf{n+1}$ associée (égale à $\varphi$ sur ${\rm Def}(\varphi)$ et à $n+1$ ailleurs), et donnée par les éléments $a_j$ de $A$. (C'est licite parce que $H_i\subset A*\underset{t\in\mathbf{n}}{\bigstar}H_t$ est envoyé par notre morphisme de groupes sur le sous-groupe $\underset{f(j)=i}{\bigstar}\,^{a_j}\!K_j$ de $A*\underset{j\in\mathbf{m}}{\bigstar}K_j$ et que $\underset{j\in\mathbf{m}\setminus {\rm Def}(\varphi)}{\bigstar}\,^{a_j}\!K_j$ (lorsqu'il est non trivial) est égal à $\underset{f(j)=n+1}{\bigstar}\,^{a_j}\!K_j$.)
 \end{itemize}
\end{pr}

\subsection{L'espace de lacets infinis $\N(\cc(A))$}\label{seli}
Pour démontrer le théorème~\ref{th-hs}, on utilise la construction classique de Segal \cite{Seg}, en suivant également la présentation de Bousfield-Friedlander \cite{BF}. Rappelons ces constructions (attention, nos notations ne concordent pas exactement avec celles des références précédentes). Un {\em $\Gamma$-espace} est un foncteur $F$ de $\Gamma$ vers la catégorie $\fct(\ds^{op},\mathbf{Ens}_\bullet)$ des ensembles simpliciaux pointés qui préserve les objets nuls ; il est dit {\em spécial} si l'application canonique $F(X\sqcup Y)\to F(X)\times F(Y)$ est une équivalence faible pour tous ensembles finis $X$ et $Y$. Un $\Gamma$-espace $F$ détermine un spectre connectif $Sp(F)$ (en fait, la catégorie des $\Gamma$-espaces peut être prise comme modèle des spectres connectifs) de façon fonctorielle.

Un $\Gamma$-espace $F$ se prolonge en un foncteur $\bar{F} : \mathbf{Ens}_\bullet\to\fct(\ds^{op},\mathbf{Ens}_\bullet)$ par extension de Kan à gauche le long du foncteur $\Gamma\to\mathbf{Ens}_\bullet$ composé de l'équivalence de $\Gamma$ avec la catégorie $\mathbf{ens}_\bullet$ des ensembles finis pointés (qui sur les objets s'obtient par ajout d'un point de base externe, et sur les morphismes envoie une fonction partiellement définie sur l'application coïncidant avec elle là où elle est définie et envoyant sur le point de base les éléments sur lesquels la fonction de départ n'était pas définie) et de l'inclusion $\mathbf{ens}_\bullet\to\mathbf{Ens}_\bullet$. On peut ensuite prolonger $\bar{F}$ en un endofoncteur de $\fct(\ds^{op},\mathbf{Ens}_\bullet)$ en considérant
$$\fct(\ds^{op},\mathbf{Ens}_\bullet)\xrightarrow{\fct(\ds^{op},\bar{F})}\fct(\ds^{op},\fct(\ds^{op},\mathbf{Ens}_\bullet))\simeq\fct(\ds^{op}\times\ds^{op},\mathbf{Ens}_\bullet)$$
puis en précomposant par la diagonale $\ds^{op}\to\ds^{op}\times\ds^{op}$. Ce prolongement sera encore noté $\bar{F}$ ; il est muni d'applications structurales canoniques $X\wedge\bar{F}(Y)\to\bar{F}(X\wedge Y)$. Avec le modèle classique des spectres \cite{BF}, le spectre $Sp(F)$ peut se voir comme la suite $(\bar{F}(\mathbb{S}^n))_n$ (où l'on se fixe un modèle simplicial pointé $\mathbb{S}^1$ du cercle et $\mathbb{S}^n:=(\mathbb{S}^1)^{\wedge n}$) munie des morphismes déduits des applications structurales précédentes. Pour tout ensemble simplicial pointé $X$, on dispose d'un morphisme canonique de spectres $Sp(F)\wedge X\to\Sigma^\infty(\bar{F}(X))$ qui est une équivalence faible \cite[{\em Lemma}~4.1]{BF}.

Soit $(\C,+,0)$ une catégorie monoïdale symétrique ; comme nous n'aurons besoin que des isomorphismes dans cette catégorie, nous supposerons d'emblée que c'est un groupoïde. Nous supposerons également que $\C$ est {\em régulière} au sens où $a+b$ ne peut être isomorphe à $0$ que si c'est le cas de $a$ et $b$ et où $0$ n'a pas d'endomorphisme non trivial. Nous noterons $\C'$ la sous-catégorie pleine des objets non isomorphes à l'unité $0$ (qui monoïdale sans unité par régularité ; cette notation concorde avec celles employées précédemment, $gr$ et $gr'$, pour les groupes libres).

Soit $X$ un ensemble pointé : on a défini au §\,\ref{sus-cg} un foncteur $\eta(X) : \Pi^{op}\to\mathbf{Cat}$ (dont les valeurs $\eta(X)(E)=X^E$ sont discrètes) ; on remarque que $X\mapsto\eta(X)$ définit même un foncteur $\mathbf{Ens}_\bullet\to\fct(\Pi^{op},\mathbf{Cat})$, avec un effet évident sur les morphismes.

On définit un foncteur $\mathfrak{C}(\C,-) : \mathbf{Ens}_\bullet\to\mathbf{Cat}$ par
\[\mathfrak{C}(\C,X):=\varepsilon(\C',+)\int\Pi^{op}\int\eta(X)\]
sur les objets, l'effet sur les morphismes provenant de la fonctorialité de $\eta$.

Soit $X$ un ensemble pointé fini, notons $x_0$ son point de base et posons $X_-:=X\setminus\{x_0\}$. On dispose d'un foncteur $F_X : \C^{X_-}\to\mathfrak{C}(\C,X)$ associant à une famille $(c_x)_{x\in X_-}$ d'objets de $\C$ la famille finie $(c_x,x)_{x\in E}$, où $E$ est l'ensemble des $x$ tels que $c_x$ ne soit pas isomorphe à $0$ (avec l'effet évident sur les morphismes --- on rappelle que $\C$ ne contient que des isomorphismes). On dispose également d'un foncteur $G_X : \mathfrak{C}(\C,X)\to\C^{X_-}$ associant à une famille finie $(c_t,a_t)_{t\in E}$ de ${\rm Ob}\,\C'\times X$ l'objet $\Big(\underset{a_t=x}{\sum}c_t\Big)_{x\in X_-}$ de $\C^{X_-}$ (avec encore un effet clair sur les morphismes). 

\begin{lm}\label{lmac} Le foncteur $G_X$ est adjoint à droite à $F_X$.
\end{lm}

\begin{proof} Soient $\mathrm{d}:=(d_x)_{x\in X_-}$ un objet de $\C^{X_-}$ et $\mathrm{c}:=(c_t,a_t)_{t\in E}$ (où $E$ est un ensemble fini, les $a_t$ sont des éléments de $X$ et les $c_t$ des objets de $\C'$) un objet de $\mathfrak{C}(\C,X)$. On définit une fonction
\[\C^{X_-}(\mathrm{d},G_X(\mathrm{c}))\to\mathfrak{C}(\C,X)(F_X(\mathrm{d}),\mathrm{c})\]
comme suit. Un élément $f$ de $\C^{X_-}(\mathrm{d},G_X(\mathrm{c}))$ est une collection de morphismes $\alpha_x : d_x\to\underset{a_t=x}{\sum}c_t$ de $\C$ pour $x\in X_-$ ; on lui associe une fonction $E\to X$ qui envoie $t$ sur $a_t$. Celle-ci se restreint en une fonction {\em partiellement définie} $\varphi$ de $E$ {\em sur} l'ensemble $Y$ des $x$ de $X_-$ tels que $d_x$ n'est pas isomorphe à $0$. La surjectivité provient de ce que $d_x\simeq\underset{a_t=x}{\sum}c_t$, puisque toutes les flèches de $\C$ sont des isomorphismes.

Le morphisme $\varphi\in\Pi(E,Y)$ et le morphisme
\[(d_x)_{x\in Y}\to\varepsilon(\varphi)\big((c_t)_{t\in E}\big)=\Big(\sum_{\varphi(t)=x} c_t\Big)_{x\in Y}=\Big(\sum_{a_t=x}c_t\Big)_{x\in Y}\]
 de $\varepsilon(Y)=\C'^Y$ dont les composantes sont les $\alpha_x$ forment un morphisme $F_X(\mathrm{d})\to\mathrm{c}$ dans $\mathfrak{C}(\C,X)$. En effet, la condition de compatibilité à $\eta(X)$ s'écrit $a_t=\varphi(t)$ si $\varphi(t)$ est défini et $a_t=x_0$ sinon ; elle résulte de la définition de $\varphi$ et de ce que $\varphi(t)$ est toujours défini sauf si $a_t=x_0$ ou si $d_{a_t}$ est isomorphe à $0$, mais cette dernière condition est impossible puisque $\C$ est supposé régulier et que $d_{a_t}\simeq\underset{a_u=a_t}{\sum}c_u$ est une somme non vide d'éléments de $\C$ non isomorphes à $0$.
 
 Le fait que la fonction $\C^{X_-}(\mathrm{d},G_X(\mathrm{c}))\to\mathfrak{C}(\C,X)(F_X(\mathrm{d}),\mathrm{c})$ qu'on vient de définir est une bijection résulte de ce que les morphismes $\alpha_x$ sont uniquement déterminés pour $x\notin Y$, puisqu'on a fait l'hypothèse de régularité que $0$ n'a pas d'endomorphisme non trivial. Sa naturalité en $\mathrm{c}$ et en $\mathrm{d}$ est immédiate, d'où le lemme.
\end{proof}

En particulier, $F_X$ et $G_X$ induisent des équivalences d'homotopie mutuellement quasi-inverses entre les ensembles simpliciaux $\N(\C^{X_-})$ et $\N(\mathfrak{C}(\C,X))$. De plus, ces foncteurs respectent les points de base canoniques de ces catégories (la famille constante en l'objet unité $0$ dans $\C^{X_-}$, et la famille vide dans $\mathfrak{C}(\C,X)$), de sorte qu'il s'agit d'équivalences d'homotopie d'ensembles simpliciaux {\em pointés}. En fait, le foncteur $G_X : \mathfrak{C}(\C,X)\to\C^{X_-}$ est {\em monoïdal} (au sens fort), où la catégorie source est munie de la structure monoïdale symétrique donnée par la concaténation des objets et la catégorie but de la structure monoïdale symétrique produit induite par celle de $\C$. De surcroît, si $Y$ est un autre ensemble pointé fini (on notera de même $Y_-$ le complémentaire du point de base de $Y$), le diagramme
$$\xymatrix{\C^{X_-\sqcup Y_-}=\C^{(X\vee Y)_-}\ar[d]_\simeq\ar[r]^-{F_{X\vee Y}} & \mathfrak{C}(\C,X\vee Y)\ar[d] \\
\C^{X_-}\times\C^{Y_-}\ar[r]^-{F_X\times F_Y} & \mathfrak{C}(\C,X)\times\mathfrak{C}(\C,Y)
}$$
(dont les flèches verticales sont les foncteurs canoniques ; $\vee$ désigne la somme catégorie de $\mathbf{Ens}_\bullet$) commute à isomorphisme près. On en déduit que le $\Gamma$-espace obtenu en composant le foncteur canonique $\Gamma\to\mathbf{Ens}_\bullet$, le foncteur $\mathfrak{C}(\C,-) : \mathbf{Ens}_\bullet\to\mathbf{Cat}$ et le foncteur nerf (dont on a vu qu'il se factorisait naturellement par les ensembles simpliciaux {\em pointés}) est spécial et a le même type d'homotopie que le $\Gamma$-espace associé à la petite catégorie monoïdale symétrique $(\C,+,0)$ par la machinerie de Segal \cite[§\,2]{Seg} (voir aussi \cite[§\,4]{T79}), que nous noterons $\sigma(\C,+,0)$, ou simplement $\sigma(\C)$.

\begin{pr}\label{pr-ehpp} Il existe une équivalence d'homotopie d'ensembles simpliciaux pointés
\[\N\left(\varepsilon(\C',+)\int\Pi^{op}\int\eta(\T,t)\right)\simeq\overline{\sigma(\C)}\big(\N(\T)\big)\]
naturelle en la petite catégorie $\T$ munie d'un objet $t$.
\end{pr}

\begin{proof} Lorsque la catégorie $\T$ est {\em discrète} et finie, le résultat découle du lemme~\ref{lmac} et de la discussion précédente. Il s'étend aussitôt au cas où $\T$ est discrète et infinie, en écrivant une telle catégorie comme la colimite filtrante de ses sous-catégories finies.

Le cas général s'en déduit par un argument formel de colimite homotopique : si $X$ est un ensemble simplicial (pointé), on dispose d'une équivalence d'homotopie naturelle
\[\N(X)\simeq\underset{\mathbf{n}\in\ds^{op}}{\mathrm{hocolim}}\,X_n\]
où $X_n$ est vu comme comme ensemble (pointé) discret, et si $F$ est un $\Gamma$-espace, le prolongement $\bar{F}$ de $F$ aux ensembles simpliciaux pointés donne lieu par construction à une équivalence d'homotopie naturelle d'ensembles simpliciaux pointés
\[\bar{F}(X)\simeq\underset{\mathbf{n}\in\ds^{op}}{\mathrm{hocolim}}\,\bar{F}(X_n).\]

Comme le théorème de la colimite homotopique \cite[{\em Theorem}~1.2]{T79} fournit une équivalence d'homotopie naturelle
\[\N\left(\varepsilon(\C',+)\int\Pi^{op}\int\eta(\T,t)\right)\simeq\underset{(c_e)_{e\in E}\in\varepsilon(\C',+)\int\Pi^{op}}{\mathrm{hocolim}}\,\N\big(\eta(\T,t)(E)\big),\]
les observations précédentes permettent d'étendre l'équivalence naturelle obtenue pour une petite catégorie discrète $\T$ à une petite catégorie quelconque $\T$.
\end{proof}

En utilisant \cite[{\em Theorem}~4.4]{BF}, on en déduit :
\begin{cor}\label{cor-ehpp} Si $\T$ est une petite catégorie connexe munie d'un objet $t$, il existe une équivalence d'homotopie naturelle d'ensembles simpliciaux pointés
\[\N\left(\varepsilon(\C',+)\int\Pi^{op}\int\eta(\T,t)\right)\simeq\Omega^\infty\big(Sp(\sigma(\C))\wedge\N(\T)\big).\]
\end{cor}

Le lemme classique suivant se déduit du théorème~A de Quillen \cite{QK} ; c'est également un cas particulier de la forme générale (i.e. pour des foncteurs {\em faibles}) du théorème de la colimite homotopique de Thomason \cite[{\em Theorem}~3.3]{T79}.

\begin{lm}[Quillen, Thomason]\label{lmtq} Soient $\D$ une petite catégorie et $\Phi : \D\to\mathbf{Cat}$ un foncteur faible tel que, pour tout objet $d$ de $\D$, le nerf de la catégorie $\Phi(d)$ est contractile. Alors le foncteur d'oubli $\D\int\Phi\to\D$ induit une équivalence d'homotopie entre les nerfs de ces catégories.
\end{lm}

\begin{pr}\label{pr-thprpr} Soit $A$ un groupe libre de rang fini. Il existe une équivalence d'homotopie d'ensembles simpliciaux pointés
\[\N(\cc(A))\simeq\Omega^\infty\Sigma^\infty\big(\bb(A)\big)\]
qui est équivariante par rapport aux actions tautologiques du groupe ${\rm Aut}(A)$.
\end{pr}

\begin{proof} Le théorème~\ref{th-ccon} et le lemme~\ref{lmtq} montrent que le foncteur canonique $\cc(A)\int\kappa^*\mathfrak{D}_A\to\cc(A)$, qui est ${\rm Aut}(A)$-équivariant, induit une équivalence d'homotopie d'ensembles simpliciaux pointés (toutes les flèches préservent les points de base canoniques)
\[\N(\cc(A))\simeq\N\Big(\cc(A)\int\kappa^*\mathfrak{D}_A\Big).\]

La proposition~\ref{pr-egro} fournit par ailleurs une équivalence d'homotopie
\[\N\Big(\cc(A)\int\kappa^*\mathfrak{D}_A\Big)\simeq\N\Big(\varepsilon(gr',*)\int\Pi^{op}\int\eta(A)\Big)\]
qui est manifestement pointée et ${\rm Aut}(A)$-équivariante.

On dispose enfin d'équivalence d'homotopie pointées
\[\N\Big(\varepsilon(gr',*)\int\Pi^{op}\int\eta(A)\Big)\simeq\Omega^\infty\big(Sp(\sigma(gr))\wedge\bb(A)\big)\simeq\Omega^\infty\Sigma^\infty\big(\bb(A)\big)\]
grâce au corollaire~\ref{cor-ehpp} et au théorème de Galatius \cite{Gal} montrant que le foncteur monoïdal canonique $\Sigma\to gr$ (où $\Sigma$ désigne le groupoïde sous-jacent à la catégorie monoïdale symétrique $\mathbf{ens}$) induit une équivalence $Sp(\sigma(\Sigma))\to Sp(\sigma(gr))$, le spectre source étant, classiquement \cite[proposition~3.5]{Seg}, équivalent au spectre des sphères. La naturalité en $\T$ dans le corollaire~\ref{cor-ehpp} montre l'équivariance requise, d'où la proposition.
\end{proof}

Pour terminer la démonstration du théorème~\ref{th-hs}, il reste à montrer que l'équivalence d'homotopie de la proposition~\ref{pr-thprpr} est fonctorielle en $A$, non seulement par rapport aux isomorphismes (ce qu'indique la propriété d'équivariance), mais aussi relativement à toutes les flèches de $\scg$, ce qui nécessite un petit plus de travail et sera effectué dans la section~\ref{ssmfc}.

\begin{rem}\label{rq-li}
 L'équivalence de la proposition~\ref{pr-thprpr} est une équivalence d'espaces de lacets infinis. En effet, la catégorie $\cc(A)$ est munie d'une structure monoïdale symétrique donnée par la somme (ou produit libre) amalgamée au-dessus de $A$. Cette structure se relève à $\cc(A)\int\kappa^*\mathfrak{D}_A$ (en utilisant la concaténation des suites de groupes libres), ce qui fait du foncteur canonique $\cc(A)\int\kappa^*\mathfrak{D}_A\to\cc(A)$ un foncteur monoïdal (au sens fort), de sorte que l'équivalence d'homotopie qu'il induit entre les nerfs de ces catégories est une équivalence d'espaces de lacets infinis. Il en est de même pour les autres équivalences d'homotopie utilisées dans la démonstration de la proposition~\ref{pr-thprpr}. Cela provient de ce que l'équivalence de catégories de la proposition~\ref{pr-egro} est monoïdale (où $\varepsilon(gr',*)\int\Pi^{op}\int\eta(A)$ est munie de la structure monoïdale donnée par la concaténation des suites de groupes libres) et, pour les dernières utilisées, c'est une conséquence formelle du travail de Segal \cite{Seg}.
 
 En revanche, si $A\to B$ est un morphisme de $\scg$, le foncteur $\cc(B)\to\cc(A)$ qu'il induit {\em n'}est {\em pas} monoïdal, de sorte qu'il n'est pas clair {\em a priori} que l'application $\N(\cc(B))\to\N(\cc(A))$ est un morphisme d'espaces de lacets infinis. C'est en fait le cas, comme nous allons le voir.
 
 Ces structures d'espaces de lacets infinis ne sont pas nécessaires à l'établissement des résultats du présent article, mais elles procèdent des mêmes idées que celles que nous allons maintenant mettre en \oe uvre pour établir la fonctorialité de l'équivalence du théorème~\ref{th-hs}.
\end{rem}

\subsection{Structure monoïdale du foncteur $\cc : \scg^{op}\to\mathbf{Cat}$ et fin de la démonstration du théorème~\ref{th-hs}}\label{ssmfc}

On rappelle que, si $A$ est un groupe libre de rang fini, $\cc(A)$ est la catégorie d'éléments $\G_{\gamma^*\scg(A,-)}$ ; la fonctorialité du foncteur Hom fait de $\cc$ un foncteur $\scg^{op}\to\mathbf{Cat}$. 

Nous établirons la fonctorialité de l'équivalence du théorème~\ref{th-hs} (construite dans la section~\ref{seli}) à l'aide du lemme immédiat qui suit.

\begin{lm}\label{lmfi} La catégorie $\scg$ est engendrée par les automorphismes et les morphismes canoniques $A\to A*B$ déduits de ce que l'unité de la structure monoïdale symétrique $*$ de $\scg$ en est un objet initial.

Par conséquent, si $\D$ est une catégorie quelconque et que $X, Y : \scg\to\D$ sont des foncteurs, toute collection de morphismes $f_A : X(A)\to Y(A)$ de $\D$ équivariants pour l'action de ${\rm Aut}(A)$, pour $A\in {\rm Ob}\,\scg$, définit une transformation naturelle de $X$ vers $Y$ si le diagramme
\[\xymatrix{X(A)\ar[d]_{f_A}\ar[r] & X(A*B)\ar[d]^{f_{A*B}}\\
Y(A)\ar[r] & Y(A*B)
}\]
dont les flèches horizontales sont induites par le morphisme canonique commute pour tous objets $A$ et $B$ de $\scg$.
\end{lm}

On dispose par ailleurs d'une structure monoïdale symétrique sur le foncteur $\cc$ obtenue comme suit.
Soient $A$ et $B$ deux objets de $\scg$. Le produit libre induit un foncteur $\cc(A)\times\cc(B)\to\cc(A*B)$ naturel en $A$ et $B$ et vérifiant les conditions de cohérence monoïdales symétriques usuelles. De plus, la composée de ce foncteur avec le foncteur $\cc(A*B)\to\cc(A)\times\cc(B)$ induit par les morphismes canoniques $A\to A*B$ et $B\to A*B$ de $\scg$ est homotope à l'identité. En effet, il existe une transformation naturelle de l'identité de $\cc(A)\times\cc(B)$ vers cette composée, donnée sur l'objet $(\underline{G},\underline{H})$, où $\underline{G}=(G,A\to G\to A)$ et $\underline{H}=(H,B\to H\to B)$, par les morphismes canoniques $G\to G*H$ et $H\to G*H$ de $\G$.

\begin{pr}\label{mpr} Soient $A$ et $B$ des objets de $\scg$. Le diagramme suivant de la catégorie homotopique des ensembles simpliciaux pointés commute
\begin{equation}\label{dcmon}
\xymatrix{\N\cc(A)\times\N\cc(B)\ar[r]\ar[d]^\simeq & \N\cc(A*B)\ar[r]\ar[d]^\simeq & \N\cc(A)\times\N\cc(B)\ar[d]^\simeq \\
Q\bb(A)\times Q\bb(B)\ar[r]^-\simeq & Q\bb(A*B)\ar[r]^-\simeq & Q\bb(A)\times Q\bb(B)
}
\end{equation}
où l'on note $Q$ le foncteur $\Omega^\infty\Sigma^\infty$, les isomorphismes verticaux sont ceux de la proposition~\ref{pr-thprpr}, les flèches horizontales inférieures sont les isomorphismes canoniques et les flèches horizontales supérieures sont obtenues en appliquant le foncteur $\N$ (qui commute aux produits) aux foncteurs qu'on vient de discuter entre $\cc(A)\times\cc(B)$ et $\cc(A*B)$.
\end{pr}

\begin{proof} Les composées horizontales supérieure et inférieure de ce diagramme sont les identités. Il suffit donc d'établir que le carré de gauche est commutatif. Cela se voit en reprenant la définition de l'équivalence de la proposition~\ref{pr-thprpr} et en notant que chacune des équivalences qui la composent est monoïdale.

En effet, la structure monoïdale $\cc(A)\times\cc(B)\to\cc(A*B)$ sur $\cc$ se relève (en utilisant la concaténation des suites de groupes libres) en une structure monoïdale $\cc(A)\int\kappa^*\mathfrak{D}_A\times\cc(B)\int\kappa^*\mathfrak{D}_B\to\cc(A*B)\int\kappa^*\mathfrak{D}_{A*B}$, de sorte que le diagramme
 \[\xymatrix{\cc(A)\int\kappa^*\mathfrak{D}_A\times\cc(B)\int\kappa^*\mathfrak{D}_B\ar[r]\ar[d] & \cc(A*B)\int\kappa^*\mathfrak{D}_{A*B}\ar[d] \\
 \cc(A)\times\cc(B)\ar[r] & \cc(A*B)
 }\]
dont les flèches verticales sont les foncteurs canoniques commute à isomorphisme canonique près.

De plus, on peut également former un diagramme commutatif (à isomorphisme canonique près)
  \[\xymatrix{\cc(A)\int\kappa^*\mathfrak{D}_A\times\cc(B)\int\kappa^*\mathfrak{D}_B\ar[r]\ar[d]_\simeq & \cc(A*B)\int\kappa^*\mathfrak{D}_{A*B}\ar[d]^\simeq \\
 \Big(\varepsilon(gr',*)\int\Pi^{op}\int\eta(A)\Big)\times\Big(\varepsilon(gr',*)\int\Pi^{op}\int\eta(B)\Big)\ar[r] & \varepsilon(gr',*)\int\Pi^{op}\int\eta(A*B)
 }\]
 où les flèches verticales sont les équivalences de la proposition~\ref{pr-egro} et la structure monoïdale donnée par la flèche horizontale inférieure est toujours donnée par la concaténation des suites de groupes libres.
 
Quant aux équivalences d'homotopie
\[\N\Big(\varepsilon(gr',*)\int\Pi^{op}\int\eta(A)\Big)\simeq\overline{\sigma(gr)}(\eta(A))\simeq Q\bb(A),\]
données par le corollaire~\ref{cor-ehpp} et \cite{Gal}, elles sont monoïdales en $A$, car $\sigma(gr)$ est un $\Gamma$-espace spécial. Cela termine la démonstration.
\end{proof}

\begin{proof}[Fin de la démonstration du théorème~\ref{th-hs}] Le théorème s'obtient en combinant la proposition~\ref{pr-thprpr}, le lemme~\ref{lmfi} et la commutation du carré de droite du diagramme \eqref{dcmon} de la proposition~\ref{mpr}.
\end{proof}

\bibliographystyle{plain}

\bibliography{bibli-hodge.bib}

\begin{thebibliography}{10}

\bibitem{BEc}
M.~G. Barratt and Peter~J. Eccles.
\newblock {$\Gamma ^{+}$}-structures. {III}. {T}he stable structure of {$\Omega
  ^{\infty }\Sigma ^{\infty }A$}.
\newblock {\em Topology}, 13:199--207, 1974.

\bibitem{BF}
A.~K. Bousfield and E.~M. Friedlander.
\newblock Homotopy theory of {$\Gamma $}-spaces, spectra, and bisimplicial
  sets.
\newblock In {\em Geometric applications of homotopy theory ({P}roc. {C}onf.,
  {E}vanston, {I}ll., 1977), {II}}, volume 658 of {\em Lecture Notes in Math.},
  pages 80--130. Springer, Berlin, 1978.

\bibitem{CMT}
F.~R. Cohen, J.~P. May, and L.~R. Taylor.
\newblock Splitting of certain spaces {$CX$}.
\newblock {\em Math. Proc. Cambridge Philos. Soc.}, 84(3):465--496, 1978.

\bibitem{Dja-JKT}
Aur\'elien Djament.
\newblock Sur l'homologie des groupes unitaires à coefficients polynomiaux.
\newblock {\em J. K-Theory}, 10(1):87--139, 2012.

\bibitem{DPV}
Aur\'elien {Djament}, Teimuraz {Pirashvili}, and Christine {Vespa}.
\newblock {Cohomologie des foncteurs polynomiaux sur les groupes libres.}
\newblock {\em {Doc. Math., J. DMV}}, 21:205--222, 2016.

\bibitem{DV}
Aur{\'e}lien Djament and Christine Vespa.
\newblock Sur l'homologie des groupes orthogonaux et symplectiques \`a
  coefficients tordus.
\newblock {\em Ann. Sci. \'Ec. Norm. Sup\'er. (4)}, 43(3):395--459, 2010.

\bibitem{DV15}
Aur{\'e}lien Djament and Christine Vespa.
\newblock Sur l'homologie des groupes d'automorphismes des groupes libres \`a
  coefficients polynomiaux.
\newblock {\em Comment. Math. Helv.}, 90(1):33--58, 2015.

\bibitem{Dold}
Albrecht Dold.
\newblock Zur {H}omotopietheorie der {K}ettenkomplexe.
\newblock {\em Math. Ann.}, 140:278--298, 1960.

\bibitem{EML}
Samuel Eilenberg and Saunders Mac~Lane.
\newblock On the groups {$H(\Pi,n)$}. {II}. {M}ethods of computation.
\newblock {\em Ann. of Math. (2)}, 60:49--139, 1954.

\bibitem{FPs}
Vincent Franjou and Teimuraz Pirashvili.
\newblock Stable {$K$}-theory is bifunctor homology (after {A}. {S}corichenko).
\newblock In {\em Rational representations, the {S}teenrod algebra and functor
  homology}, volume~16 of {\em Panor. Synth\`eses}, pages 107--126. Soc. Math.
  France, Paris, 2003.

\bibitem{Gal}
S{\o}ren Galatius.
\newblock Stable homology of automorphism groups of free groups.
\newblock {\em Ann. of Math. (2)}, 173(2):705--768, 2011.

\bibitem{HPV}
Manfred Hartl, Teimuraz Pirashvili, and Christine Vespa.
\newblock Polynomial functors from algebras over a set-operad and nonlinear
  {M}ackey functors.
\newblock {\em Int. Math. Res. Not. IMRN}, (6):1461--1554, 2015.

\bibitem{HVW}
Allan Hatcher, Karen Vogtmann, and Natalie Wahl.
\newblock Erratum to: ``{H}omology stability for outer automorphism groups of
  free groups [{A}lgebr. {G}eom. {T}opol. {\bf 4} (2004), 1253--1272
  (electronic)] by {H}atcher and {V}ogtmann.
\newblock {\em Algebr. Geom. Topol.}, 6:573--579 (electronic), 2006.

\bibitem{HV04}
Allen Hatcher and Karen Vogtmann.
\newblock Homology stability for outer automorphism groups of free groups.
\newblock {\em Algebr. Geom. Topol.}, 4:1253--1272, 2004.

\bibitem{HW}
Allen Hatcher and Nathalie Wahl.
\newblock Stabilization for the automorphisms of free groups with boundaries.
\newblock {\em Geom. Topol.}, 9:1295--1336 (electronic), 2005.

\bibitem{HW-erra}
Allen Hatcher and Nathalie Wahl.
\newblock Erratum to: ``{S}tabilization for the automorphisms of free groups
  with boundaries'' [{G}eom. {T}opol. {\bf 9} (2005), 1295--1336; 2174267].
\newblock {\em Geom. Topol.}, 12(2):639--641, 2008.

\bibitem{Hoa}
A.~H.~M. Hoare.
\newblock On length functions and {N}ielsen methods in free groups.
\newblock {\em J. London Math. Soc. (2)}, 14(1):188--192, 1976.

\bibitem{PJ}
Mamuka Jibladze and Teimuraz Pirashvili.
\newblock Cohomology of algebraic theories.
\newblock {\em J. Algebra}, 137(2):253--296, 1991.

\bibitem{Kahn}
Daniel~S. Kahn.
\newblock On the stable decomposition of {$\Omega ^{\infty }S^{\infty }A$}.
\newblock In {\em Geometric applications of homotopy theory ({P}roc. {C}onf.,
  {E}vanston, {I}ll., 1977), {II}}, volume 658 of {\em Lecture Notes in Math.},
  pages 206--214. Springer, Berlin, 1978.

\bibitem{K-Magnus}
Nariya Kawazumi.
\newblock Cohomological aspects of {M}agnus expansions.
\newblock arXiv : math.GT/0505497, 2006.

\bibitem{K-mm}
Nariya Kawazumi.
\newblock Twisted {M}orita-{M}umford classes on braid groups.
\newblock In {\em Groups, homotopy and configuration spaces}, volume~13 of {\em
  Geom. Topol. Monogr.}, pages 293--306. Geom. Topol. Publ., Coventry, 2008.

\bibitem{Lli}
Jean-Louis Loday.
\newblock {\em Cyclic homology}, volume 301 of {\em Grundlehren der
  Mathematischen Wissenschaften [Fundamental Principles of Mathematical
  Sciences]}.
\newblock Springer-Verlag, Berlin, second edition, 1998.
\newblock Appendix E by Mar\'{i}a O. Ronco, Chapter 13 by the author in
  collaboration with Teimuraz Pirashvili.

\bibitem{Lyn}
Roger~C. Lyndon.
\newblock Length functions in groups.
\newblock {\em Math. Scand.}, 12:209--234, 1963.

\bibitem{LS}
Roger~C. Lyndon and Paul~E. Schupp.
\newblock {\em Combinatorial group theory}.
\newblock Springer-Verlag, Berlin-New York, 1977.
\newblock Ergebnisse der Mathematik und ihrer Grenzgebiete, Band 89.

\bibitem{Mor}
Shigeyuki Morita.
\newblock Cohomological structure of the mapping class group and beyond.
\newblock In {\em Problems on mapping class groups and related topics},
  volume~74 of {\em Proc. Sympos. Pure Math.}, pages 329--354. Amer. Math.
  Soc., Providence, RI, 2006.

\bibitem{Nak}
Minoru Nakaoka.
\newblock Decomposition theorem for homology groups of symmetric groups.
\newblock {\em Ann. of Math. (2)}, 71:16--42, 1960.

\bibitem{Nak2}
Minoru Nakaoka.
\newblock Homology of the infinite symmetric group.
\newblock {\em Ann. of Math. (2)}, 73:229--257, 1961.

\bibitem{Ni18}
J.~Nielsen.
\newblock \"{U}ber die {I}somorphismen unendlicher {G}ruppen ohne {R}elation.
\newblock {\em Math. Ann.}, 79(3):269--272, 1918.

\bibitem{Petresco}
Julian Petresco.
\newblock Sur les groupes libres.
\newblock {\em Bull. Sci. Math. (2)}, 80:6--32, 1956.

\bibitem{P-hodge}
Teimuraz Pirashvili.
\newblock Hodge decomposition for higher order {H}ochschild homology.
\newblock {\em Ann. Sci. \'Ecole Norm. Sup. (4)}, 33(2):151--179, 2000.

\bibitem{QK}
Daniel Quillen.
\newblock Higher algebraic {$K$}-theory. {I}.
\newblock In {\em Algebraic {$K$}-theory, {I}: {H}igher {$K$}-theories ({P}roc.
  {C}onf., {B}attelle {M}emorial {I}nst., {S}eattle, {W}ash., 1972)}, pages
  85--147. Lecture Notes in Math., Vol. 341. Springer, Berlin, 1973.

\bibitem{RW}
Oscar Randal-Williams.
\newblock The stable cohomology of automorphisms of free groups with
  coefficients in the homology representation.
\newblock arXiv : math.AT/1012.1433, 2010.

\bibitem{RW2}
Oscar Randal-Williams.
\newblock Cohomology of automorphism groups of free groups with twisted
  coefficients.
\newblock {\em Selecta Math. (N.S.)}, 24(2):1453--1478, 2018.

\bibitem{RWW}
Oscar Randal-Williams and Nathalie Wahl.
\newblock Homological stability for automorphism groups.
\newblock {\em Adv. Math.}, 318:534--626, 2017.

\bibitem{Sat1}
Takao Satoh.
\newblock Twisted first homology groups of the automorphism group of a free
  group.
\newblock {\em J. Pure Appl. Algebra}, 204(2):334--348, 2006.

\bibitem{Sat2}
Takao Satoh.
\newblock Twisted second homology groups of the automorphism group of a free
  group.
\newblock {\em J. Pure Appl. Algebra}, 211(2):547--565, 2007.

\bibitem{Sco}
Alexander Scorichenko.
\newblock {\em Stable {K}-theory and functor homology over a ring}.
\newblock PhD thesis, Evanston, 2000.

\bibitem{Seg}
Graeme Segal.
\newblock Categories and cohomology theories.
\newblock {\em Topology}, 13:293--312, 1974.

\bibitem{Snaith}
V.~P. Snaith.
\newblock A stable decomposition of {$\Omega ^{n}S^{n}X$}.
\newblock {\em J. London Math. Soc. (2)}, 7:577--583, 1974.

\bibitem{T79}
R.~W. Thomason.
\newblock Homotopy colimits in the category of small categories.
\newblock {\em Math. Proc. Cambridge Philos. Soc.}, 85(1):91--109, 1979.

\bibitem{ves-cal}
Christine Vespa.
\newblock Extensions between functors from free groups.
\newblock {\em Bull. Lond. Math. Soc.}, 50(3):401--419, 2018.

\bibitem{Weib}
Charles~A. Weibel.
\newblock {\em An introduction to homological algebra}, volume~38 of {\em
  Cambridge Studies in Advanced Mathematics}.
\newblock Cambridge University Press, Cambridge, 1994.

\end{thebibliography}
\end{document}